\documentclass[11pt]{article}

\usepackage[a4paper,margin=28mm]{geometry}
\usepackage{amsmath,amssymb,amsthm,mathtools}
\usepackage{enumitem}
\usepackage{microtype}
\usepackage[
  colorlinks=true,
  linkcolor=blue,
  citecolor=blue,
  urlcolor=blue
]{hyperref}
\usepackage{tikz}
\usepackage[colorinlistoftodos]{todonotes}
\usetikzlibrary{positioning,calc,fit,arrows.meta,shapes.multipart}
\newtheorem{theorem}{Theorem}[section]
\newtheorem{proposition}[theorem]{Proposition}
\newtheorem{lemma}[theorem]{Lemma}
\newtheorem{corollary}[theorem]{Corollary}
\theoremstyle{definition}
\newtheorem{definition}[theorem]{Definition} 
\theoremstyle{remark}

\theoremstyle{definition}
\newtheorem{assumption}{Assumption}

\newcommand{\T}{\mathbb T}
\newcommand{\R}{\mathbb R}
\newcommand{\Comm}{\operatorname{Comm}}
\newcommand{\err}{\mathfrak e_N}

\hypersetup{
  pdftitle={Uniform-in-Time and Uniform-in-Temperature Particle Approximations of Coulomb MMD Flows},
  pdfauthor={Junho Lee, Promit Ghosal}
}

\title{Uniform-in-Time and Uniform-in-Temperature Particle Approximations of Coulomb MMD Flows}
\author{
  Junho Lee\thanks{
    Committee on Computational and Applied Mathematics,
    University of Chicago.
    Email: \texttt{jhlee99@uchicago.edu}
  }
  \and
  Promit Ghosal\thanks{
    Department of Statistics,
    University of Chicago.
    Email: \texttt{promit@uchicago.edu}
  }
}
\date{}

\begin{document}
\maketitle

\begin{abstract}
We study particle approximations of maximum mean discrepancy (MMD)
flows with the Coulomb kernel on the torus $\T^d$, $d\ge2$, driven
by thermal noise of strength $\tau\ge0$.  For smooth, strictly
positive initial and target densities and i.i.d.\ initial positions,
we prove that the root mean square Wasserstein-2 distance between the empirical
measure and the mean-field solution is at most of order $N^{-1/d}$
for $d\ge3$ and $\sqrt{(\log N)/N}$ for $d=2$, uniformly in time
and in the temperature $\tau\ge0$, including the deterministic case
$\tau=0$. At each fixed $\tau>0$, we also prove propagation of chaos in total variation for any $k$-fold marginal distributions of particles, uniformly in time, with a bound of order $\sqrt{k/\tau}$ times the empirical convergence rate. Our proof uses the modulated free-energy method of Bresch, Jabin, and
Wang \cite{BreschJabinWang2023} to reduce the particle convergence
estimate to a relaxation estimate that controls the time integral of
the mean-field velocity gradient uniformly in $\tau\ge0$.
Near zero temperature, where the smoothing effect of the noise
degenerates, we obtain this estimate by recentering the dynamics at
thermal equilibrium and using the damping generated by the Coulomb
interaction.

\end{abstract}


\bigskip
\setcounter{tocdepth}{2}

\section{Introduction}
\label{sec:introduction}

\subsection{The model and main question}
\label{subsec:intro-model}
Let $d\ge2$, and let $\pi$ be a strictly positive probability
density on the unit flat torus $\T^d=\R^d/\mathbb Z^d$.
We study the behavior of the interacting particle system 
\begin{equation}
\begin{cases}
  dX_t^i
  =-
  \left[
    \frac1N\sum_{j\ne i}
    \nabla K_1\bigl(X_t^i-X_t^j\bigr)
    -\nabla K_1*\pi\bigl(X_t^i\bigr)
  \right]dt
  +\sqrt{2\tau}\,dB_i(t),\\
  X_0^1,\ldots,X_0^N\text{ are i.i.d.\ with law }\rho_0,
\end{cases}
  \label{eq:intro-particle-system}
\end{equation}
where $\tau\ge0$ is the temperature and
$\{B_i\}_{i=1}^N$ are independent Brownian motions.
The interaction kernel $K_1$ is the Coulomb kernel, the $s=1$ case of
the periodic Riesz kernels $K_s$ determined on $\T^d$ for $s\ge1$
by $(-\Delta)^sK_s=\delta_0-1$ and $\int_{\T^d}K_s=0$.
In the singularity-exponent convention of the periodic Riesz literature
\cite{ChodronRosenzweigSerfaty2025}, $K_1$ corresponds to the exponent
$d-2$ for the inverse of inter-particle distance in dimensions $d\ge3$ and to the logarithmic case in
dimension $d=2$.

The system \eqref{eq:intro-particle-system} describes particles with
pairwise Coulomb repulsion moving in the external potential
$-K_1*\pi$ \cite{ArmstrongSerfaty2022,Serfaty2026}.
It is also a particle approximation of a Coulomb discrepancy flow
within the maximum mean discrepancy (MMD) framework for distribution
matching \cite{ArbelKorbaSalimGretton2019,BoufadeneVialard2025},
with entropy regularization when $\tau>0$.

In particular, for the configuration $X_N(t):=(X_t^1,\ldots,X_t^N)$, consider the empirical measure
\[
  \mu_{X_N}^t:=\frac1N\sum_{i=1}^N\delta_{X_t^i}.
\]
A formal application of It\^o's lemma suggests
that as the number of particles $N$ increases to $\infty$, $\mu_{X_N}^t$ converges weakly to $\rho_t^\tau$, the solution of the mean-field PDE
\begin{equation}
  \begin{cases}
    \partial_t\rho_t^\tau
  =
  \tau\Delta\rho_t^\tau
  +\nabla\cdot\left(
    \rho_t^\tau
    \bigl(\nabla K_1*(\rho_t^\tau-\pi)\bigr)
  \right)\\
  \rho_{t=0}^\tau=\rho_0.
  \end{cases}
  \label{eq:intro-mean-field-pde}
\end{equation}
This equation is the Wasserstein-$2$ gradient flow
\cite{AmbrosioGigliSavare2008} of the free-energy functional (see \cite{BoufadeneVialard2025,ChodronRosenzweigSerfaty2025} for details)
\begin{equation}
  \mathcal F_\tau(\rho)
  :=
  \frac12\|\rho-\pi\|_{\dot H^{-1}}^2
  +
  \tau\int_{\T^d}\rho\log\rho\,dx.
  \label{eq:intro-free-energy}
\end{equation}
By Plancherel's identity and the definition of
$K_1=(-\Delta)^{-1}$ on mean-zero functions, we can rewrite the negative homogeneous Sobolev norm in \eqref{eq:intro-free-energy} as
\[
  \|\rho-\pi\|_{\dot H^{-1}}^2
  =
  \iint_{\T^d\times\T^d}
  K_1(x-y)\,
  d(\rho-\pi)^{\otimes2}(x,y).
\]
This is the Coulomb energy of the signed measure $\rho-\pi$
\cite{Serfaty2026}, studied under the name Coulomb discrepancy in
\cite{BoufadeneVialard2025,ChodronRosenzweig2026}.
The Coulomb discrepancy has the same quadratic form as squared reproducing kernel Hilbert space (RKHS) MMD \cite{GrettonBorgwardtRaschScholkopfSmola2006}, with $K_1$ in place of an RKHS kernel, which motivates the name Coulomb MMD \cite{BoufadeneVialard2025,ChodronRosenzweig2026}.
For $d\ge2$, however, $K_1$ is singular on the diagonal and the corresponding test space $\dot H^1(\T^d)$ has no bounded point evaluations, so the usual RKHS interpretation does not apply \cite[footnote~1]{ChizatColomboColomboFernandezReal2026}.
The discrepancy is finite for densities of finite $\dot H^{-1}$ energy, while particle comparisons require removal of the infinite diagonal self-interactions \cite{Serfaty2020,ChodronRosenzweigSerfaty2025}.

Stein variational gradient descent (SVGD) \cite{LiuWang2016} also drives
interacting particles toward $\pi$, but its drift involves the score
$\nabla\log\pi$ evaluated at the particles.
In \eqref{eq:intro-particle-system}, by contrast, the target enters only
through the convolution $\nabla K_1*\pi$.
For smooth kernels, such convolutions can be approximated from samples of
$\pi$, but the MMD functional need not be convex along Wasserstein geodesics, which
makes its long-time convergence analysis delicate
\cite{ArbelKorbaSalimGretton2019}.

We expect that, as $t\to\infty$, the solution $\rho_t^\tau$ of \eqref{eq:intro-mean-field-pde} converges
to a \emph{thermal equilibrium} $\rho_\infty^\tau$, a stationary solution of
\eqref{eq:intro-mean-field-pde}
(see \cite{ArmstrongSerfaty2022} for whole-space thermal equilibria), which approaches the target $\pi$ as
$\tau\downarrow0$, and that the empirical measure $\mu_{X_N}^t$
converges to $\rho_\infty^\tau$ as $t,N\to\infty$. We can decompose the error between the empirical measure and the target as
\[
  \mu_{X_N}^t-\pi
  =
  \underbrace{(\mu_{X_N}^t-\rho_t^\tau)}_{\text{finite-particle error}}
  +
  \underbrace{(\rho_t^\tau-\rho_\infty^\tau)}_{\text{relaxation}}
  +\underbrace{(\rho_\infty^\tau-\pi)}_{\text{thermal bias}}.
\]
The last two terms concern the continuum equation alone, and the first
compares the particle system with it.  All three depend on the interaction kernel.  A singular kernel drives
the relaxation of the continuum flow, and the same singularity
complicates the comparison with the particle system.

At zero temperature, for smooth kernels, the continuum MMD flow
need not converge to $\pi$ and may instead approach a
nonoptimal stationary state
\cite[Section~3]{ArbelKorbaSalimGretton2019}.  If the negative Sobolev distance between $\rho_t^0$ and $\pi$,
weighted by $\rho_t^0$, remains uniformly bounded in time,
\cite[Definition~1 and Proposition~7(i)]{ArbelKorbaSalimGretton2019}
gives $O(t^{-1})$ decay of the squared MMD.
For the periodic Riesz family $K_s$,
$\|\rho_t^0-\pi\|_{\dot H^{-1}}$ decays exponentially when $s=1$,
for bounded initial and target densities with positive lower bounds
\cite[Theorem~1.2]{ChizatColomboColomboFernandezReal2026}.
For $s>1$, the flow with kernel $K_s$ converges to $\pi$
at a polynomial rate in $\dot H^{-s}$, under suitable regularity
and smallness assumptions
\cite[Theorem~1.4]{ChizatColomboColomboFernandezReal2026}. We study the rate of convergence $\mu_{X_N}^t$ of the empirical measure of the particles when $s=1$.


We also consider the evolution
of its joint law $f_N^\tau(t)$, which is governed by the \emph{Liouville equation}
\begin{equation}
  \partial_t f_N^\tau
  =
  \sum_{i=1}^N\nabla_{x_i}\cdot\left(
    f_N^\tau\left[
      \frac1N\sum_{j\ne i}\nabla K_1(x_i-x_j)
      -\nabla K_1*\pi(x_i)
    \right]
  \right)
  +\tau\sum_{i=1}^N\Delta_{x_i}f_N^\tau.
  \label{eq:intro-liouville-equation}
\end{equation}
For the i.i.d. initialization in \eqref{eq:intro-particle-system},
$f_N^\tau(0)=\rho_0^{\otimes N}$. For each fixed $k$, we define the $k$-particle marginal as
$f_N^\tau(t)$,
\[
  f_{N;k}^\tau(t,x_1,\ldots,x_k)
  :=
  \int_{(\T^d)^{N-k}}
  f_N^\tau(t,x_1,\ldots,x_N)\,dx_{k+1}\cdots dx_N
\]
and ask whether
$f_{N;k}^\tau(t)\to(\rho_t^\tau)^{\otimes k}$ weakly as $N\to\infty$.
This convergence is the notion of \emph{propagation of chaos} (see the
review \cite{ChaintronDiez2022}).  For exchangeable systems,
propagation of chaos is qualitatively equivalent to convergence in
probability of the empirical measure to a deterministic limit, in the
weak topology of probability measures
\cite{HaurayMischler2014} (see also
\cite[Section~6.2.1]{Serfaty2026} and
\cite[Lemma~3.19]{ChaintronDiez2022}).
A stronger notion is propagation of chaos in total variation
\cite[Definition~3.22]{ChaintronDiez2022}.
For the densities considered here, this is equivalent to
$L^1$-convergence of each fixed marginal to the tensorized
mean-field density.



We ask whether there are bounds, depending on the particle
number $N$ but uniform in time, for
(i) the distance between the empirical measure and the solution of the
mean-field PDE, uniformly in temperature, and
(ii) the distance between the $k$-particle marginal and the $k$-fold
product of the solution of the mean-field PDE, at each fixed positive
temperature. Our main results in what follows address these two questions.

\subsection{Main results}
\label{subsec:intro-main-results}
In this subsection, we  state our two main results. The first one studies the empirical convergence to the mean-field solution uniformly in time and temperature, and the second one shows the propagation of chaos at each positive temperature. We then state how these estimates, together with relaxation of the mean-field solution, bound the distance between the empirical measure and the target. We begin with stating the assumptions for our results.
\begin{assumption}
\label{assump:standing-data}
Throughout the paper, $\pi$ and $\rho_0$ are probability densities
on $\T^d$ such that
\[
  \pi,\rho_0\in C^4(\T^d),
  \qquad
  0<c_\pi\le\pi\le C_\pi,
  \qquad
  0<c_0\le\rho_0\le C_0.
\]
\end{assumption}
We set $c_*:=\min\{c_0,c_\pi\}$ and $C_*:=\max\{C_0,C_\pi\}$.
Let $\err$ be the finite-particle energy correction defined precisely in
\eqref{eq:finite-particle-error}.  Its order is
\[
  \err\asymp
  \begin{cases}
    (\log N)/N, & d=2,\\
    N^{-2/d}, & d\ge3.
  \end{cases}
\]
It comes from sharp Coulomb energy lower bounds
\cite[Corollary~3.5, Eq.~(3.22)]{Serfaty2020},
\cite[Remark~6.5, Eq.~(6.10)]{ChodronRosenzweigSerfaty2025}.

Throughout, for $\tau>0$, $\rho^\tau$ denotes the unique global
classical solution given by Lemma~\ref{lem:positive-temperature-wp}.
At $\tau=0$, $\rho^0$ denotes the unique bounded global weak
solution in the sense of
\cite[Definition~1.1]{ChodronRosenzweig2026}, whose existence and
uniqueness follow from
\cite[Theorem~1.2 and Proposition~2.4]{ChodronRosenzweig2026}.

The particle results below use two solution concepts for
\eqref{eq:intro-liouville-equation}: for $\tau>0$, $f_N^\tau$ is any
global entropy solution (a weak solution satisfying an entropy inequality)
in the sense of
Definition~\ref{def:targeted-entropy-solution} with initial law
$\rho_0^{\otimes N}$. This includes the joint law of the
particle system \eqref{eq:intro-particle-system} by
Proposition~\ref{prop:targeted-entropy-construction}.
At $\tau=0$, where the equation reduces to a
transport equation, $f_N^0$ is the push-forward of
$\rho_0^{\otimes N}$ along the deterministic particle flow.

\begin{theorem}[Uniform empirical mean-field convergence]
\label{thm:uniform-mean-field}
Under Assumption~\ref{assump:standing-data}, let the initial positions in
\eqref{eq:intro-particle-system} be i.i.d.\ with law $\rho_0$, and let
$f_N^\tau$ and $\rho_t^\tau$ be as above.
There exists $C>0$, depending only on $d$, $\rho_0$,
and $\pi$, such that, for every $N\ge2$,
\[
  \sup_{\tau\ge0}
  \operatorname*{ess\,sup}_{t\ge0}
  \mathbb E_{f_N^\tau(t)}\left[
    \mathcal W_2\!\left(\mu_{X_N},\rho_t^\tau\right)^2
  \right]
  \le
  C\err.
\]
Here $\mathcal W_2$ is the $2$-Wasserstein distance on $\T^d$.
At $\tau=0$, the essential supremum may be replaced by
$\sup_{t\ge0}$. 
\end{theorem}

For $\zeta>d/2+1$, Sobolev embedding $H^{\zeta}(\T^d)\hookrightarrow W^{1,\infty}(\T^d)$ enables to apply the Kantorovich--Rubinstein duality, as in \eqref{eq:negative-sobolev-wasserstein-comparison}, together with $\mathcal W_1\le\mathcal W_2$, to further obtain
\[
  \sup_{\tau\ge0}
  \operatorname*{ess\,sup}_{t\ge0}
  \mathbb E_{f_N^\tau(t)}\left[
    \|\mu_{X_N}-\rho_t^\tau\|_{\dot H^{-\zeta}}^2
  \right]
  \le
  C_\zeta\err.
\]
where $C_\zeta$ depends on $\zeta$, $d$, $\rho_0$, and $\pi$.

The following theorem establishes quantitative propagation of chaos
in total variation, uniformly in time at each fixed positive temperature.
\begin{theorem}[Positive-temperature uniform-in-time propagation of chaos]
\label{thm:uniform-propagation-chaos}
Under the Assumption~\ref{assump:standing-data}, there exists $C>0$, depending only on $d$, $\rho_0$, and $\pi$,
such that, for every $N\ge2$, $\tau>0$, and $1\le k\le N$,
\begin{equation}
  \operatorname*{ess\,sup}_{t\ge0}
  \|f_{N;k}^\tau(t)-(\rho_t^\tau)^{\otimes k}\|_{L^1}
  \le
  C\left(\frac{k\err}{\tau}\right)^{1/2},
  \label{eq:fixed-temperature-marginal-chaos}
\end{equation}
and, in particular, for fixed $k$ and $\tau>0$, the convergence is uniform
in time as $N\to\infty$.
\end{theorem}
The estimate degenerates as $\tau\downarrow0$, whereas the empirical
mean-field convergence in Theorem~\ref{thm:uniform-mean-field} still
yields weak propagation of chaos at $\tau=0$
\cite[Section~6.2.1]{Serfaty2026}. As a consequence, Corollary~\ref{cor:intro-target-approximation}
gives the following bound on the error between the empirical measure
and the target:
\[
  \left(
    \mathbb E_{f_N^\tau(t)}
    \left[
      \mathcal W_2(\mu_{X_N},\pi)^2
    \right]
  \right)^{1/2}
  \le
  C\left(
    \err^{1/2}
    +e^{-\frac{c_*}{2}(\frac12+\kappa_*\tau)t}
    +\tau
  \right).
\]
Here $\kappa_*=(2\pi)^2/(2C_*)$.
This holds for almost every $t\ge0$ at each $\tau>0$, and for every
$t\ge0$ at $\tau=0$.
For nonconstant $\pi$, the thermal bias is of exact order
$\tau$ in $\mathcal W_2$ as $\tau\downarrow0$
(Corollary~\ref{cor:thermal-bias-lower}).

\subsection{Proof strategy}
\label{subsec:intro-proof-strategy}

We describe the ideas behind Theorems~\ref{thm:uniform-mean-field}
and~\ref{thm:uniform-propagation-chaos} in three steps.  The first step
compares the particles with the mean-field solution.  This comparison has
a Gr\"onwall factor involving the time integral of the Lipschitz norm of
the mean-field velocity.  The other two steps show that this integral is
bounded independently of the temperature, first near zero temperature
and then for temperatures bounded away from zero.  This is the point at
which a finite-time particle estimate becomes uniform in time.  Earlier
positive-temperature estimates use parabolic smoothing and lose their
constants as $\tau\downarrow0$
\cite[Remark~1.8]{ChodronRosenzweigSerfaty2025},
while \cite{Serfaty2020} gives a finite-time modulated-energy
estimate at zero temperature.

\medskip
\noindent\emph{Step 1:}
The first step is to reduce the problem of rate of convergence of particle dynamics to that of mean-field density of particle system. To this end, let us denote the modulated energy for a collision-free configuration $X_N=(x_1,\ldots,x_N)$ and a
bounded density $\rho$ as (see also Definition~\ref{def:modulated_energy})
\[
  F_N(X_N,\rho)
  :=\iint_{(\T^d)^2\setminus\Delta}
       K_1(x-y)\,d(\mu_{X_N}-\rho)^{\otimes2}(x,y),
  \qquad \Delta:=\{(x,x):x\in\T^d\}.
\]
It measures the Coulomb discrepancy between the empirical measure and
$\rho$, with the infinite self-interaction of each particle removed.
This removal means that $F_N$ need not be positive.  The sharp
Coulomb lower bound makes $\frac12F_N+\err$ nonnegative, where the
finite-particle correction $\err$ is defined in
\eqref{eq:finite-particle-error} and has order $(\log N)/N$ for
$d=2$ and $N^{-2/d}$ for $d\ge3$.

At $\tau=0$, differentiating $F_N$ along the particle trajectories
produces a nonnegative squared force discrepancy and the commutator term defined below.
At $\tau>0$, applying It\^o's formula to $F_N$ alone produces a
quadratic-variation term that the Coulomb estimate does not control
\cite[Section~1.2]{ChodronRosenzweigSerfaty2025}.  We therefore
compare the laws, as in \cite{BreschJabinWang2023}, by adding relative
entropy.  With (see also Definition~\ref{def:mod_energy})
\[
  \mathcal E_N^\tau(t)
  :=\frac12\mathbb E_{f_N^\tau(t)}
       F_N(X_N,\rho_t^\tau)
       +\tau H_N\bigl(f_N^\tau(t)\mid
                          (\rho_t^\tau)^{\otimes N}\bigr),
\]
where $H_N$ measuring relative entropy between $f_N^\tau(t)$ and $(\rho_t^\tau)^{\otimes N}$, the diffusion terms in the two parts of $\mathcal E_N^\tau$ combine
into a nonnegative dissipation.  The prescribed target adds a one-body
potential to the unnormalized particle Gibbs weight
\[
  G_N(X_N):=\exp\!\left[-\frac1{2N\tau}
      \sum_{i\ne j}K_1(x_i-x_j)
      +\frac1\tau\sum_{i=1}^N(K_1*\pi)(x_i)\right].
\]
Comparing this weight with its mean-field counterpart recovers
$F_N$ and gives the targeted modulated free-energy inequality in
Proposition~\ref{prop:targeted-mfe}. In both calculations, the remaining term is expressed through
the Coulomb commutator
\[
  \Comm_N[v](X_N,\rho)
  :=\iint_{(\T^d)^2\setminus\Delta}
       (v(x)-v(y))\cdot\nabla K_1(x-y)\,
       d(\mu_{X_N}-\rho)^{\otimes2}(x,y).
\]
It captures the change in modulated energy when the configuration and its
reference density are transported by $v$.  The estimate in
Lemma~\ref{prop:static-coulomb-bounds} bounds its absolute value by
$C_d\|\nabla v\|_{L^\infty}(\frac12F_N+\err)$.  We apply it with
the velocity,
$u_t^\tau:=\tau\nabla\log\rho_t^\tau
               +\nabla K_1*(\rho_t^\tau-\pi)$.
Writing
$\Psi_\tau(t):=\int_0^t\|\nabla u_s^\tau\|_{L^\infty}\,ds$,
Gr\"onwall's inequality gives
\[
  \mathcal E_N^\tau(t)+\err
  \le e^{C_d\Psi_\tau(t)}
       \bigl(\mathcal E_N^\tau(0)+\err\bigr),
  \qquad \tau>0,
\]
for almost every $t$. At $\tau=0$ the same bound in a path wise way  holds with
$\mathcal E_N^\tau$ replaced by $\frac12F_N$.
Thus the particle
question reduces to proving
\[
  \Psi_\infty
  :=\sup_{\tau\ge0}\int_0^\infty
        \|\nabla u_t^\tau\|_{L^\infty}\,dt<\infty.
\]
Once this is known (discussed in Step 2 \& Step 3), the Wasserstein bound of
Lemma~\ref{lem:distance-coercivity} (which upper bounds $\mathcal W_2(\mu_{X_N},\rho)^2$ by $C(\frac12F_N+\err)$) turns the shifted modulated-energy
bound into Theorem~\ref{thm:uniform-mean-field}.  For $\tau>0$,
since we have $\tau H_N\le\mathcal E_N^\tau+\err$,  one obtains the fixed-marginal
$L^1$ estimate in Theorem~\ref{thm:uniform-propagation-chaos} utilizing entropy subadditivity and
the Csisz\'ar--Kullback--Pinsker inequality.

\medskip
\noindent\emph{Step 2:} To use Step~1, we need an integrable bound on the velocity gradient.
Here we prove the stronger estimate
$\|\nabla u_t^\tau\|_{L^\infty}\le Ce^{-ct}$, with constants
independent of $0\le\tau\le\tau_0$ for some small $\tau_0>0$.
At positive temperature the stationary density generally differs from
the target $\pi$.  We therefore use the thermal equilibrium
$\rho_\infty^\tau$, the stationary density from
Proposition~\ref{prop:thermal-equilibrium-family} with recalling that 
$\rho_\infty^0:=\pi$.  Writing
$q_\tau=\log\rho_\infty^\tau$, the stationary equation becomes
$-\tau\Delta q_\tau+e^{q_\tau}=\pi$.
The positive lower bound of $\pi$ provides uniform bounds on
$q_\tau$ and its first four derivatives as $\tau\downarrow0$
(Proposition~\ref{prop:stationary-all-orders}).  It also gives
$\rho_\infty^\tau-\pi=\tau\Delta q_\tau=O(\tau)$ in
$C^2$ (Proposition~\ref{prop:stat-bias}).

Let $g_t^\tau:=\rho_t^\tau-\rho_\infty^\tau$ be the density error.
The equilibrium Euler--Lagrange equation gives
$u_\infty^\tau
 :=\tau\nabla\log\rho_\infty^\tau
   +\nabla K_1*(\rho_\infty^\tau-\pi)=0$.
Subtracting this identity from the velocity gives
\begin{equation}
  \nabla u_t^\tau
  =\nabla^2K_1*g_t^\tau
   +\tau\bigl(\nabla^2\log\rho_t^\tau
                -\nabla^2\log\rho_\infty^\tau\bigr).
  \label{eq:intro-recentered-velocity-gradient}
\end{equation}
Thus we need decay of the Coulomb Hessian in the first term and of the log-Hessian difference in the second.  The operator
$\nabla^2(-\Delta)^{-1}=\nabla^2K_1$ does not map $L^\infty$ into itself, requiring stronger control of $g_t^\tau$.
The periodic Schauder estimate
(Lemma~\ref{lem:periodic-schauder}) controls the first by a
H\"older norm of $g_t^\tau$ whereas the identity
$\nabla^2\log\rho=\rho^{-1}\nabla^2\rho
-\rho^{-2}\nabla\rho\otimes\nabla\rho$ reduces the second to
$C^2$-decay of $g_t^\tau$ under the condition that the densities stay positive. While parabolic derivative bounds with energy decay estimates give $C^2$-decay of $g_t^\tau$ for each fixed $\tau>0$, the constants degenerate as $\tau\downarrow0$ \cite[Remark~1.8]{ChodronRosenzweigSerfaty2025}. We therefore seek a $\tau$-uniform estimate of $g_t^\tau$ in the following steps.

To prove the bound on $\nabla u_t^\tau$, we want to begin with
$L^\infty$-decay, then H\"older and $C^2$-decay of $g_t^\tau$.
However, $L^\infty$ estimate contains $\nabla K_1*g_t^\tau$ term, which cannot be absorbed by Young's convolution inequality.
We therefore
seek a weaker estimate first, which is provided by the free-energy gap between $\rho_t^\tau$ and $\rho_\infty^\tau$, identified in 
Lemma~\ref{lem:exact-free-energy-gap}. This gives control of the $\dot H^{-1}$-norm
\[
  \mathcal F_\tau(\rho_t^\tau)-\mathcal F_\tau(\rho_\infty^\tau)
  =\tfrac12\|g_t^\tau\|_{\dot H^{-1}}^2
    +\tau\operatorname{Ent}(\rho_t^\tau\mid\rho_\infty^\tau),
  \qquad
  \operatorname{Ent}(\rho\mid\nu)
  :=\int_{\T^d}\rho\log(\rho/\nu)\,dx.
\]
The time derivative of this gap is
$-\int\rho_t^\tau|u_t^\tau|^2$.
The density lower bound of Proposition~\ref{prop:maximum-principle}
and $\dot H^{-1}$--$\dot H^1$ duality compare this dissipation
with the gap.  Proposition~\ref{prop:hminus-exp} then yields
$\|g_t^\tau\|_{\dot H^{-1}}\le Ce^{-c(1+\tau)t}$.
This weak decay is the input to the near--far estimate of
Lemma~\ref{lem:potential-estimate}, which starts the $L^\infty$
estimate below without temperature-dependent smoothing.

Now we seek $L^\infty$-decay of $g^{\tau}_t$.  Subtracting the stationary equation
from the mean-field PDE gives \eqref{eq:equilibrium-error-equation}.
In this error equation, diffusion and transport act on
$g_t^\tau$, while the other terms are
$-\rho_t^\tau g_t^\tau$,
$\nabla\rho_\infty^\tau\cdot\nabla K_1*g_t^\tau$, and
$-(\rho_\infty^\tau-\pi)g_t^\tau$.
The first term damps the error because
$\rho_t^\tau\ge c_*>0$.  The last term costs only $O(\tau)$ of this
damping when $\tau\le\tau_0$.  At a spatial maximum, the remaining
term is the Coulomb field $\nabla K_1*g_t^\tau$, as made precise by
Hamilton's maximum lemma (Lemma~\ref{lem:hamilton}) and
Lemma~\ref{lem:linfty-differential-inequality}.  A direct convolution
estimate cannot make its coefficient small enough.  Adapting the near--far splitting in
\cite[proof of Lemma~A.4(i)]{ChizatColomboColomboFernandezReal2026},
Lemma~\ref{lem:potential-estimate} instead gives, for mean-zero $g$,
\[
  \|\nabla K_1*g\|_{L^\infty}
  \le C\|g\|_{L^\infty}^{d/(d+2)}
          \|g\|_{\dot H^{-1}}^{2/(d+2)}
       +C\|g\|_{\dot H^{-1}}.
\]
Young's inequality absorbs part of the first term into the damping whereas
the $\dot H^{-1}$-decay controls what is left.  This proves
exponential $L^\infty$-decay
(Proposition~\ref{prop:linfty-exp}).

The $L^\infty$ estimate of $g^{\tau}_t$ does not yet control the Coulomb Hessian.
After differentiating the error equation once, it appears as
$\nabla^2K_1*g_t^\tau$.  Lemma~\ref{lem:yudovich} bounds this
term by $C\|g_t^\tau\|_{L^\infty}
[1+\log_+(\|\nabla g_t^\tau\|_{L^\infty}/
\|g_t^\tau\|_{L^\infty})]$
\cite[Lemma~A.5(i)]{ChizatColomboColomboFernandezReal2026}.  With
$M_j^\tau(t):=\max_{|\beta|=j}
\|\partial^\beta g_t^\tau\|_{L^\infty}$, the differentiated
maximum-principle calculation shows the following inequality
\eqref{eq:first-derivative-osgood}
\[
  D^+M_1^\tau(t)
  \le C(1+t)e^{-ct}(1+M_1^\tau(t))
                 \log(e+M_1^\tau(t)).
\]
where $D^+$ denotes the upper right derivative. The coefficient in the right hand side is integrable in time, so $M_1^\tau$ remains
uniformly bounded
(Proposition~\ref{prop:uniform-lipschitz}).  H\"older interpolation inequality with the
$L^\infty$-decay now gives $C^{0,1/2}$-decay, and the periodic
Schauder estimate provides decay of the Coulomb Hessian
(Corollary~\ref{cor:holder-decay}).

Now we need to control the second term (thermal term) of the right side of 
\eqref{eq:intro-recentered-velocity-gradient}.  To obtain
$C^2$-decay, we differentiate the error equation twice.  The
singular term $\partial_{ij}(\nabla K_1*g_t^\tau)$ seems to require
one extra derivative.  Moving one derivative to $g_t^\tau$ and
using the Schauder estimate with Lemma~\ref{lem:appC-ehrling}
bounds it by
$\varepsilon M_2^\tau+C_\varepsilon\|g_t^\tau\|_{L^\infty}$.
The Landau--Kolmogorov inequality of
Lemma~\ref{lem:appC-landau-kolmogorov} treats the products of first
derivatives.  Choosing $\varepsilon$ and then $\tau_0$ small
preserves the damping in the maximum-principle inequality
\eqref{eq:C2-differential-inequality}.  This yields exponential
$C^2$-decay and, by the logarithmic-Hessian identity above, the
desired velocity-gradient bound for $0<\tau\le\tau_0$
(Proposition~\ref{prop:low-temperature-relaxation}).

Finally, at $\tau=0$ the available solution is initially only a
bounded weak solution, so the differentiated equations cannot be
applied to it directly.  The uniform $W^{2,\infty}$ estimates give
compactness in $C_{\mathrm{loc}}([0,\infty);C^1(\T^d))$, using
time equicontinuity from the PDE.  Uniqueness of bounded weak solutions
\cite[Proposition~2.4]{ChodronRosenzweig2026} identifies the limit
as $\rho^0$.  Weak-* lower
semicontinuity transfers the second-derivative decay in
$W^{2,\infty}$ and the Coulomb-Hessian bound passes to the limit as
well.  Proposition~\ref{prop:tau0-endpoint} thus supplies the same
velocity-gradient estimate at $\tau=0$.

\medskip
\noindent\emph{Step 3:} In the third step, the goal is to demonstrate that the bound in Step~2 extends also in the temperature bounded away from $0$.
Fix the threshold $\tau_0$ from Step~2.  For
$\tau\ge\tau_0$, the heat-semigroup formulation gives a local
classical solution
\cite[proofs of Proposition~2.1 and Lemma~2.7]
{ChodronRosenzweigSerfaty2025}. We obtain the bounds on its first four spatial derivatives
and a global $W^{4,\infty}$ bound uniform on this temperature
half-line from Lemma~\ref{lem:positive-temperature-wp}
(proved in Subsection~\ref{subsec:appD-global-regularity}),
with constants depending on $\tau_0$ and the data in
Assumption~\ref{assump:standing-data}.

The free-energy argument in Step~2 already shows that the
$\dot H^{-1}$ error decays at a rate proportional to
$1+\tau$.  To turn this weak decay into the strong estimate needed
in \eqref{eq:intro-recentered-velocity-gradient}, we interpolate it
with the uniform $W^{4,\infty}$ bound.  In particular, the
Littlewood--Paley estimate
\[
  \|g_t^\tau\|_{C^{2,1/2}}
  \le C_d\|g_t^\tau\|_{\dot H^{-1}}^{3/(d+10)}
       \|g_t^\tau\|_{W^{4,\infty}}^{1-3/(d+10)}
\]
gives exponential $C^{2,1/2}$-decay, again at a rate proportional
to $1+\tau$.  Schauder estimates control the Coulomb term of
\eqref{eq:intro-recentered-velocity-gradient} whereas the displayed
identity for $\nabla^2\log\rho$ controls its thermal term.
Consequently, Proposition~\ref{prop:positive-temperature-halfline}
gives
\[
  \|\nabla u_t^\tau\|_{L^\infty}
  \le C_{\tau_0}(1+\tau)e^{-c(1+\tau)t},
  \qquad \tau\ge\tau_0.
\]
Together with Step~2, this completes the proof of Theorem~\ref{thm:uniform-envelope}, i.e., it shows 
for constants $C,\gamma>0$ independent of temperature,
\[
  \|\nabla u_t^\tau\|_{L^\infty}
  \le C(1+\tau)e^{-\gamma(1+\tau)t},
  \qquad t\ge0,\quad \tau\ge0.
\]
Its integral is bounded by $C/\gamma$, so
$\Psi_\infty<\infty$.  Inserting this bound into the particle
estimate of Step~1 completes the uniform-in-time argument.
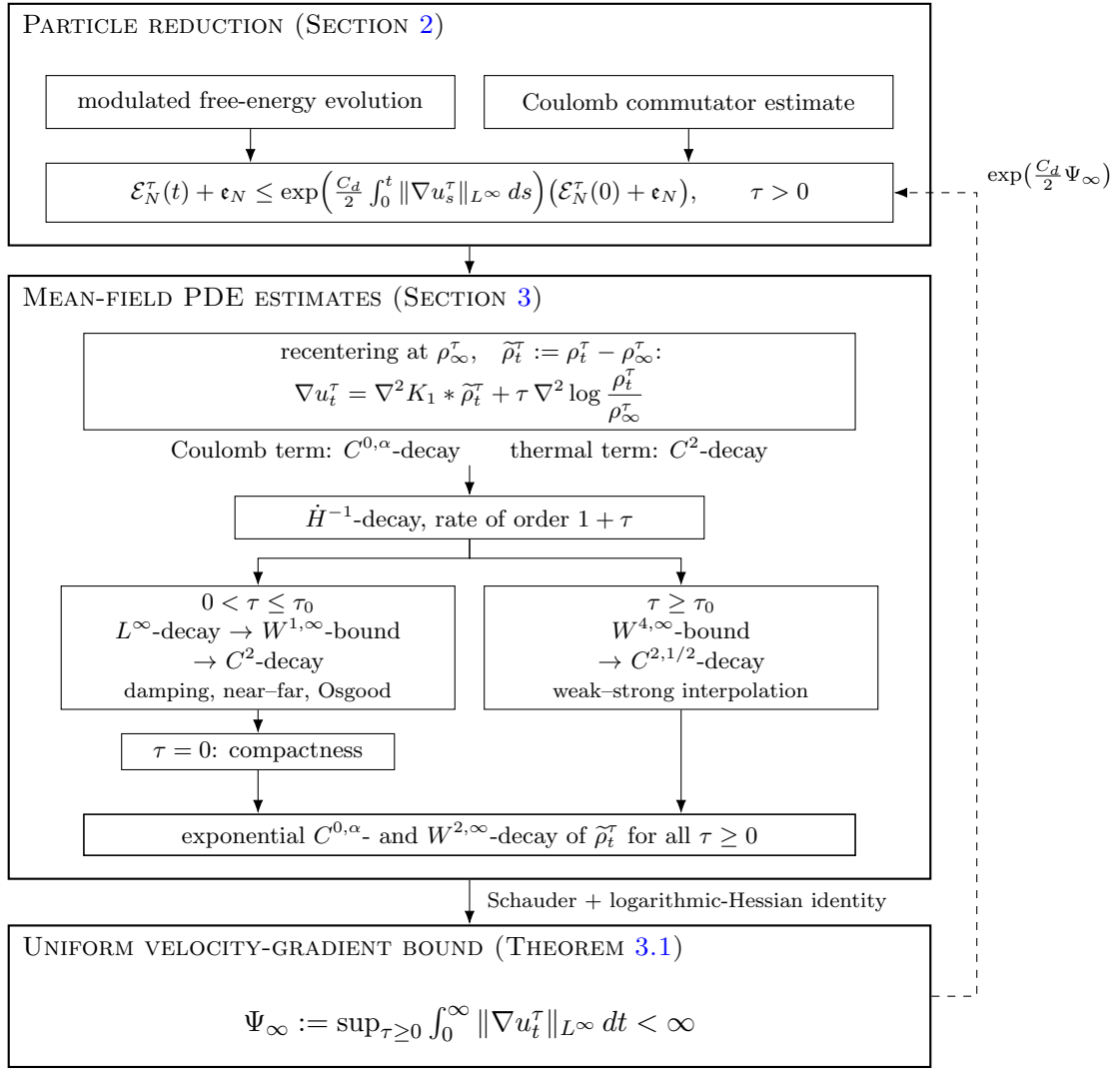
\begin{figure}[t]
\centering
\begin{tikzpicture}[
  x=1mm, y=1mm,
  sm/.style={draw, thin, rectangle, text badly centered, inner sep=3pt,
             font=\footnotesize},
  res/.style={draw, line width=0.6pt, rectangle, text badly centered,
              inner sep=3pt, font=\footnotesize},
  txt/.style={text badly centered, inner sep=1pt, font=\footnotesize},
  cont/.style={draw, line width=0.8pt, inner sep=3mm},
  ttl/.style={anchor=north west, inner sep=1pt, font=\small\scshape},
  arr/.style={-{Latex[length=1.8mm]}, thin},
  back/.style={-{Latex[length=1.8mm]}, thin, dashed}
]
\coordinate (T1) at (-58,120);
\coordinate (W1a) at (-58,105); \coordinate (W1b) at (58,105);
\node[sm, text width=52mm, minimum height=7mm] (P1) at (-29,110)
  {modulated free-energy evolution};
\node[sm, text width=52mm, minimum height=7mm] (P2) at (29,110)
  {Coulomb commutator estimate};
\node[sm, text width=110mm] (P3) at (0,98)
  {$\mathcal E_N^\tau(t)+\err
    \le\exp\Bigl(\tfrac{C_d}{2}\int_0^t\|\nabla u_s^\tau\|_{L^\infty}\,ds\Bigr)
    \bigl(\mathcal E_N^\tau(0)+\err\bigr),\qquad\tau>0$};
\node[cont, fit=(T1)(W1a)(W1b)(P1)(P2)(P3)] (C1) {};
\node[ttl] at ($(C1.north west)+(1.5,-1)$)
  {Particle reduction (Section~\ref{sec:modulated-identity})};
\draw[arr] (P1.south) -- (P1.south |- P3.north);
\draw[arr] (P2.south) -- (P2.south |- P3.north);
\coordinate (T2) at (-58,84);
\coordinate (W2a) at (-58,50); \coordinate (W2b) at (58,50);
\node[sm, text width=100mm] (R) at (0,73) {%
  recentering at $\rho_\infty^\tau$,\quad
  $\widetilde\rho_t^\tau:=\rho_t^\tau-\rho_\infty^\tau$:\\[1pt]
  $\nabla u_t^\tau
   =\nabla^2K_1*\widetilde\rho_t^\tau
   +\tau\,\nabla^2\log\dfrac{\rho_t^\tau}{\rho_\infty^\tau}$};
\node[txt] (N) at (0,64)
  {Coulomb term: $C^{0,\alpha}$-decay\qquad
   thermal term: $C^2$-decay};
\node[sm, text width=60mm] (H) at (0,55)
  {$\dot H^{-1}$-decay, rate of order $1+\tau$};
\node[sm, text width=50mm, anchor=north] (L) at (-28,46) {%
  $0<\tau\le\tau_0$\\
  $L^\infty$-decay $\to$ $W^{1,\infty}$-bound\\
  $\to$ $C^2$-decay\\
  {\scriptsize damping, near--far, Osgood}};
\node[sm, text width=50mm, anchor=north] (G) at (28,46) {%
  $\tau\ge\tau_0$\\
  $W^{4,\infty}$-bound\\
  $\to$ $C^{2,1/2}$-decay\\
  {\scriptsize weak--strong interpolation}};
\node[sm, text width=34mm] (E) at (-28,24)
  {$\tau=0$: compactness};
\node[res, text width=100mm] (O) at (0,13)
  {exponential $C^{0,\alpha}$- and $W^{2,\infty}$-decay of
   $\widetilde\rho_t^\tau$ for all $\tau\ge0$};
\node[cont, fit=(T2)(W2a)(W2b)(R)(N)(H)(L)(G)(E)(O)] (C2) {};
\node[ttl] at ($(C2.north west)+(1.5,-1)$)
  {Mean-field PDE estimates (Section~\ref{sec:stationary-family})};
\draw[arr] (N.south) -- (H.north);
\draw[arr] (H.south) -- ++(0,-2.5) -| (L.north);
\draw[arr] (H.south) -- ++(0,-2.5) -| (G.north);
\draw[arr] (L.south) -- (E.north);
\draw[arr] (E.south) -- (E.south |- O.north);
\draw[arr] (G.south) -- (G.south |- O.north);
\coordinate (T3) at (-58,-2);
\coordinate (W3a) at (-58,-12); \coordinate (W3b) at (58,-12);
\node[txt, font=\normalsize] (F) at (0,-12)
  {$\Psi_\infty:=\sup_{\tau\ge0}\int_0^\infty
    \|\nabla u_t^\tau\|_{L^\infty}\,dt<\infty$};
\node[cont, fit=(T3)(W3a)(W3b)(F)] (C3) {};
\node[ttl] at ($(C3.north west)+(1.5,-1)$)
  {Uniform velocity-gradient bound
   (Theorem~\ref{thm:uniform-envelope})};
\draw[arr] (C1.south) -- (C2.north);
\draw[arr] (C2.south) -- (C3.north)
  node[midway, right=1mm, font=\scriptsize]
  {Schauder $+$ logarithmic-Hessian identity};
\draw[back] (C3.east) -- ++(6,0) |- (P3.east);
\coordinate (K) at ($(C3.east)+(7,0)$);
\node[anchor=south west, inner sep=1pt, font=\scriptsize]
  at (K |- P3.east) {$\exp\bigl(\tfrac{C_d}{2}\Psi_\infty\bigr)$};
\end{tikzpicture}
\caption{Structure of the proof.  Section~\ref{sec:modulated-identity}
reduces Theorem~\ref{thm:uniform-mean-field} to a bound on the
integrated velocity gradient $\Psi_\infty$.
Section~\ref{sec:stationary-family} proves that bound by recentering at
the thermal equilibrium and treating temperatures near zero and
bounded away from zero separately, with the deterministic endpoint
recovered by compactness.}
\label{fig:proof-map}
\end{figure}

\subsection{Relation to previous work}
\label{subsec:intro-related-work}

In a recent work, Chodron de Courcel, Rosenzweig, and Serfaty study uniform in time convergence
uniform for untargeted periodic Coulomb and more singular
Riesz flows at fixed positive temperature
\cite{ChodronRosenzweigSerfaty2025}. Our work is heavily influenced by this result. In fact, we use their lower bound,
smearing and commutator estimates, and adapt their singular
regularization for entropy solutions
\cite[Remark~6.5, Eq.~(6.10), Proposition~2.13,
Eq.~(2.15), Lemma~6.2, Proposition~6.4, and
Appendix~A]{ChodronRosenzweigSerfaty2025}. However, their constants
deteriorate as $\tau\downarrow0$, whereas our
Theorem~\ref{thm:uniform-mean-field} gives one empirical bound for
possibly nonconstant targets and all $\tau\ge0$.

Modulated energy for singular flows originates in the works 
\cite{Duerinckx2016,Serfaty2020} and extends to the full Riesz range
in \cite{NguyenRosenzweigSerfaty2022}. With diffusion,
relative entropy \cite{JabinWang2018} and modulated free energy based methods
\cite{BreschJabinWang2019,BreschJabinWang2020,BreschJabinWang2023}
shows a novel way to compare joint laws which we adapt in our setting. On the other hand, periodic commutator estimate
\cite[Proposition~2.13]{ChodronRosenzweigSerfaty2025},
adapted from the Euclidean estimate
\cite[Theorem~1.1]{RosenzweigSerfaty2026},
controls our transport term. Section~\ref{sec:modulated-identity}
adapts this comparison to the target, and
Appendix~\ref{app:targeted-entropy} proves the corresponding
modulated free energy inequality for entropy solutions.

For deterministic targeted Coulomb flow,
\cite{BoufadeneVialard2025} proves exponential convergence in
Coulomb energy and $W_2$ for positive H\"older initial and target
densities.
The closest result on the deterministic continuum dynamics is
\cite{ChizatColomboColomboFernandezReal2026},
which proves exponential discrepancy decay and, under
dimension-dependent Sobolev regularity assumptions, time integrability
of the $L^\infty$-norm of the velocity gradient.
Our particle comparison requires time integrability
of the continuum velocity gradient uniformly for $\tau\ge0$, without imposing data regularity assumptions that depend on the dimension.  
\cite{ChodronRosenzweig2026} develops weak
solution theory, allowing probability measures as initial data. 
We establish $\tau$-uniform weak density decay based on \cite{BoufadeneVialard2025} and upgrade it to derivative decay. For $\tau>0$, thermal equilibrium in the potential $-K_1*\pi$
generally differs from $\pi$. Armstrong and Serfaty study such
equilibria in the whole space and their expansion at low temperature
\cite{ArmstrongSerfaty2022}. Section~\ref{sec:stationary-family}
constructs the periodic family, proves derivative bounds for
relaxation, and gives its first order expansion
\eqref{eq:stat-bias-expansion} in a similar spirit as in
\cite[Eqs.~(1.28) to (1.30), Proposition~6.1]{ArmstrongSerfaty2022}.
Related models coupling drift, diffusion, and Poisson's equation
appear in \cite{CaffarelliDolbeaultMarkowichSchmeiser2000,
GasserLevermoreMarkowichSchmeiser2001}.

For regular interactions, uniform-in-time propagation of chaos goes back
to \cite{Malrieu2001,Malrieu2003}. Sharp fixed-marginal bounds under uniform
functional inequalities were obtained in \cite{LackerLeFlem2023}.  For
singular interactions, fixed-positive-diffusion results include
\cite{GuillinLeBrisMonmarche2025,RosenzweigSerfaty2023}.  Chodron de
Courcel, Rosenzweig, and Serfaty proved sharp uniform-in-time convergence
for periodic Coulomb and super-Coulomb Riesz flows at each fixed positive
temperature, with constants that degenerate at zero temperature
\cite{ChodronRosenzweigSerfaty2025}.  Theorem
\ref{thm:uniform-mean-field} supplies one constant for all
$\tau\in[0,\infty)$, including the deterministic endpoint, for a
possibly nonconstant prescribed target.

For smooth kernels, \cite{ArbelKorbaSalimGretton2019} studies MMD
flows and their particle approximations, including a scheme that
evaluates the drift at randomly perturbed positions.
Finite-particle MMD bounds for such noise-injected schemes are
established under a gradient-dominance condition in
\cite{ChenKarvonenKanagawaBriolOates2026}.
Other approaches obtain convergence guarantees by modifying the
discrepancy or the transport geometry
\cite{ChenMustafiGlaserKorbaGrettonSriperumbudur2025,
TianSriperumbudurGrettonChen2026,GladinDvurechenskyMielkeZhu2024}.
For SVGD,
\cite{BanerjeeBalasubramanianGhosal2025} uses joint law entropy
relative to $\pi^{\otimes N}$ to bound time averages through
kernel Stein discrepancy. For smooth SVGD,
\cite{BalasubramanianBanerjeeKorba2026} gives bounds uniform over
averaging horizons and physical time bounds for special finite rank
systems in feature metrics. The entropy approach treats periodic
Riesz SVGD without self interaction \cite{TeolisDeHoop2026}.
Those SVGD flows use the score of $\pi$, while our target enters
through $\nabla K_1*\pi$. Our empirical bound is uniform in time
and temperature unlike the previous works.

\subsection*{Organization}

Section~\ref{sec:modulated-identity} states the conditional particle
estimate first and then proves it from the static Coulomb bounds, the
deterministic identity, and the positive-temperature modulated free-energy
inequality.  Section
\ref{sec:mean-field-analysis} proves the uniform velocity-gradient theorem
through the equilibrium, low-temperature, and positive-temperature
arguments.  Section~\ref{sec:uniform-consequences} proves the two main
theorems and records the approximation of the prescribed target.  The
appendices contain the deterministic collision argument, the construction of
entropy solutions for the singular Coulomb kernel, differentiated
maximum-principle calculations, positive-temperature regularity, and the
distance-coercivity estimate.

\subsection*{Acknowledgment}
We have used AI tools in organizing the manuscript and correcting grammatical errors. PG's research was partly supported by the NSF CAREER grant DMS 2542899.

\subsection{Notation}
\label{subsec:notation}
We close the introduction by collecting the notation used throughout the
paper.

\noindent\emph{Torus and Fourier conventions.}
Lebesgue measure on $\T^d$ is normalized so that $|\T^d|=1$, and
\[
  \widehat\nu(m)
  :=
  \int_{\T^d}e^{-2\pi i m\cdot x}\,d\nu(x),
  \qquad m\in\mathbb Z^d.
\]
For a mean-zero distribution $g$ and $s\in\R$,
\[
  \|g\|_{\dot H^s}^2
  :=
  \sum_{m\in\mathbb Z^d\setminus\{0\}}
  (4\pi^2|m|^2)^s|\widehat g(m)|^2.
\]

\medskip
\noindent\emph{Function spaces.}
$L^p(\T^d)$, $W^{k,p}(\T^d)$, $C^k(\T^d)$, and
$C^{k,\alpha}(\T^d)$ are the usual Lebesgue, Sobolev, and H\"older
spaces on the torus.  For $r>0$, $\log_+r:=\max\{\log r,0\}$.

\medskip
\noindent\emph{Measures.}
$\mathcal P(X)$ is the set of Borel probability measures on $X$,
for $X=\T^d$ or $X=(\T^d)^N$.  An absolutely continuous
probability measure is identified with its density, and
$\nu^{\otimes k}$ denotes the $k$-fold product of a measure
$\nu$.

\medskip
\noindent\emph{Particle configurations.}
For $N\ge2$, the collision-free configuration space is
\[
  \mathcal X_N
  :=
  \left\{
    X_N=(x_1,\ldots,x_N)\in(\T^d)^N:
    x_i\ne x_j\ \text{whenever }i\ne j
  \right\},
\]
and the empirical measure of $X_N\in(\T^d)^N$ is
$\mu_{X_N}:=\frac1N\sum_{i=1}^N\delta_{x_i}$.  For a time-dependent
configuration $X_N(t)$, we abbreviate $\mu_{X_N}^t:=\mu_{X_N(t)}$.

\medskip
\noindent\emph{Constants.}
Unless a subscript indicates otherwise, $C,c>0$ denote constants
whose values may change from line to line and which depend only on
$d$ and on the data $\rho_0$, $\pi$ of
Assumption~\ref{assump:standing-data}, never on $N$, $t$, or
$\tau$.  Any additional dependence is displayed in the notation.
We write $a\lesssim b$ if $a\le Cb$ for such a constant, and
$a\asymp b$ if $a\lesssim b$ and $b\lesssim a$.

\section{Reduction to the integrated velocity gradient}
\label{sec:modulated-identity}
\label{sec:particle-convergence}

This section reduces the particle estimates in
Theorems~\ref{thm:uniform-mean-field}
and~\ref{thm:uniform-propagation-chaos} to a bound on
$\int_0^t\|\nabla u_s^\tau\|_{L^\infty}\,ds$, where
$u^\tau$ is the mean-field velocity defined in
\eqref{eq:free-energy-velocity}.  We first define the quantities that
compare the particle law with the mean-field solution.  The static
Coulomb bounds below control the error caused by removing particle
self-interactions, and Proposition~\ref{prop:assembly} states the
resulting conditional particle estimate.  Its proof follows the
deterministic and positive-temperature calculations in the subsection
below.

For a collision-free configuration $X_N\in\mathcal X_N$, the
empirical measure $\mu_{X_N}$ was defined in
Subsection~\ref{subsec:notation}.  Write
$\Delta:=\{(x,x):x\in\T^d\}$.  The Coulomb energy of
$\mu_{X_N}-\rho$ would include the infinite interaction of each
particle with itself.  We remove those terms in the following
definition.

\begin{definition}[Modulated energy]\label{def:modulated_energy}
For $N\ge2$, $X_N\in\mathcal X_N$, and a bounded probability
density $\rho$ on $\T^d$, set
\[
  F_N(X_N,\rho)
  :=\iint_{(\T^d)^2\setminus\Delta}
  K_1(x-y)\,d(\mu_{X_N}-\rho)^{\otimes2}(x,y).
\]
\end{definition}

At $\tau=0$, differentiating $F_N$ along a collision-free
trajectory yields a particle identity.  At $\tau>0$, applying
It\^o's formula directly at the Coulomb singularity produces a term
that this energy does not control
\cite[Section~1.2]{ChodronRosenzweigSerfaty2025}.  We instead work
with the joint density $f_N$ and add relative entropy, following
\cite[Proposition~2.2]{BreschJabinWang2023}.

\begin{definition}[Modulated free energy]\label{def:mod_energy}
For a symmetric probability density $f_N$ on $(\T^d)^N$ and a
positive bounded probability density $\rho$ on $\T^d$, let
$dX_N:=dx_1\cdots dx_N$ and set
\[
  H_N(f_N\mid\rho^{\otimes N})
  :=\frac1N\int_{(\T^d)^N}
  f_N\log\frac{f_N}{\rho^{\otimes N}}\,dX_N.
\]
The modulated free energy is
\[
  \mathcal E_N^\tau(f_N,\rho)
  :=\frac12\int_{(\T^d)^N}F_N(X_N,\rho)f_N\,dX_N
    +\tau H_N(f_N\mid\rho^{\otimes N}).
\]
As usual, $0\log0:=0$, and the quantities may take the value
$+\infty$.  Along the positive-temperature evolution, we write
$\mathcal E_N^\tau(t)
 :=\mathcal E_N^\tau(f_N^\tau(t),\rho_t^\tau)$.
\end{definition}

The entropy term supplies the cancellation needed when diffusion acts
on the expected modulated energy.  The remaining transport term is
expressed using the following quantity.  For a Lipschitz vector
field $v\in W^{1,\infty}(\T^d;\R^d)$, recall
\[
  \Comm_N[v](X_N,\rho)
  =\iint_{(\T^d)^2\setminus\Delta}
  (v(x)-v(y))\cdot\nabla K_1(x-y)\,
  d(\mu_{X_N}-\rho)^{\otimes2}(x,y).
\]
The difference $v(x)-v(y)$ cancels one order of the singularity.
The next lemma bounds this commutator by the modulated energy after
adding a finite-particle correction.

\begin{lemma}[Static Coulomb bounds]
\label{prop:static-coulomb-bounds}
Let $N\ge2$, $X_N\in\mathcal X_N$, and let $\rho$ be a
probability density satisfying $\|\rho\|_{L^\infty}\le C_*$.
For a sufficiently large dimensional constant $C_d$, define
\begin{equation}
  \err
  :=C_d\left[
    \mathbf 1_{\{d=2\}}\frac{\log(NC_*)}{N}
    +C_*^{(d-2)/d}N^{-2/d}
  \right].
  \label{eq:finite-particle-error}
\end{equation}
Then
\begin{equation}
  \frac12F_N(X_N,\rho)+\err
  \ge\frac12\err\ge0,
  \label{eq:modulated-energy-lower-bound}
\end{equation}
and every $v\in W^{1,\infty}(\T^d;\R^d)$ satisfies
\begin{equation}
  |\Comm_N[v](X_N,\rho)|
  \le C_d\|\nabla v\|_{L^\infty}
  \left(\frac12F_N(X_N,\rho)+\err\right).
  \label{eq:sharp-commutator}
\end{equation}
\end{lemma}

\begin{proof}
The periodic Coulomb lower bound in
\cite[Remark~6.5, Eq.~(6.10)]{ChodronRosenzweigSerfaty2025}
controls the negative part of $F_N$ by the two terms in
\eqref{eq:finite-particle-error}.  Choosing $C_d$ large enough
leaves half of $\err$ in \eqref{eq:modulated-energy-lower-bound}.
The periodic commutator estimate
\cite[Proposition~2.13]{ChodronRosenzweigSerfaty2025}
then gives \eqref{eq:sharp-commutator} with the same choice after
increasing $C_d$ if needed.
\end{proof}

The identity \eqref{eq:deterministic-modulated-identity} at
$\tau=0$ and the inequality \eqref{eq:targeted-mfe-inequality}
at $\tau>0$ each have $\Comm_N[u_t^\tau]$ as their only
term without a fixed sign.  Applying \eqref{eq:sharp-commutator}
to that term gives the estimate used in the main theorems.

\begin{proposition}[Particle estimate by velocity gradient]
\label{prop:assembly}
Let $\rho^\tau$ be the global mean-field solution and let
$u^\tau$ be given by \eqref{eq:free-energy-velocity}.
\begin{enumerate}
\item If $\tau>0$, let $f_N^\tau$ be a global entropy solution
of \eqref{eq:intro-liouville-equation} in the sense of
Definition~\ref{def:targeted-entropy-solution}.  For almost every
$t\ge0$,
\begin{equation}
  0\le\mathcal E_N^\tau(t)+\err
  \le\exp\left(\frac{C_d}{2}
      \int_0^t\|\nabla u_s^\tau\|_{L^\infty}\,ds\right)
      \bigl(\mathcal E_N^\tau(0)+\err\bigr).
  \label{eq:noisy-particle-gronwall}
\end{equation}
\item If $\tau=0$, let $X_N^0(t)$ be the deterministic
solution of \eqref{eq:intro-particle-system} starting from
$X_N^0(0)\in\mathcal X_N$.  For every $t\ge0$,
\begin{equation}
\begin{aligned}
  0&\le\frac12F_N(X_N^0(t),\rho_t^0)+\err\\
   &\le\exp\left(\frac{C_d}{2}
       \int_0^t\|\nabla u_s^0\|_{L^\infty}\,ds\right)
       \left[\frac12F_N(X_N^0(0),\rho_0)+\err\right].
\end{aligned}
  \label{eq:deterministic-particle-gronwall}
\end{equation}
\end{enumerate}
With i.i.d.\ initial positions of density $\rho_0$, the initial
quantity on the right of \eqref{eq:noisy-particle-gronwall} is at
most $\err$.  The expectation of the initial quantity in
\eqref{eq:deterministic-particle-gronwall} has the same bound.
\end{proposition}

\subsection{Particle inequalities}

We prove Proposition~\ref{prop:assembly} in two stages.  First,
Proposition~\ref{prop:deterministic-modulated-identity} derives an
identity along a deterministic particle trajectory.  Then
Proposition~\ref{prop:targeted-mfe} gives an inequality for the law
at positive temperature.  Both calculations expose the commutator
controlled by Lemma~\ref{prop:static-coulomb-bounds}.  We apply that
lemma after proving the two particle inequalities.

For $X_N=(x_1,\ldots,x_N)\in\mathcal X_N$, define the error in the
Coulomb field at particle $i$ by
\begin{equation}
  E_i(X_N,\rho)
  :=\frac1N\sum_{j\ne i}\nabla K_1(x_i-x_j)
      -\nabla K_1*\rho(x_i),
  \qquad 1\le i\le N.
  \label{eq:microscopic-field-error}
\end{equation}
The sum omits the field of particle $i$ acting on itself.  When
writing $\nabla K_1*(\mu_{X_N}-\rho)(x_i)$, we will always mean
the expression in \eqref{eq:microscopic-field-error}.  This is the
principal-value convention at a particle point.

For the rest of the deterministic calculation, abbreviate
$\rho_t:=\rho_t^0$ and
$u_t:=u_t^0=\nabla K_1*(\rho_t-\pi)$.  The solution $\rho^0$
is the bounded global weak solution specified in
Subsection~\ref{subsec:intro-main-results}.  Subtracting and adding
$\nabla K_1*\rho_t(x_i)$ in the particle drift gives
\begin{equation}
  \dot x_i(t)
  =-E_i(X_N(t),\rho_t)-u_t(x_i(t)).
  \label{eq:particle-drift-field-error}
\end{equation}

\begin{proposition}[Deterministic modulated-energy identity]
\label{prop:deterministic-modulated-identity}
Let $X_N(0)\in\mathcal X_N$, and let $X_N(t)$ be its global
collision-free solution at $\tau=0$.  Then, for every $t\ge0$,
\begin{align}
  &\frac12F_N(X_N(t),\rho_t)
   +\int_0^t\frac1N\sum_{i=1}^N
       |E_i(X_N(s),\rho_s)|^2\,ds\notag\\
  &\qquad=\frac12F_N(X_N(0),\rho_0)
   -\frac12\int_0^t
       \Comm_N[u_s](X_N(s),\rho_s)\,ds.
  \label{eq:deterministic-modulated-identity}
\end{align}
Equivalently, for almost every $t>0$,
\begin{equation}
  \frac{d}{dt}\frac12F_N(X_N(t),\rho_t)
   +\frac1N\sum_{i=1}^N|E_i(X_N(t),\rho_t)|^2
  =-\frac12\Comm_N[u_t](X_N(t),\rho_t).
  \label{eq:deterministic-modulated-differential-identity}
\end{equation}
\end{proposition}

\begin{proof}
We differentiate the three terms of $F_N$ and use
\eqref{eq:particle-drift-field-error} to recover the nonnegative
sum of squared field errors.  The remaining transport term becomes
the commutator by symmetry of the kernel.  This follows the
deterministic calculation of \cite[Lemma~2.1 and Remark~2.2]{Serfaty2020},
with the prescribed target included in $u_t$.  The regularity in
Proposition~\ref{prop:tau0-endpoint} and the collision-free flow of
Proposition~\ref{prop:deterministic-particle-flow} justify the
differentiations below.

\noindent\emph{Differentiate the energy.}
Write $\mu_t:=\mu_{X_N(t)}$ and
$E_i:=E_i(X_N(t),\rho_t)$.  Since $K_1$ is even, the definition
of $F_N$ expands to
\begin{equation}
  F_N(X_N(t),\rho_t)
  =\frac1{N^2}\sum_{i\ne j}K_1(x_i-x_j)
   -\frac2N\sum_i(K_1*\rho_t)(x_i)
   +\int_{\T^d}(K_1*\rho_t)\rho_t\,dx.
  \label{eq:deterministic-energy-expansion}
\end{equation}
The particle derivatives of the first two terms combine into
$(2/N)\sum_i\dot x_i\cdot E_i$.  Self-adjointness of convolution
with $K_1$ combines the time derivatives of their density terms.
Thus
\begin{equation}
  \frac{d}{dt}F_N(X_N(t),\rho_t)
  =\frac2N\sum_i\dot x_i\cdot E_i
   -2\int_{\T^d}K_1*(\mu_t-\rho_t)\,
            \partial_t\rho_t\,dx.
  \label{eq:deterministic-energy-derivative}
\end{equation}
The mean-field equation is
$\partial_t\rho_t=\nabla\cdot(\rho_tu_t)$.  Integration by parts
in \eqref{eq:deterministic-energy-derivative} and substitution of
\eqref{eq:particle-drift-field-error} therefore give
\begin{equation}
\begin{aligned}
  \frac{d}{dt}F_N(X_N(t),\rho_t)
  ={}&-\frac2N\sum_i|E_i|^2
       -\frac2N\sum_i u_t(x_i)\cdot E_i\\
      &+2\int_{\T^d}\rho_tu_t\cdot
             \nabla K_1*(\mu_t-\rho_t)\,dx.
\end{aligned}
  \label{eq:deterministic-energy-dissipation-remainder}
\end{equation}

\noindent\emph{Identify the commutator.}
Set $\nu_t:=\mu_t-\rho_t$.  In the first term involving $u_t$
in \eqref{eq:deterministic-energy-dissipation-remainder}, the field
at $x_i$ is $E_i$, with its diagonal interaction omitted.  Hence
\begin{equation}
  \frac1N\sum_i u_t(x_i)\cdot E_i
  =\iint_{(\T^d)^2\setminus\Delta}
      u_t(x)\cdot\nabla K_1(x-y)\,
      d\mu_t(x)\,d\nu_t(y).
  \label{eq:field-error-pv-identity}
\end{equation}
The last term of \eqref{eq:deterministic-energy-dissipation-remainder}
is the same double integral with $d\rho_t(x)$ in place of
$d\mu_t(x)$.  Their difference is the integral
\[
  I_t:=\iint_{(\T^d)^2\setminus\Delta}
      u_t(x)\cdot\nabla K_1(x-y)\,d\nu_t^{\otimes2}(x,y).
\]
Exchanging $x$ and $y$ shows that $I_t$ also equals the
negative of the same integral with $u_t(y)$ in place of
$u_t(x)$.  The oddness of $\nabla K_1$ therefore gives
\begin{equation}
  2I_t=\Comm_N[u_t](X_N(t),\rho_t).
  \label{eq:commutator-symmetrization}
\end{equation}
Equations~\eqref{eq:deterministic-energy-dissipation-remainder}
and~\eqref{eq:commutator-symmetrization} yield the differential
identity.  Integration from $0$ to $t$ proves
\eqref{eq:deterministic-modulated-identity}.
\end{proof}

We now treat $\tau>0$.  The particle law solves a diffusion
equation, so the deterministic identity is replaced by an inequality
for $\mathcal E_N^\tau$.  To formulate that inequality for the
singular kernel, we first specify the weak solutions used here.

For fixed $\tau>0$, define the unnormalized particle Gibbs
weight on $\mathcal X_N$ by
\begin{equation}
  G_N(X_N)
  =\exp\left[-\frac1{2N\tau}
            \sum_{i\ne j}K_1(x_i-x_j)
            +\frac1\tau\sum_{i=1}^N K_1*\pi(x_i)\right].
  \label{eq:targeted-gibbs-weight}
\end{equation}
Its logarithmic derivative is the particle drift divided by
$\tau$.  Thus \eqref{eq:intro-liouville-equation} takes the
weighted form
\begin{equation}
  \partial_t f_N^\tau
  =\tau\sum_{i=1}^N\nabla_{x_i}\cdot J_i,
  \qquad
  J_i:=G_N\nabla_{x_i}\frac{f_N^\tau}{G_N}.
  \label{eq:singular-targeted-weighted-equation}
\end{equation}
The quantity $J_i$ is the weighted particle flux.  The entropy
associated with this equation is
$\int f\log(f/G_N)\,dX_N$, where $f$ is a probability density.
The next definition asks that this entropy dissipate even when the
Coulomb drift is singular.

\begin{definition}[Entropy solution of the targeted Liouville equation]
\label{def:targeted-entropy-solution}
Fix $T>0$ and $\tau>0$.  A nonnegative symmetric density
$f_N^\tau\in L^\infty(0,T;L^1((\T^d)^N))$ of unit mass is an
entropy solution of \eqref{eq:intro-liouville-equation} with
initial probability density $f_N^\tau(0)$ if the following hold.
The initial weighted entropy is finite.  The flux $J_i$ in
\eqref{eq:singular-targeted-weighted-equation}, initially understood
in distributions off the collision set, belongs to
$L^1((0,T)\times(\T^d)^N)$ for each $i$.  Equation
\eqref{eq:singular-targeted-weighted-equation} holds in distributions
on the full configuration space, and for almost every $t\in[0,T]$,
\begin{align}
  &\int_{(\T^d)^N}f_N^\tau(t)
       \log\frac{f_N^\tau(t)}{G_N}\,dX_N
   +\tau\sum_{i=1}^N\int_0^t\int_{(\T^d)^N}
       \frac{|J_i(s)|^2}{f_N^\tau(s)}\,dX_N\,ds\notag\\
  &\qquad\le\int_{(\T^d)^N}f_N^\tau(0)
       \log\frac{f_N^\tau(0)}{G_N}\,dX_N.
  \label{eq:targeted-entropy-inequality}
\end{align}
Here $|J_i|^2/f_N^\tau$ is zero where both numerator and
denominator vanish and is $+\infty$ where $f_N^\tau=0$ but
$J_i\ne0$.  The law $f_N^\tau(t)\,dX_N$ converges weakly to
$f_N^\tau(0)\,dX_N$ as $t\downarrow0$.
\end{definition}

An entropy solution is global if it satisfies the definition on
every finite interval.  The following proposition both constructs
such solutions and gives the particle inequality needed in
Proposition~\ref{prop:assembly}.

\begin{proposition}[Entropy solutions and the modulated free-energy inequality]
\label{prop:targeted-entropy-existence}
\label{prop:targeted-mfe}
Let $N\ge2$ and $\tau>0$.
\begin{enumerate}[label=\textup{(\roman*)},leftmargin=*,itemsep=3pt,topsep=3pt]
\item Every symmetric initial probability density with finite
weighted entropy has a symmetric global entropy solution of
\eqref{eq:intro-liouville-equation}.
\item Let $T>0$, let $f_N^\tau$ be any entropy solution on
$[0,T]$, and let $\rho^\tau$ be a positive solution of
\eqref{eq:intro-mean-field-pde} such that
\[
  \rho^\tau\in C([0,T];W^{2,\infty}(\T^d)),
  \qquad \inf_{[0,T]\times\T^d}\rho^\tau>0.
\]
Then, for almost every $t\in[0,T]$,
\begin{align}
  \mathcal E_N^\tau(t)
  &+\frac1N\sum_{i=1}^N
    \int_0^t\int_{(\T^d)^N}
    \left|\tau\nabla_{x_i}
       \log\frac{f_N^\tau(s)}{(\rho_s^\tau)^{\otimes N}}
      +E_i(X_N,\rho_s^\tau)\right|^2
       f_N^\tau(s)\,dX_N\,ds\notag\\
  &\qquad\le\mathcal E_N^\tau(0)
      -\frac12\int_0^t\mathbb E_{f_N^\tau(s)}\left[
        \Comm_N[u_s^\tau](\,\cdot\,,\rho_s^\tau)
       \right]ds.
  \label{eq:targeted-mfe-inequality}
\end{align}
The square in this display is understood through the weighted flux
in \eqref{eq:singular-targeted-weighted-equation}.
\end{enumerate}
\end{proposition}

The initial density $\rho_0^{\otimes N}$ has finite weighted
entropy under Assumption~\ref{assump:standing-data}.
For $\tau>0$, Proposition~\ref{prop:targeted-entropy-construction}
shows that the joint density of the particle system
\eqref{eq:intro-particle-system} with initial law
$\rho_0^{\otimes N}$ is a global entropy solution.

\begin{proof}[Proof of Proposition~\ref{prop:targeted-mfe}]
We construct a solution by smoothing the pair interaction while
keeping the smooth target force fixed.  We then derive the
inequality for smooth densities and justify it directly for
every entropy solution.  The construction adapts the compactness
argument of
\cite[Lemma~6.2 and Appendix~A]{ChodronRosenzweigSerfaty2025},
using heat regularization in place of their truncation.
The weighted comparison follows
\cite[Proposition~2.3]{BreschJabinWang2023};
details for this target are given in
Subsections~\ref{subsec:targeted-entropy-compactness}
and~\ref{subsec:singular-targeted-mfe-passage}.

\noindent\emph{Existence of an entropy solution.}
For $\varepsilon>0$, replace the pair kernel by
$K_{1,\varepsilon}:=e^{\varepsilon\Delta}K_1$ and retain the
term $K_1*\pi$ in \eqref{eq:targeted-gibbs-weight}.  Denote the
resulting smooth weight, solution, and flux by
$G_{N,\varepsilon}$, $f_{N,\varepsilon}^\tau$, and
$J_{i,\varepsilon}$.  Differentiating the weighted entropy for
the smooth equation gives
\begin{equation}
\begin{aligned}
  &\int f_{N,\varepsilon}^\tau(t)
       \log\frac{f_{N,\varepsilon}^\tau(t)}{G_{N,\varepsilon}}\,dX_N
   +\tau\sum_i\int_0^t\int
       \frac{|J_{i,\varepsilon}|^2}
            {f_{N,\varepsilon}^\tau}\,dX_N\,ds\\
  &\qquad\le\int f_{N,\varepsilon}^\tau(0)
       \log\frac{f_{N,\varepsilon}^\tau(0)}{G_{N,\varepsilon}}\,dX_N.
\end{aligned}
  \label{eq:regularized-weighted-entropy}
\end{equation}
The Coulomb heat regularization obeys
$K_{1,\varepsilon}-\varepsilon\uparrow K_1$ as
$\varepsilon\downarrow0$.  With the prescribed initial density
fixed, this bounds the right side of
\eqref{eq:regularized-weighted-entropy} uniformly in
$\varepsilon$.  An upper bound on $G_{N,\varepsilon}$ gives a
lower bound on the weighted entropy, so both the entropy and its
nonnegative flux dissipation are bounded on every finite time
interval.

The same upper bound on the weight controls the ordinary entropy
of $f_{N,\varepsilon}^\tau$, namely
$\int f_{N,\varepsilon}^\tau\log f_{N,\varepsilon}^\tau\,dX_N$.
It makes the densities uniformly integrable.  Cauchy--Schwarz and
the dissipation bound then make the fluxes uniformly integrable as
well, since for a spacetime set $A$,
\begin{equation}
  \int_A|J_{i,\varepsilon}|
  \le\left(\int_A\frac{|J_{i,\varepsilon}|^2}
                           {f_{N,\varepsilon}^\tau}\right)^{1/2}
      \left(\int_A f_{N,\varepsilon}^\tau\right)^{1/2}.
  \label{eq:regularized-flux-uniform-integrability}
\end{equation}
We may therefore pass to weak $L^1$ limits in the weighted
equation.  Off the collision set, the regularized drifts converge
locally uniformly, so integration by parts identifies the limiting
flux with $G_N\nabla(f_N^\tau/G_N)$.  The collision set has
Lebesgue measure zero, and the identity holds almost everywhere on
the full configuration space.  Lower semicontinuity transfers
\eqref{eq:regularized-weighted-entropy} to
\eqref{eq:targeted-entropy-inequality}.  The initial law and one
global solution follow by the compactness and diagonal argument in
Subsection~\ref{subsec:targeted-entropy-compactness}.

\noindent\emph{The modulated free-energy inequality.}
Fix $T>0$ and write $\rho_t:=\rho_t^\tau$.  We choose a
time-dependent mean-field weight so that its ratio with $G_N$
recovers the modulated energy.  Define
\begin{equation}
\begin{aligned}
  G_\rho(t,x)
  &:=\exp\left[-\frac1\tau K_1*(\rho_t-\pi)(x)
       +\frac1{2\tau}\int_{\T^d}(K_1*\rho_t)\rho_t\,dx\right],\\
  G_{\rho,N}(t,X_N)&:=\prod_{i=1}^N G_\rho(t,x_i).
\end{aligned}
  \label{eq:mean-field-gibbs-weight}
\end{equation}
Expanding $F_N$ as in
\eqref{eq:deterministic-energy-expansion} shows that the choice of
constant in $G_\rho$ gives
\begin{equation}
  \frac{G_N(X_N)}{G_{\rho,N}(t,X_N)}
   =\exp\left[-\frac{N}{2\tau}F_N(X_N,\rho_t)\right].
  \label{eq:target-weight-ratio}
\end{equation}
The logarithmic derivative of $G_{\rho,N}$, together with
\eqref{eq:free-energy-velocity}, also gives
\begin{equation}
  \tau\nabla_{x_i}
       \log\frac{\rho_t^{\otimes N}}{G_{\rho,N}}
     =u_t^\tau(x_i),
  \qquad
  \tau\nabla_{x_i}\log\frac{G_{\rho,N}}{G_N}
     =E_i(X_N,\rho_t).
  \label{eq:weighted-gradient-identities}
\end{equation}
Thus the tensorized mean-field equation has the same weighted
diffusion form as the particle equation, with $G_{\rho,N}$ in
place of $G_N$.  More directly, \eqref{eq:target-weight-ratio}
turns the definition of $\mathcal E_N^\tau$ into
\begin{equation}
  \frac{\mathcal E_N^\tau(t)}{\tau}
  =\frac1N\int_{(\T^d)^N}f_N^\tau(t)
     \log\left(\frac{f_N^\tau(t)}{G_N}
                 \frac{G_{\rho,N}(t)}{\rho_t^{\otimes N}}\right)dX_N.
  \label{eq:weighted-entropy-representation}
\end{equation}

We first use a smooth even kernel and smooth positive densities to
differentiate \eqref{eq:weighted-entropy-representation}.  Put
$Z_i:=\tau\nabla_{x_i}\log(f_N^\tau/\rho_t^{\otimes N})
        +E_i(X_N,\rho_t)$.  The two identities in
\eqref{eq:weighted-gradient-identities} give
\begin{equation}
  \tau J_i=f_N^\tau(Z_i+u_t^\tau(x_i)),
  \qquad
  \tau\nabla_{x_i}\log\left(
       \frac{f_N^\tau}{G_N}
       \frac{G_{\rho,N}}{\rho_t^{\otimes N}}\right)=Z_i.
  \label{eq:weighted-flux-square-terms}
\end{equation}
Differentiating the first weighted entropy in
\eqref{eq:weighted-entropy-representation} by the particle equation
and differentiating its second term in time yield
\begin{equation}
\begin{aligned}
  \frac{d}{dt}\mathcal E_N^\tau(t)
  ={}&-\frac1N\sum_i\int_{(\T^d)^N}
       f_N^\tau|Z_i|^2\,dX_N
      -\frac1N\sum_i\int_{(\T^d)^N}
       f_N^\tau u_t^\tau(x_i)\cdot Z_i\,dX_N\\
     &+\frac\tau N\int_{(\T^d)^N}f_N^\tau
       \partial_t\log\frac{G_{\rho,N}}{\rho_t^{\otimes N}}\,dX_N.
\end{aligned}
  \label{eq:weighted-square-intermediate}
\end{equation}
Using $\partial_t\rho_t=\nabla\cdot(\rho_tu_t^\tau)$
in the last line, the two terms that involve $u_t^\tau$
combine into $-\tfrac12\mathbb E_{f_N^\tau(t)}
[\Comm_N[u_t^\tau](\,\cdot\,,\rho_t)]$.  This is the same
exchange of $x$ and $y$ as in
\eqref{eq:commutator-symmetrization}.  We obtain the smooth identity
\begin{equation}
\begin{aligned}
  \frac{d}{dt}\mathcal E_N^\tau(t)
  ={}&-\frac1N\sum_i\int_{(\T^d)^N}
       \left|\tau\nabla_{x_i}
          \log\frac{f_N^\tau}{\rho_t^{\otimes N}}
            +E_i(X_N,\rho_t)\right|^2
             f_N^\tau\,dX_N\\
     &-\frac12\mathbb E_{f_N^\tau(t)}
          [\Comm_N[u_t^\tau](\,\cdot\,,\rho_t)].
\end{aligned}
  \label{eq:smooth-targeted-mfe-identity}
\end{equation}
The full integration by parts for the time-dependent weight is in
Subsection~\ref{subsec:smooth-weighted-calculation}.

For an entropy solution, we use the weighted entropy inequality
\eqref{eq:targeted-entropy-inequality} in place of differentiating
the first term of \eqref{eq:weighted-entropy-representation}.
The weak equation \eqref{eq:singular-targeted-weighted-equation}
applied to the test function
$\log(\rho_t^{\otimes N}/G_{\rho,N}(t))$ gives the change in its
second term.  Combining the two calculations reproduces the square
and commutator in \eqref{eq:smooth-targeted-mfe-identity}, now with
an inequality.  Subsection~\ref{subsec:singular-targeted-mfe-passage}
justifies the passage at the Coulomb singularity and yields
\eqref{eq:targeted-mfe-inequality}.
\end{proof}

\begin{proof}[Proof of Proposition~\ref{prop:assembly}]
We bound the only unsigned term in the two particle inequalities by
Lemma~\ref{prop:static-coulomb-bounds}, then apply Gr\"onwall's
lemma.  At positive temperature, relative entropy is nonnegative,
so \eqref{eq:modulated-energy-lower-bound} implies
$\mathcal E_N^\tau(t)+\err\ge0$.  Averaging
\eqref{eq:sharp-commutator} against $f_N^\tau(t)$ gives
\begin{equation}
  \mathbb E_{f_N^\tau(t)}
    |\Comm_N[u_t^\tau](\,\cdot\,,\rho_t^\tau)|
  \le C_d\|\nabla u_t^\tau\|_{L^\infty}
       (\mathcal E_N^\tau(t)+\err).
  \label{eq:assembly-commutator-average}
\end{equation}
Discarding the nonnegative integral in
\eqref{eq:targeted-mfe-inequality}, we consequently have
\begin{equation}
  \mathcal E_N^\tau(t)+\err
  \le\mathcal E_N^\tau(0)+\err
    +\frac{C_d}{2}\int_0^t
       \|\nabla u_s^\tau\|_{L^\infty}
       (\mathcal E_N^\tau(s)+\err)\,ds.
  \label{eq:assembly-noisy-integral-inequality}
\end{equation}
Gr\"onwall's lemma proves \eqref{eq:noisy-particle-gronwall}.
For $\tau=0$, the lower bound
\eqref{eq:modulated-energy-lower-bound} and
\eqref{eq:deterministic-modulated-identity} give the same integral
inequality along each trajectory.  Gr\"onwall's lemma then gives
\eqref{eq:deterministic-particle-gronwall}.

It remains to evaluate the initial quantities for i.i.d.\ data.
The positive-temperature relative entropy vanishes at
$f_N^\tau(0)=\rho_0^{\otimes N}$.  Independence and omission of
the diagonal in $F_N$ give
\begin{equation}
\begin{aligned}
  \mathbb E_{\rho_0^{\otimes N}}F_N(X_N,\rho_0)
  &=-\frac1N\iint_{(\T^d)^2}
     K_1(x-y)\rho_0(x)\rho_0(y)\,dx\,dy\\
  &=-\frac1N\|\rho_0-1\|_{\dot H^{-1}}^2.
\end{aligned}
  \label{eq:iid-modulated-energy-calculation}
\end{equation}
Here $K_1*1=0$, so the last equality follows from the
Coulomb energy representation of the $\dot H^{-1}$ norm in the
introduction.  In particular,
\begin{equation}
  0\le\mathcal E_N^\tau(0)+\err
   =-\frac1{2N}\|\rho_0-1\|_{\dot H^{-1}}^2+\err
   \le\err.
  \label{eq:iid-initial-free-energy}
\end{equation}
At zero temperature, the same identity bounds the expected initial
shifted energy.  Initial configurations are collision-free almost
surely by Corollary~\ref{cor:random-initial-collision-free}, so we
may average \eqref{eq:deterministic-particle-gronwall} over them.
This proves both initialization claims.
\end{proof}

Proposition~\ref{prop:assembly} now leaves one mean-field quantity
to control.  The next section proves a bound, uniform in
$\tau\ge0$, for
\[
  \int_0^\infty\|\nabla u_s^\tau\|_{L^\infty}\,ds.
\]

\section{Uniform control of the mean-field velocity gradient}
\label{sec:stationary-family}
\label{sec:mean-field-analysis}

The particle estimates of Section~\ref{sec:particle-convergence} depend
on the time integral of the spatial derivative of the mean-field
velocity
\begin{equation}
  u_t^\tau
  :=\tau\nabla\log\rho_t^\tau
    +\nabla K_1*(\rho_t^\tau-\pi),
  \qquad \tau\ge0.
  \label{eq:free-energy-velocity}
\end{equation}
where the first term is zero when $\tau=0$.  The goal of this section
is to bound that integral with one constant for all temperatures.  We
prove the stronger pointwise estimate below.

\begin{theorem}[Uniform velocity-gradient bounds]
\label{thm:uniform-envelope}
Under Assumption~\ref{assump:standing-data}, there are constants
$C<\infty$ and $\gamma>0$, independent of $t$ and $\tau$, such
that
\begin{equation}
  \|\nabla u_t^\tau\|_{L^\infty}
  \le C(1+\tau)e^{-\gamma(1+\tau)t},
  \qquad t\ge0,\quad\tau\ge0.
  \label{eq:uniform-velocity-envelope}
\end{equation}
In particular, the quantity required by the particle estimates is
finite:
\begin{equation}
  \Psi_\infty
  :=\sup_{\tau\ge0}\int_0^\infty
      \|\nabla u_t^\tau\|_{L^\infty}\,dt
  \le \frac C\gamma<\infty.
  \label{eq:uniform-velocity-gradient-decay}
\end{equation}
\end{theorem}

At positive temperature we use the following existence and regularity
result for the mean-field equation~\eqref{eq:intro-mean-field-pde}.
Its proof, including the continuation argument, is in
Appendix~\ref{app:positive-temperature-wellposedness}.

\begin{lemma}[Positive-temperature global well-posedness and regularity]
\label{lem:positive-temperature-wp}
For every $\tau>0$, equation~\eqref{eq:intro-mean-field-pde} has a
unique global classical solution $\rho^\tau$ with
$\rho_t^\tau\in\mathcal P(\T^d)$.  It is unique among bounded mild
solutions.  For every $\tau_{\mathrm L}>0$, there is a finite
constant $C_{\tau_{\mathrm L}}$ such that
\begin{equation}
  \sup_{\tau\ge\tau_{\mathrm L}}\sup_{t\ge0}
  \|\rho_t^\tau\|_{W^{4,\infty}(\T^d)}
  \le C_{\tau_{\mathrm L}}.
  \label{eq:positive-temperature-W4}
\end{equation}
\end{lemma}

The bound~\eqref{eq:positive-temperature-W4} will supply the strong
norms used in Section~\ref{subsec:high-temperature-estimates} for
temperatures bounded away from zero.  Its constant can depend on
$\tau_{\mathrm L}$, so the proof near zero temperature will instead
use estimates whose constants do not involve parabolic smoothing.
We begin with the stationary state and an energy estimate common to all
positive temperatures in Section~\ref{subsec:stationary-family}.  We
then treat temperatures near zero in Section~\ref{subsec:low-temperature}
and temperatures bounded away from zero in
Section~\ref{subsec:high-temperature-estimates}.  The first of these
arguments also recovers the endpoint $\tau=0$.

\subsection{Thermal equilibrium and common energy estimates}
\label{subsec:stationary-family}

The velocity at a positive-temperature stationary state vanishes.
To use this fact, we first construct the stationary density
$\rho_\infty^\tau$ and bound it uniformly in $\tau$ in
Propositions~\ref{prop:thermal-equilibrium-family}
and~\ref{prop:stationary-all-orders}.  Proposition~\ref{prop:maximum-principle}
keeps the evolving density positive.  Lemma~\ref{lem:exact-free-energy-gap}
and Proposition~\ref{prop:hminus-exp} then show that
$\rho_t^\tau-\rho_\infty^\tau$ decays in $\dot H^{-1}$.  The next
two subsections strengthen this decay until it controls
$\nabla u_t^\tau$.

\begin{proposition}[Thermal equilibrium family]
\label{prop:thermal-equilibrium-family}
For each $\tau>0$, there is a unique probability density
$\rho_\infty^\tau$ and a unique constant $\lambda_\tau\in\R$
such that
\begin{equation}
  \tau\log\rho_\infty^\tau
   +K_1*(\rho_\infty^\tau-\pi)
   =\lambda_\tau,
  \qquad
  c_\pi\le\rho_\infty^\tau\le C_\pi.
  \label{eq:stationary-euler-lagrange}
\end{equation}
We write $\rho_\infty^0:=\pi$.
\end{proposition}

\begin{proof}
We construct $\rho_\infty^\tau$ by solving a local equation for
its logarithm.  If $q=\log\rho_\infty^\tau$, applying $-\Delta$
to the first identity in~\eqref{eq:stationary-euler-lagrange} gives
\begin{equation}
  -\tau\Delta q+e^q=\pi.
  \label{eq:stationary-local-equation}
\end{equation}
Conversely, a solution of~\eqref{eq:stationary-local-equation} has
$\int_{\T^d}e^q=\int_{\T^d}\pi=1$.  Since
$-\Delta[K_1*(e^q-\pi)]=e^q-\pi$, we have
\[
  -\Delta\bigl[\tau q+K_1*(e^q-\pi)\bigr]
  =-\tau\Delta q+e^q-\pi=0.
\]
The expression in brackets is harmonic on the torus and therefore
constant, as required by~\eqref{eq:stationary-euler-lagrange}. To obtain a solution of~\eqref{eq:stationary-local-equation}, minimize
the strictly convex functional
\begin{equation}
  J_\tau(q)
  :=\int_{\T^d}
     \left(\frac\tau2|\nabla q|^2+e^q-\pi q\right)dx
  \label{eq:stationary-auxiliary-functional}
\end{equation}
over $q\in H^1(\T^d)$ for which $e^q$ is integrable.  Writing
$q=m+h$ with $m=\int_{\T^d}q$ and $\int_{\T^d}h=0$, Jensen's
and Poincar\'e's inequalities give
\[
  J_\tau(q)
  \ge \frac\tau2\|\nabla h\|_{L^2}^2
       -C\|\nabla h\|_{L^2}+e^m-m.
\]
Thus a minimizing sequence is bounded in $H^1$.  Compactness in
$L^2$ and lower semicontinuity of the gradient and exponential
integrals yield a minimizer.  Variations by bounded functions give
its weak Euler--Lagrange equation~\eqref{eq:stationary-local-equation}.

It remains to show that the solution stays between the extremal
values of $\pi$.  Test the weak equation for
$-\tau\Delta q+e^q-\pi=0$ first with a bounded truncation of
$(\log c_\pi-q)_+$.  On its support both terms in the resulting
identity are nonpositive:
\[
  0=-\tau\!\int_{\{0<\log c_\pi-q<M\}}\!|\nabla q|^2dx
     +\int_{\T^d}(e^q-\pi)
       \min\{(\log c_\pi-q)_+,M\}\,dx.
\]
Hence $q\ge\log c_\pi$.  Testing with a bounded truncation of
$(q-\log C_\pi)_+$ likewise gives $q\le\log C_\pi$.  Standard
elliptic regularity now makes $q$ smooth enough for the preceding
identities to hold pointwise.  If $q_1,q_2$ are two solutions,
testing the difference of their equations by $q_1-q_2$ gives
\[
  0=\tau\|\nabla(q_1-q_2)\|_{L^2}^2
    +\int_{\T^d}(e^{q_1}-e^{q_2})(q_1-q_2)\,dx.
\]
Both terms are nonnegative, and the second vanishes only when
$q_1=q_2$.  This proves uniqueness of
$\rho_\infty^\tau=e^q$, and the constant in
\eqref{eq:stationary-euler-lagrange} is then fixed.  Further
regularity details are in Subsection~\ref{subsec:stationary-estimates}.
\end{proof}

The elliptic equation~\eqref{eq:stationary-local-equation} also yields
derivative estimates without factors of $\tau^{-1}$.  After one or
more differentiations, the highest derivative is acted on by the
operator $\rho_\infty^\tau-\tau\Delta$.  The lower bound in
Proposition~\ref{prop:thermal-equilibrium-family} makes its
maximum-principle estimate uniform in temperature.

\begin{proposition}[Uniform stationary regularity]
\label{prop:stationary-all-orders}
Under Assumption~\ref{assump:standing-data}, a constant $C$ depending
only on $d$, $c_\pi^{-1}$, and $\|\pi\|_{C^4}$ satisfies
\begin{equation}
  \sup_{\tau\ge0}
  \bigl(\|\rho_\infty^\tau\|_{C^4}
       +\|\log\rho_\infty^\tau\|_{C^4}\bigr)
  \le C.
  \label{eq:uniform-stationary-regularity}
\end{equation}
\end{proposition}

\begin{proof}
Fix $\tau>0$ and put $q_\tau=\log\rho_\infty^\tau$.  The
regularity argument in Subsection~\ref{subsec:stationary-estimates}
first gives, for $p>\max\{d,2\}$,
\[
  q_\tau\in W^{2,p}(\T^d)
   \longrightarrow W^{4,p}(\T^d)
   \longrightarrow W^{6,p}(\T^d)
   \subset C^5(\T^d).
\]
These bounds at a fixed $\tau$ justify the differentiated equation.
We now derive bounds that are independent of $\tau$. For any solution of
$(\rho_\infty^\tau-\tau\Delta)w=f$, the pointwise maximum
principle gives
\begin{equation}
  \|w\|_{L^\infty}
  \le c_\pi^{-1}\|f\|_{L^\infty}.
  \label{eq:stationary-resolvent-bound}
\end{equation}
Truncation tests give the same estimate when
$w$ is only a weak solution. For a multi-index $\alpha$ with $|\alpha|=k$, differentiation of
\eqref{eq:stationary-local-equation} separates the highest derivative
of $q_\tau$ from products of lower derivatives:
\begin{equation}
  (\rho_\infty^\tau-\tau\Delta)\partial^\alpha q_\tau
  =\partial^\alpha\pi
    -\rho_\infty^\tau
       \mathcal P_\alpha(Dq_\tau,\ldots,D^{k-1}q_\tau),
  \qquad 1\le k\le4.
  \label{eq:stationary-differentiated-equation}
\end{equation}
Here $\mathcal P_\alpha$ is the finite sum of products from the
chain rule that contain at least two derivatives of $q_\tau$.  It
vanishes for $k=1$.  Applying
\eqref{eq:stationary-resolvent-bound} to
\eqref{eq:stationary-differentiated-equation} and inducting over
$k=1,\ldots,4$ bounds $\|q_\tau\|_{C^4}$ uniformly.  The chain
rule then bounds $\rho_\infty^\tau=e^{q_\tau}$ in $C^4$.  At
$\tau=0$, the assertion follows from $\rho_\infty^0=\pi$.
\end{proof}

We can now express the velocity relative to equilibrium.  Write
$\widetilde\rho_t^\tau:=\rho_t^\tau-\rho_\infty^\tau$ for the
density error.  Taking the
gradient of~\eqref{eq:stationary-euler-lagrange} shows that the
equilibrium velocity is zero:
\[
  u_\infty^\tau
  :=\tau\nabla\log\rho_\infty^\tau
    +\nabla K_1*(\rho_\infty^\tau-\pi)=0.
\]
Subtracting this identity from the definition of $u_t^\tau$ gives
the two terms we must estimate in the next subsections:
\begin{equation}
  \nabla u_t^\tau
  =\nabla^2K_1*\widetilde\rho_t^\tau
   +\tau\bigl(\nabla^2\log\rho_t^\tau
                   -\nabla^2\log\rho_\infty^\tau\bigr).
  \label{eq:equilibrium-centered-deformation}
\end{equation}
The first term is the Hessian of the Coulomb potential of the density
error.  The second term requires control of two derivatives of the
logarithmic density.  Before seeking these strong estimates, we
establish a lower bound on the density and a weak decay estimate that
hold uniformly in temperature.

The following maximum principle supplies that lower bound.  Its
positive-temperature proof follows the argument of
\cite[proof of Proposition~2.8]{BoufadeneVialard2025}.  We obtain
the zero-temperature assertion later by taking the limit in the
positive-temperature bounds, as part of
Proposition~\ref{prop:tau0-endpoint}.

\begin{proposition}[Maximum principle]
\label{prop:maximum-principle}
\label{prop:density-bounds}
For every $\tau>0$, the classical solution of
\eqref{eq:intro-mean-field-pde} satisfies
\begin{equation}
  c_*\le\rho_t^\tau(x)\le C_*,
  \qquad t\ge0,\quad x\in\T^d.
  \label{eq:maximum-principle}
\end{equation}
The same bounds hold for the bounded global weak solution
$\rho^0$ at $\tau=0$.
\end{proposition}

\begin{proof}
We first prove the assertion for $\tau>0$.  The identity
$-\Delta K_1*(\rho_t^\tau-\pi)=\rho_t^\tau-\pi$ puts
the mean-field equation in the form
\begin{equation}
  \partial_t\rho_t^\tau
  =\tau\Delta\rho_t^\tau
   +\nabla K_1*(\rho_t^\tau-\pi)\cdot\nabla\rho_t^\tau
   -\rho_t^\tau(\rho_t^\tau-\pi).
  \label{eq:density-nondivergence-form}
\end{equation}
The transport term vanishes at a spatial extremum.  Set
$M(t)=\max_x\rho_t^\tau(x)$ and
$m(t)=\min_x\rho_t^\tau(x)$.  Hamilton's maximum lemma
(Lemma~\ref{lem:hamilton}) permits differentiation of these extremal
values for almost every time and gives
\begin{equation}
  M'(t)\le M(t)(C_\pi-M(t)),
  \qquad
  m'(t)\ge m(t)(c_\pi-m(t)).
  \label{eq:density-extremal-inequalities}
\end{equation}
The scalar comparison principle and the initial bounds in
Assumption~\ref{assump:standing-data} keep $M(t)$ below $C_*$
and $m(t)$ above $c_*$.  The solution exists globally by
Lemma~\ref{lem:positive-temperature-wp}, so this gives
\eqref{eq:maximum-principle} for all $t\ge0$.  Proposition~\ref{prop:tau0-endpoint}
identifies $\rho^0$ as a limit of positive-temperature solutions
and transfers the same pointwise bounds to that limit.
\end{proof}

We next measure relaxation toward equilibrium by subtracting its
free energy from that of an arbitrary density.  For probability
densities $\rho$ and $\nu>0$, write
\[
  \operatorname{Ent}(\rho\mid\nu)
  :=\int_{\T^d}\rho\log\frac\rho\nu\,dx.
\]
The equilibrium identity makes the free-energy gap an exact sum of
the Coulomb distance and this relative entropy.

\begin{lemma}[Recentered free energy and dissipation]
\label{lem:exact-free-energy-gap}
Let $\tau>0$.  Every probability density $\rho$ with finite
free energy satisfies
\begin{equation}
  \widetilde{\mathcal F}_\tau(\rho)
  :=\mathcal F_\tau(\rho)-\mathcal F_\tau(\rho_\infty^\tau)
  =\frac12\|\rho-\rho_\infty^\tau\|_{\dot H^{-1}}^2
   +\tau\operatorname{Ent}(\rho\mid\rho_\infty^\tau).
  \label{eq:exact-free-energy-gap}
\end{equation}
If $\widetilde{\mathcal F}_\tau(t)
 :=\widetilde{\mathcal F}_\tau(\rho_t^\tau)$, then
\begin{equation}
  \frac{d}{dt}\widetilde{\mathcal F}_\tau(t)
  =-\int_{\T^d}\rho_t^\tau|u_t^\tau|^2\,dx.
  \label{eq:continuum-dissipation-identity}
\end{equation}
\end{lemma}

\begin{proof}
The first claim follows by expanding both terms of
$\mathcal F_\tau$ around $\rho_\infty^\tau$.  The terms linear
in $\rho-\rho_\infty^\tau$ combine into
\begin{equation}
  \widetilde{\mathcal F}_\tau(\rho)
  =\frac12\|\rho-\rho_\infty^\tau\|_{\dot H^{-1}}^2
   +\tau\operatorname{Ent}(\rho\mid\rho_\infty^\tau)
   +\int_{\T^d}(\rho-\rho_\infty^\tau)
     \bigl[\tau\log\rho_\infty^\tau
           +K_1*(\rho_\infty^\tau-\pi)\bigr]dx.
  \label{eq:expanded-equilibrium-gap}
\end{equation}
The bracket equals the constant $\lambda_\tau$ by
\eqref{eq:stationary-euler-lagrange}.  The final integral is zero
because both densities have total mass one.  This proves
\eqref{eq:exact-free-energy-gap}.

For the second claim, mass conservation allows us to omit the
constant in the first variation of $\mathcal F_\tau$.  Differentiating
along the classical solution and using
$\partial_t\rho_t^\tau=\nabla\cdot(\rho_t^\tau u_t^\tau)$, we obtain
\begin{align*}
  \frac{d}{dt}\widetilde{\mathcal F}_\tau(t)
  &=\int_{\T^d}
     \bigl[K_1*(\rho_t^\tau-\pi)
       +\tau\log\rho_t^\tau\bigr]
       \partial_t\rho_t^\tau\,dx\\
  &=-\int_{\T^d}\rho_t^\tau
     \left|\nabla K_1*(\rho_t^\tau-\pi)
                 +\tau\nabla\log\rho_t^\tau\right|^2dx
   =-\int_{\T^d}\rho_t^\tau|u_t^\tau|^2dx.
\end{align*}
This is~\eqref{eq:continuum-dissipation-identity}.
\end{proof}

The identity~\eqref{eq:exact-free-energy-gap} controls the density
error in $\dot H^{-1}$, while
\eqref{eq:continuum-dissipation-identity} decreases that gap.  We now
compare the dissipation with the gap.  This gives the weak relaxation
estimate used in both temperature regimes.

\begin{proposition}[Uniform $\dot H^{-1}$ relaxation]
\label{prop:hminus-exp}
Set $\kappa_*:=(2\pi)^2/(2C_*)$.  For every $\tau>0$ and
$t\ge0$, the recentered free energy obeys
\begin{equation}
  \widetilde{\mathcal F}_\tau(t)
  \le\widetilde{\mathcal F}_\tau(0)
       e^{-c_*(\frac12+\kappa_*\tau)t}.
  \label{eq:appD-gap-accelerated-relaxation}
\end{equation}
Consequently, there is a constant $C$, independent of $\tau$ and
$t$, for which
\begin{align}
  \|\rho_t^\tau-\rho_\infty^\tau\|_{\dot H^{-1}}
  &\le C e^{-\frac{c_*}{2}(\frac12+\kappa_*\tau)t},
    \label{eq:appD-Hminus-accelerated-relaxation}\\
  \|\rho_t^\tau-\rho_\infty^\tau\|_{\dot H^{-1}}
  &\le C e^{-c_*t/4}.
    \label{eq:hminus-relaxation}
\end{align}
\end{proposition}

\begin{proof}
We first show that the gap controls the squared Coulomb distance with
a coefficient that grows with $\tau$.  Let
$g_t^\tau:=\rho_t^\tau-\rho_\infty^\tau$, which has mean zero.
For $h(r)=r\log r$, the bound $h''(r)=1/r\ge C_*^{-1}$ on
$[c_*,C_*]$ and Proposition~\ref{prop:maximum-principle} give
\begin{equation}
  \operatorname{Ent}(\rho_t^\tau\mid\rho_\infty^\tau)
  \ge\frac{1}{2C_*}\|g_t^\tau\|_{L^2}^2
  \ge\kappa_*\|g_t^\tau\|_{\dot H^{-1}}^2.
  \label{eq:equilibrium-entropy-lower-bound}
\end{equation}
The last step is Poincar\'e's inequality for mean-zero functions on
the unit torus.  Together with~\eqref{eq:exact-free-energy-gap}, it
yields $\dot H^{-1}$ coercivity of the recentered free energy
\begin{equation}
  \widetilde{\mathcal F}_\tau(t)
  \ge a_\tau\|g_t^\tau\|_{\dot H^{-1}}^2,
  \qquad a_\tau:=\frac12+\kappa_*\tau.
  \label{eq:appD-gap-coercivity}
\end{equation}
We next relate the dissipation to the same gap.  Define the
recentered potential
\[
  \varphi_t^\tau
  :=K_1*g_t^\tau
     +\tau\log\frac{\rho_t^\tau}{\rho_\infty^\tau}.
\]
so that equilibrium identity~\eqref{eq:stationary-euler-lagrange} implies
$\nabla\varphi_t^\tau=u_t^\tau$.  Pairing this potential with
$g_t^\tau$ produces the sum of the two relative entropies:
\begin{align}
  \int_{\T^d}g_t^\tau\varphi_t^\tau\,dx
  &=\|g_t^\tau\|_{\dot H^{-1}}^2
    +\tau\bigl[
       \operatorname{Ent}(\rho_t^\tau\mid\rho_\infty^\tau)
       +\operatorname{Ent}(\rho_\infty^\tau\mid\rho_t^\tau)
     \bigr]
   \ge\widetilde{\mathcal F}_\tau(t).
  \label{eq:potential-gap-pairing}
\end{align}
Since $g_t^\tau$ has zero mean, $\dot H^{-1}$--$\dot H^1$
duality applies to this pairing after subtracting the spatial mean
of $\varphi_t^\tau$.  Therefore
\begin{equation}
  \widetilde{\mathcal F}_\tau(t)
  \le\|g_t^\tau\|_{\dot H^{-1}}
       \|\nabla\varphi_t^\tau\|_{L^2}
  =\|g_t^\tau\|_{\dot H^{-1}}\|u_t^\tau\|_{L^2}.
  \label{eq:potential-duality-bound}
\end{equation}
Combining this with~\eqref{eq:appD-gap-coercivity}, and using
\eqref{eq:maximum-principle} in the dissipation identity, gives
\begin{equation}
  \int_{\T^d}\rho_t^\tau|u_t^\tau|^2dx
  \ge c_*\|u_t^\tau\|_{L^2}^2
  \ge c_*a_\tau\widetilde{\mathcal F}_\tau(t).
  \label{eq:equilibrium-dissipation-controls-gap}
\end{equation}
Thus $\widetilde{\mathcal F}_\tau'(t)
\le-c_*a_\tau\widetilde{\mathcal F}_\tau(t)$.  Integration proves
\eqref{eq:appD-gap-accelerated-relaxation}. Finally, we check that the initial factor in the Coulomb-distance
bound does not grow with $\tau$.  The minimum free energy is
nonnegative since both the Coulomb energy and
$\int\rho\log\rho$ are nonnegative on the unit torus.  Hence
\begin{equation}
  \frac{\widetilde{\mathcal F}_\tau(0)}{a_\tau}
  \le\frac{\mathcal F_\tau(\rho_0)}{a_\tau}
  \le\|\rho_0-\pi\|_{\dot H^{-1}}^2
    +\kappa_*^{-1}\int_{\T^d}\rho_0\log\rho_0\,dx.
  \label{eq:uniform-initial-gap-factor}
\end{equation}
The right side depends only on the initial and target densities.
Using~\eqref{eq:appD-gap-coercivity} and
\eqref{eq:appD-gap-accelerated-relaxation} now yields
\[
  \|g_t^\tau\|_{\dot H^{-1}}^2
  \le\frac{\widetilde{\mathcal F}_\tau(t)}{a_\tau}
  \le C e^{-c_*a_\tau t}.
\]
Taking square roots proves the first density estimate.  The second
follows from $a_\tau\ge\frac12$.
\end{proof}

\subsection{Low temperature and the deterministic endpoint}
\label{subsec:low-temperature}

We now turn the common $\dot H^{-1}$ relaxation of
Proposition~\ref{prop:hminus-exp} into a bound on the velocity gradient
when the temperature is small.  Write
\[
  \widetilde\rho_t^\tau:=\rho_t^\tau-\rho_\infty^\tau,
  \qquad \rho_\infty^0:=\pi.
\]
The equilibrium centered identity
\eqref{eq:equilibrium-centered-deformation} separates the quantity we
must control into a Coulomb term and a thermal term:
\begin{equation*}
  \|\nabla u_t^\tau\|_{L^\infty}
  \le
  \|\nabla^2K_1*\widetilde\rho_t^\tau\|_{L^\infty}
  +\tau\|\nabla^2\log\rho_t^\tau
          -\nabla^2\log\rho_\infty^\tau\|_{L^\infty}.
\end{equation*}
By Lemma~\ref{lem:periodic-schauder}, the first term is bounded by
$C\|\widetilde\rho_t^\tau\|_{C^{0,1/2}}$.  Once we have a uniform
$C^2$ bound, the positive density bounds and the identity
\[
  \nabla^2\log\rho
  =\rho^{-1}\nabla^2\rho
   -\rho^{-2}\nabla\rho\otimes\nabla\rho
\]
bound the second term by
$C\tau\|\widetilde\rho_t^\tau\|_{C^2}$.  We will therefore obtain
$L^\infty$ decay, a uniform first derivative bound, and then $C^2$
decay.  Proposition~\ref{prop:tau0-endpoint} will pass these estimates
to $\tau=0$.

Fix $\tau_0>0$, to be reduced at the absorption steps below.  Its final
value and all constants in this subsection depend only on the data in
Assumption~\ref{assump:standing-data}.  For $0<\tau\le\tau_0$, set
\[
  M_r^\tau(t)
  :=\max_{|\beta|=r}
       \|\partial^\beta\widetilde\rho_t^\tau\|_{L^\infty},
  \qquad r=0,1,2.
\]
Lemma~\ref{lem:hamilton} makes these maximum norms locally Lipschitz
in time.  At an active spatial maximum, the transport term vanishes
and diffusion has a favorable sign.  We use $D^+M(t)$ for the upper
right Dini derivative of a function $M$.

To see the source of damping, subtract the stationary equation from
\eqref{eq:intro-mean-field-pde}.  By
\eqref{eq:stationary-euler-lagrange}, its Coulomb drift is
$-\tau\nabla\log\rho_\infty^\tau$.  Set
\[
  b_t^\tau:=\nabla K_1*\widetilde\rho_t^\tau
            -\tau\nabla\log\rho_\infty^\tau.
\]
The density difference then solves
\begin{align}
  \partial_t\widetilde\rho^\tau
  &={}
  \tau\Delta\widetilde\rho^\tau
  +b_t^\tau\cdot\nabla\widetilde\rho^\tau
  +\nabla\rho_\infty^\tau\cdot
    \nabla K_1*\widetilde\rho^\tau
  \notag\\
  &\quad
  -\rho^\tau\widetilde\rho^\tau
  -(\rho_\infty^\tau-\pi)\widetilde\rho^\tau.
  \label{eq:equilibrium-error-equation}
\end{align}
Here $-\rho^\tau\widetilde\rho^\tau$ supplies damping since
$\rho^\tau\ge c_*$ by \eqref{eq:maximum-principle}.  We first show
that the other zeroth order coefficient,
$\rho_\infty^\tau-\pi$, is small.

The following expansion also identifies the thermal bias.  It is the
periodic counterpart of the first order bulk expansion in
\cite[Eqs.~(1.28)--(1.30) and Proposition~6.1]{ArmstrongSerfaty2022}.
\begin{proposition}[Low temperature thermal bias]
\label{prop:stat-bias}
There exists $C<\infty$ such that, for every $\tau\ge0$,
\begin{align}
  \|\rho_\infty^\tau-\pi\|_{C^2}
  &\le C\tau,
  \label{eq:stat-bias-linear}\\
  \|\rho_\infty^\tau-\pi-\tau\Delta\log\pi\|_{L^\infty}
  &\le C\tau^2.
  \label{eq:stat-bias-expansion}
\end{align}
\end{proposition}

\begin{proof}
The claim at $\tau=0$ follows from $\rho_\infty^0=\pi$.  For
$\tau>0$, we first express the bias using the stationary equation.
Taking the Laplacian of \eqref{eq:stationary-euler-lagrange} and using
$-\Delta K_1*f=f$ for mean zero $f$ gives
\begin{equation}
  \rho_\infty^\tau-\pi
  =\tau\Delta\log\rho_\infty^\tau.
  \label{eq:stationary-poisson-boltzmann}
\end{equation}
Proposition~\ref{prop:stationary-all-orders} bounds
$\log\rho_\infty^\tau$ in $C^4$, uniformly in $\tau$.  Equation
\eqref{eq:stationary-poisson-boltzmann} therefore gives
\[
  \|\rho_\infty^\tau-\pi\|_{C^2}
  \le\tau\|\log\rho_\infty^\tau\|_{C^4}
  \le C\tau.
\]

To obtain the first order expansion, we must replace
$\log\rho_\infty^\tau$ in
\eqref{eq:stationary-poisson-boltzmann} by $\log\pi$.  The identity
\[
  \log\rho_\infty^\tau-\log\pi
  =(\rho_\infty^\tau-\pi)
    \int_0^1
    \frac{ds}{\pi+s(\rho_\infty^\tau-\pi)}
\]
shows that this replacement costs $O(\tau)$ in $C^2$.  Indeed, the
denominator is at least $c_\pi$, while
\eqref{eq:uniform-stationary-regularity} bounds its derivatives
through order two.  Substituting into
\eqref{eq:stationary-poisson-boltzmann} now yields
\[
  \|\rho_\infty^\tau-\pi-\tau\Delta\log\pi\|_{L^\infty}
  =\tau\|\Delta(\log\rho_\infty^\tau-\log\pi)\|_{L^\infty}
  \le C\tau^2.
\]
\end{proof}

We next use the $O(\tau)$ bias to isolate the only term that cannot
yet be absorbed, the Coulomb field
$\nabla K_1*\widetilde\rho_t^\tau$.
\begin{lemma}[$L^\infty$ differential inequality]
\label{lem:linfty-differential-inequality}
For every $\tau>0$ and $t\ge0$, there is a constant $C$ independent
of $\tau$ and $t$ such that
\begin{equation}
  D^+\|\widetilde\rho_t^\tau\|_{L^\infty}
  +(c_*-C\tau)\|\widetilde\rho_t^\tau\|_{L^\infty}
  \le C\|\nabla K_1*\widetilde\rho_t^\tau\|_{L^\infty}.
  \label{eq:linfty-differential-inequality}
\end{equation}
\end{lemma}

\begin{proof}
We evaluate the error equation
\eqref{eq:equilibrium-error-equation} where
$|\widetilde\rho_t^\tau|$ is largest.  By
Lemma~\ref{lem:hamilton}, choose a sign $\sigma\in\{-1,1\}$
and a maximizing point $x_t$ so that
\[
  \sigma\widetilde\rho_t^\tau(x_t)
  =\|\widetilde\rho_t^\tau\|_{L^\infty},
  \quad
  \nabla\widetilde\rho_t^\tau(x_t)=0,
  \quad
  \sigma\Delta\widetilde\rho_t^\tau(x_t)\le0.
\]
The first two terms on the right of
\eqref{eq:equilibrium-error-equation} are therefore nonpositive or
zero after multiplication by $\sigma$.  The remaining terms give
\[
  D^+\|\widetilde\rho_t^\tau\|_{L^\infty}
  \le
  \bigl(\|\rho_\infty^\tau-\pi\|_{L^\infty}-c_*\bigr)
       \|\widetilde\rho_t^\tau\|_{L^\infty}
  +\|\nabla\rho_\infty^\tau\|_{L^\infty}
       \|\nabla K_1*\widetilde\rho_t^\tau\|_{L^\infty}.
\]
The density lower bound is \eqref{eq:maximum-principle}.
Proposition~\ref{prop:stationary-all-orders} bounds the stationary
gradient, and \eqref{eq:stat-bias-linear} bounds the bias by $C\tau$.
These estimates give \eqref{eq:linfty-differential-inequality}.
\end{proof}

A direct convolution bound would replace the Coulomb field in
\eqref{eq:linfty-differential-inequality} by a fixed multiple of
$\|\widetilde\rho_t^\tau\|_{L^\infty}$.  That multiple need not be
smaller than $c_*$.  The following estimate instead uses the weak
decay from Proposition~\ref{prop:hminus-exp} to make the Coulomb
field a small multiple of the maximum norm plus a decaying term.
Its proof in Subsection~\ref{subsec:appC-coulomb-field-estimates}
adapts the near and far splitting of
\cite[proof of Lemma~A.4(i)]{ChizatColomboColomboFernandezReal2026}.

\begin{lemma}[Near and far Coulomb field estimate]
\label{lem:potential-estimate}
For every mean zero
$f\in L^\infty(\T^d)\cap\dot H^{-1}(\T^d)$,
\begin{equation}
  \|\nabla K_1*f\|_{L^\infty}
  \le
  C\|f\|_{L^\infty}^{\frac d{d+2}}
    \|f\|_{\dot H^{-1}}^{\frac2{d+2}}
  +C\|f\|_{\dot H^{-1}}.
  \label{eq:potential-estimate}
\end{equation}
For $d=2$, the first term is
$C(\|f\|_{L^\infty}\|f\|_{\dot H^{-1}})^{1/2}$.
\end{lemma}

\begin{proof}
See Subsection~\ref{subsec:appC-coulomb-field-estimates}.  There the
kernel is split at a radius $r$, the near part is bounded by
$Cr\|f\|_{L^\infty}$, and the far part is bounded by
$C(1+r^{-d/2})\|f\|_{\dot H^{-1}}$.  Optimizing $r$ gives
\eqref{eq:potential-estimate}.
\end{proof}

We can now absorb the field term in the maximum norm inequality.  This
is the first place where $\tau$ must be restricted to a small range.
\begin{proposition}[Uniform $L^\infty$ relaxation]
\label{prop:linfty-exp}
Set $\gamma_1:=c_*/8$.  There exists $C<\infty$ such that, for
$0<\tau\le\tau_0$ and $t\ge0$,
\begin{equation}
  \|\rho_t^\tau-\rho_\infty^\tau\|_{L^\infty}
  \le Ce^{-\gamma_1t}.
  \label{eq:linfty-relaxation}
\end{equation}
\end{proposition}

\begin{proof}
Our goal is to close \eqref{eq:linfty-differential-inequality}
using Lemma~\ref{lem:potential-estimate}.  Reduce $\tau_0$ until
$C\tau_0\le c_*/2$ in that differential inequality.  With
$A(t):=\|\widetilde\rho_t^\tau\|_{L^\infty}$ and
$B(t):=\|\widetilde\rho_t^\tau\|_{\dot H^{-1}}$, we obtain
\[
  D^+A(t)+\frac{c_*}{2}A(t)
  \le CA(t)^{\frac d{d+2}}B(t)^{\frac2{d+2}}+CB(t).
\]
Young's inequality, with a coefficient chosen below $c_*/4$,
absorbs the first term on the right into the damping and a multiple
of $B(t)$.  The common relaxation estimate
\eqref{eq:hminus-relaxation} then yields
\[
  D^+A(t)+\frac{c_*}{4}A(t)
  \le CB(t)\le Ce^{-(c_*/4)t}.
\]
Proposition~\ref{prop:stationary-all-orders} and the standing
assumption on $\rho_0$ bound $A(0)$ independently of $\tau$.
Integrating the scalar inequality gives
\[
  A(t)
  \le e^{-(c_*/4)t}A(0)+Ct e^{-(c_*/4)t}
  \le Ce^{-(c_*/8)t},
\]
which is \eqref{eq:linfty-relaxation}.
\end{proof}

To estimate a first derivative of the density difference, we must
control $\nabla^2K_1*\widetilde\rho_t^\tau$.  This operator does not
map $L^\infty$ to $L^\infty$, so the preceding decay cannot bound it
directly.  The next lemma pays only a logarithm of the first
derivative.  It is a $W^{1,\infty}$ formulation of
\cite[Lemma~A.5(i)]{ChizatColomboColomboFernandezReal2026}.

\begin{lemma}[Logarithmic Newtonian estimate]
\label{lem:yudovich}
If $f\in W^{1,\infty}(\T^d)$ has mean zero and is not identically
zero, then
\begin{equation}
  \|\nabla^2K_1*f\|_{L^\infty}
  \le C\|f\|_{L^\infty}
    \left[1+\log_+\left(
      \frac{\|\nabla f\|_{L^\infty}}{\|f\|_{L^\infty}}
    \right)\right].
  \label{eq:yudovich-newtonian-estimate}
\end{equation}
Both sides vanish when $f\equiv0$.
\end{lemma}

\begin{proof}
See Subsection~\ref{subsec:appC-coulomb-field-estimates}
for the fixed-scale Hessian estimate and the choice of cutoff radius.
\end{proof}

The logarithm allows the differentiated maximum principle to close
without the positive temperature estimate
\eqref{eq:positive-temperature-W4}, whose constant can diverge as
$\tau\downarrow0$.
\begin{proposition}[Uniform Lipschitz bound]
\label{prop:uniform-lipschitz}
There exists $C<\infty$ such that
\begin{equation}
  \sup_{0<\tau\le\tau_0}\sup_{t\ge0}
  \|\rho_t^\tau\|_{W^{1,\infty}}
  \le C.
  \label{eq:uniform-lipschitz-bound}
\end{equation}
\end{proposition}

\begin{proof}
We estimate $M_1^\tau(t)$ from the equation for
$\widetilde\rho^\tau$.  Differentiating
\eqref{eq:equilibrium-error-equation} once leaves
$-\rho^\tau\partial_i\widetilde\rho^\tau$ as its principal damping
term.  Reduce $\tau_0$ so that $C\tau_0\le c_*/2$.
Applying Lemma~\ref{lem:hamilton} to each signed derivative and
Lemma~\ref{lem:yudovich} to the Coulomb Hessian, and using
$M_0^\tau(t)\le Ce^{-\gamma_1t}$ from
Proposition~\ref{prop:linfty-exp}, gives
\begin{equation}
  D^+M_1^\tau(t)
  \le C(1+t)e^{-\gamma_1t}
       (1+M_1^\tau(t))\log(e+M_1^\tau(t)).
  \label{eq:first-derivative-osgood}
\end{equation}
The differentiated terms and the logarithmic coefficient are estimated
in Subsection~\ref{subsec:appC-first-order}.

The coefficient on the right is integrable over $[0,\infty)$.
Since $M_1^\tau$ is locally Lipschitz, division by
$(e+M_1^\tau)\log(e+M_1^\tau)$ at almost every time gives
\[
  \frac{d}{dt}\log\log(e+M_1^\tau(t))
  \le C(1+t)e^{-\gamma_1t}.
\]
The initial value is uniformly bounded by the $C^1$ bounds on
$\rho_0$ and $\rho_\infty^\tau$ from
Proposition~\ref{prop:stationary-all-orders}.  Integration bounds
$M_1^\tau(t)$ uniformly, and
$\rho_t^\tau=\widetilde\rho_t^\tau+\rho_\infty^\tau$ gives
\eqref{eq:uniform-lipschitz-bound}.
\end{proof}

The $L^\infty$ decay now has a uniform bound on the first derivative
with which to interpolate.  This supplies the H\"older norm required
by the periodic Schauder estimate.
\begin{corollary}[H\"older and Coulomb field decay]
\label{cor:holder-decay}
Set $\gamma_2:=\gamma_1/2=c_*/16$.  There exists $C<\infty$ such
that, for $0<\tau\le\tau_0$ and $t\ge0$,
\begin{equation}
  \|\widetilde\rho_t^\tau\|_{C^{0,1/2}}
  \le Ce^{-\gamma_2t}.
  \label{eq:holder-relaxation}
\end{equation}
Consequently,
\begin{equation}
  \|\nabla K_1*\widetilde\rho_t^\tau\|_{L^\infty}
  +\|\nabla^2K_1*\widetilde\rho_t^\tau\|_{L^\infty}
  \le Ce^{-\gamma_2t}.
  \label{eq:coulomb-field-relaxation}
\end{equation}
\end{corollary}

\begin{proof}
We first interpolate the decay in
\eqref{eq:linfty-relaxation} with the derivative bound in
\eqref{eq:uniform-lipschitz-bound}.  Lemma~\ref{lem:appC-holder-interpolation}
with exponent $1/2$ gives
\[
  [\widetilde\rho_t^\tau]_{C^{0,1/2}}
  \le C\|\widetilde\rho_t^\tau\|_{L^\infty}^{1/2}
       \|\nabla\widetilde\rho_t^\tau\|_{L^\infty}^{1/2}
  \le Ce^{-\gamma_2t}.
\]
The $L^\infty$ term in the H\"older norm decays at the faster
rate $\gamma_1$, proving \eqref{eq:holder-relaxation}.  Next,
$\nabla K_1\in L^1(\T^d)$ bounds the Coulomb field by the
$L^\infty$ norm of $\widetilde\rho_t^\tau$.  Since this density
difference has mean zero,
\[
  \|\nabla K_1*\widetilde\rho_t^\tau\|_{L^\infty}
  \le C\|\widetilde\rho_t^\tau\|_{L^\infty},
  \qquad
  -\Delta(K_1*\widetilde\rho_t^\tau)
  =\widetilde\rho_t^\tau.
\]
Applying Lemma~\ref{lem:periodic-schauder} to the second identity
and using \eqref{eq:holder-relaxation} proves
\eqref{eq:coulomb-field-relaxation}.
\end{proof}

The estimates above control the density error
$\widetilde\rho_t^\tau=\rho_t^\tau-\rho_\infty^\tau$
in $C^{0,1/2}$.  By Corollary~\ref{cor:holder-decay}, this already
controls the Coulomb part of the velocity gradient in
\eqref{eq:equilibrium-centered-deformation}.  To control its thermal
part, we next prove decay of two derivatives of the density error.

\begin{proposition}[Low-temperature relaxation]
\label{prop:low-temperature-relaxation}
\label{prop:C2-decay}
Let $\gamma_2=c_*/16$ be as in Corollary~\ref{cor:holder-decay}.
There is a constant $C<\infty$ such that, for
$0<\tau\le\tau_0$ and $t\ge0$,
\begin{equation}
  \|\rho_t^\tau-\rho_\infty^\tau\|_{C^2}
  +\left\|\nabla^2\log\rho_t^\tau
             -\nabla^2\log\rho_\infty^\tau\right\|_{L^\infty}
  \le Ce^{-\gamma_2t}.
  \label{eq:C2-relaxation}
\end{equation}
In particular, the velocity $u_t^\tau$ defined in
\eqref{eq:free-energy-velocity} satisfies
\begin{equation}
  \|\nabla u_t^\tau\|_{L^\infty}
  \le Ce^{-\gamma_2t}.
  \label{eq:low-temperature-velocity-gradient}
\end{equation}
\end{proposition}

\begin{proof}
We estimate $M_2^\tau(t)$, the maximum of the $L^\infty$ norms of
the second derivatives of $\widetilde\rho_t^\tau$, as defined at the
start of this subsection.  Differentiating the error equation
\eqref{eq:equilibrium-error-equation} twice leaves the damping
$-\rho_t^\tau\partial_{ij}\widetilde\rho_t^\tau$.  The bound
$\rho_t^\tau\ge c_*$ from Proposition~\ref{prop:maximum-principle}
therefore supplies the term $-c_*M_2^\tau$ in the maximum-principle
estimate.  The remaining task is to bound the differentiated Coulomb
field and the products of first derivatives.

The most singular term contains two derivatives of
$\nabla K_1*\widetilde\rho_t^\tau$.  We move one derivative onto the
density error, then apply the periodic Schauder estimate of
Lemma~\ref{lem:periodic-schauder} and the interpolation estimate of
Lemma~\ref{lem:appC-ehrling}.  For every $\varepsilon>0$, they give
\[
\begin{aligned}
  \partial_{ij}(\nabla K_1*\widetilde\rho_t^\tau)
  &=\partial_i\nabla K_1*\partial_j\widetilde\rho_t^\tau,\\
  \|\partial_{ij}(\nabla K_1*\widetilde\rho_t^\tau)\|_{L^\infty}
  &\le\varepsilon M_2^\tau(t)
       +C_\varepsilon M_0^\tau(t).
\end{aligned}
\]
The first identity holds in distributions, which is enough for the
Schauder estimate.  The terms with fewer derivatives of the Coulomb
potential are controlled by the decay in
\eqref{eq:coulomb-field-relaxation}.

The other terms contain products of first derivatives of the density
and the error.  Proposition~\ref{prop:uniform-lipschitz} bounds
$\nabla\rho_t^\tau$ uniformly.  The Landau--Kolmogorov inequality
of Lemma~\ref{lem:appC-landau-kolmogorov} and Young's inequality then
give, for example,
\[
\begin{aligned}
  \|(\partial_i\rho_t^\tau)
       (\partial_j\widetilde\rho_t^\tau)\|_{L^\infty}
  &\le\varepsilon M_2^\tau(t)+C_\varepsilon M_0^\tau(t).
\end{aligned}
\]

Apply Lemma~\ref{lem:hamilton} to each second derivative.  The term
with an explicit factor $\tau$ costs at most
$C\tau M_2^\tau$.  The decay in
\eqref{eq:holder-relaxation} controls the remaining coefficient
of $M_2^\tau$.  The detailed calculation in
Subsection~\ref{subsec:appC-second-order} yields
\[
  D^+M_2^\tau(t)
  \le-\bigl(c_*-C\varepsilon-C\tau_0\bigr)M_2^\tau(t)
       +Ce^{-\gamma_2t}M_2^\tau(t)
       +C_\varepsilon e^{-\gamma_2t}.
\]
Choose $\varepsilon$ small enough and then decrease $\tau_0$
so that the fixed losses use at most half the damping.  We obtain
\begin{equation}
  D^+M_2^\tau(t)
  \le\left(-\frac{c_*}{2}+Ce^{-\gamma_2t}\right)M_2^\tau(t)
       +Ce^{-\gamma_2t}.
  \label{eq:C2-differential-inequality}
\end{equation}

The initial value $M_2^\tau(0)$ is bounded uniformly in $\tau$
by Assumption~\ref{assump:standing-data} and the stationary estimate
\eqref{eq:uniform-stationary-regularity}.  Since $M_2^\tau$ is
locally Lipschitz, integration of \eqref{eq:C2-differential-inequality}
gives
\[
  M_2^\tau(t)\le Ce^{-\gamma_2t}.
\]
The bounds for $M_0^\tau$ in \eqref{eq:linfty-relaxation} and
the Landau--Kolmogorov inequality also give
\[
  M_1^\tau(t)
  \le C\bigl(M_0^\tau(t)M_2^\tau(t)\bigr)^{1/2}
  \le Ce^{-(\gamma_1+\gamma_2)t/2}.
\]
These three bounds prove the $C^2$ estimate in
\eqref{eq:C2-relaxation}.

It remains to control the logarithmic Hessian.  The density bounds in
\eqref{eq:maximum-principle}, the stationary regularity in
\eqref{eq:uniform-stationary-regularity}, and the estimate just proved
give a uniform $C^2$ bound for both densities.  Using
\[
  \nabla^2\log\rho
  =\rho^{-1}\nabla^2\rho
   -\rho^{-2}\nabla\rho\otimes\nabla\rho,
\]
we obtain
\[
  \|\nabla^2\log\rho_t^\tau
       -\nabla^2\log\rho_\infty^\tau\|_{L^\infty}
  \le C\|\widetilde\rho_t^\tau\|_{C^2}
  \le Ce^{-\gamma_2t}.
\]
Finally, insert this estimate and the Coulomb-Hessian estimate
\eqref{eq:coulomb-field-relaxation} into the recentered identity
\eqref{eq:equilibrium-centered-deformation}.  This gives
\eqref{eq:low-temperature-velocity-gradient}.
\end{proof}

The estimate above is uniform as $\tau\downarrow0$.  At
$\tau=0$, however, the solution is initially known only as a bounded
weak solution, so its equation cannot yet be differentiated as in the
preceding proof.  We now use compactness to identify it as the limit
of the positive-temperature solutions and transfer the estimates.

\begin{proposition}[Deterministic endpoint]
\label{prop:tau0-endpoint}
Let $\rho^0$ be the unique bounded global weak solution of
\eqref{eq:intro-mean-field-pde} at $\tau=0$ with initial density
$\rho_0$.  Then
\begin{equation}
  \rho^\tau\longrightarrow\rho^0
  \quad\text{in }C_{\mathrm{loc}}
    \bigl([0,\infty);C^1(\T^d)\bigr)
  \quad\text{as }\tau\downarrow0.
  \label{eq:tau-zero-strong-convergence}
\end{equation}
With $\rho_\infty^0=\pi$, the limit satisfies, for all $t\ge0$,
\begin{align}
  \|\rho_t^0-\pi\|_{W^{2,\infty}}
  &\le Ce^{-\gamma_2t},
  \label{eq:tau-zero-W2inf-relaxation}\\
  \|\nabla u_t^0\|_{L^\infty}
  &\le Ce^{-\gamma_2t}.
  \label{eq:tau-zero-velocity-gradient}
\end{align}
\end{proposition}

\begin{proof}
\medskip
\noindent\emph{Step 1.}
We first establish compactness in
$C_{\mathrm{loc}}([0,\infty);C^1(\T^d))$.
The positive-temperature estimates
\eqref{eq:C2-relaxation} and
\eqref{eq:uniform-stationary-regularity}, together with the
mean-field equation and $\nabla K_1\in L^1(\T^d)$, give
\[
  \sup_{0<\tau\le\tau_0}\sup_{t\ge0}
  \bigl(\|\rho_t^\tau\|_{W^{2,\infty}}
       +\|\partial_t\rho_t^\tau\|_{L^\infty}\bigr)\le C.
\]
The Landau--Kolmogorov inequality of
Lemma~\ref{lem:appC-landau-kolmogorov} and the time-derivative bound give
\[
  \|\rho_t^\tau-\rho_s^\tau\|_{C^1}
  \le C|t-s|^{1/2}.
\]
The compact embedding $W^{2,\infty}(\T^d)\Subset C^1(\T^d)$
and the Arzel\`a--Ascoli theorem now give a subsequence
$\tau_n\downarrow0$ and a limit $\bar\rho$ such that
$\rho^{\tau_n}\longrightarrow\bar\rho$ in
$C_{\mathrm{loc}}([0,\infty);C^1(\T^d))$.

\medskip
\noindent\emph{Step 2.}
We show that $\bar\rho$ solves \eqref{eq:intro-mean-field-pde} at
$\tau=0$ with initial density $\rho_0$ by letting $\tau_n\to0$ along the subsequence from Step 1.
The uniform
second-derivative bound gives
$\tau_n\Delta\rho^{\tau_n}\to0$ uniformly on compact time
intervals.  Strong convergence of the densities and
$\nabla K_1\in L^1$ imply that
$\rho^{\tau_n}\nabla K_1*(\rho^{\tau_n}-\pi)$ converges uniformly
on compact time intervals.
Thus $\bar\rho$ is a bounded weak solution of the zero-temperature
equation with initial density $\rho_0$.
Uniqueness of bounded weak solutions
\cite[Proposition~2.4]{ChodronRosenzweig2026} gives
$\bar\rho=\rho^0$.  Every subsequence has the same limit, which
proves \eqref{eq:tau-zero-strong-convergence}.

\medskip
\noindent\emph{Step 3.}
We next show that the limit inherits the uniform decay bounds
obtained at positive temperature.  Proposition~\ref{prop:stat-bias} gives
$\rho_\infty^{\tau_n}\to\pi$ in $C^2$, so
$\widetilde\rho_t^{\tau_n}\to\rho_t^0-\pi$ in $C^1$ for each
$t$.  The bounds $c_*\le\rho_t^{\tau_n}\le C_*$ in
\eqref{eq:maximum-principle} also pass to $\rho_t^0$.
Strong convergence passes the zeroth- and first-order bounds
in \eqref{eq:C2-relaxation} to the limit.  For the second derivatives,
the uniform $L^\infty$ bound and distributional convergence give weak-*
convergence in $L^\infty$.  Lower semicontinuity therefore yields
\[
  \|\nabla^2(\rho_t^0-\pi)\|_{L^\infty}
  \le Ce^{-\gamma_2t}.
\]
These bounds prove \eqref{eq:tau-zero-W2inf-relaxation} for every $t\ge0$.

It remains to estimate $\nabla u_t^0$.  The strong $C^1$
convergence above gives convergence of
$\widetilde\rho_t^{\tau_n}$ in $C^{0,1/2}$, by H\"older
interpolation.  Lemma~\ref{lem:periodic-schauder} therefore passes
\eqref{eq:coulomb-field-relaxation} to the limit.  Since
$\nabla u_t^0=\nabla^2K_1*(\rho_t^0-\pi)$, this proves
\eqref{eq:tau-zero-velocity-gradient}.  Further details of the
limit passage are given in Subsection~\ref{subsec:zero-temperature-limit}.
\end{proof}

The endpoint argument also completes the weak relaxation estimate for
the full temperature range.  Indeed,
\eqref{eq:hminus-relaxation} covers $\tau>0$, while
\eqref{eq:tau-zero-strong-convergence} and
\eqref{eq:stat-bias-linear} permit its limit as $\tau\downarrow0$.
With $\rho_\infty^0=\pi$, we have
\begin{equation}
  \|\rho_t^\tau-\rho_\infty^\tau\|_{\dot H^{-1}}
  \le Ce^{-c_*t/4},
  \qquad \tau\ge0,\quad t\ge0.
  \label{eq:hminus-relaxation-all-temperatures}
\end{equation}
The compactness argument uses Assumption~\ref{assump:standing-data}
in every dimension $d\ge2$.  A direct use of
\cite[Theorem~1.2(iii)]{ChizatColomboColomboFernandezReal2026}
with Sobolev embedding instead asks for
$\rho_0\in H^\gamma$ and $\pi\in H^{\gamma+1}$ with
$\gamma>d/2$.  The stated $C^4$ assumption ensures these
conditions when $d\le5$.

\subsection{Temperatures bounded away from zero}
\label{subsec:high-temperature-estimates}

The estimates of Subsection~\ref{subsec:low-temperature} control
$\nabla u_t^\tau$ for $0\le\tau\le\tau_0$.  To complete
Theorem~\ref{thm:uniform-envelope}, we must control the remaining
temperatures with a time integral that does not grow with $\tau$.
For every fixed lower bound $\tau_{\mathrm L}>0$,
Lemma~\ref{lem:positive-temperature-wp} gives a uniform
$W^{4,\infty}$ bound on the evolving densities.  We combine it
with the accelerated $\dot H^{-1}$ decay of
Proposition~\ref{prop:hminus-exp} to obtain decay in a norm that
controls both terms of the recentered velocity gradient
\eqref{eq:equilibrium-centered-deformation}.

\begin{proposition}[Positive-temperature half-line estimate]
\label{prop:positive-temperature-halfline}
\label{prop:appD-positive-halfline-decay}
For every $\tau_{\mathrm L}>0$, there are constants
$C_{\tau_{\mathrm L}}<\infty$ and $c>0$ such that, for
$\tau\ge\tau_{\mathrm L}$ and $t\ge0$,
\begin{align}
  \|\widetilde\rho_t^\tau\|_{C^{2,1/2}}
  &\le C_{\tau_{\mathrm L}}e^{-c(1+\tau)t},
  \label{eq:appD-C2half-accelerated-relaxation}\\
  \|\nabla u_t^\tau\|_{L^\infty}
  &\le C_{\tau_{\mathrm L}}(1+\tau)e^{-c(1+\tau)t}.
  \label{eq:appD-velocity-gradient-accelerated-relaxation}
\end{align}
Here $\widetilde\rho_t^\tau=\rho_t^\tau-\rho_\infty^\tau$,
as defined in Subsection~\ref{subsec:low-temperature}.
\end{proposition}

\begin{proof}
Since $\widetilde\rho_t^\tau$ has mean zero, the interpolation estimate of
Lemma~\ref{lem:appD-weak-strong-interpolation}, proved in
Subsection~\ref{subsec:appD-interpolation}, gives
\[
  \|\widetilde\rho_t^\tau\|_{C^{2,1/2}}
  \le C_d
  \|\widetilde\rho_t^\tau\|_{\dot H^{-1}}^{\theta}
  \|\widetilde\rho_t^\tau\|_{W^{4,\infty}}^{1-\theta},
  \qquad \theta:=\frac{3}{d+10}.
\]
By \eqref{eq:positive-temperature-W4} and
\eqref{eq:uniform-stationary-regularity},
$\|\widetilde\rho_t^\tau\|_{W^{4,\infty}}\le C_{\tau_{\mathrm L}}$
for $\tau\ge\tau_{\mathrm L}$ and $t\ge0$.
Substituting this bound and
\eqref{eq:appD-Hminus-accelerated-relaxation} into the interpolation
inequality yields
\[
\begin{aligned}
  \|\widetilde\rho_t^\tau\|_{C^{2,1/2}}
  &\le C_{\tau_{\mathrm L}}
  \exp\!\left[-\frac{3c_*}{2(d+10)}
    \left(\frac12+\kappa_*\tau\right)t\right]\\
  &\le C_{\tau_{\mathrm L}}e^{-c(1+\tau)t}.
\end{aligned}
\]
Here
$\frac12+\kappa_*\tau\ge
\min\{\frac12,\kappa_*\}(1+\tau)$, so $c>0$ is independent of
$\tau_{\mathrm L}$.  This proves
\eqref{eq:appD-C2half-accelerated-relaxation}.

The periodic Schauder estimate of Lemma~\ref{lem:periodic-schauder},
the density lower bound \eqref{eq:maximum-principle}, and the uniform
derivative bounds give
\[
\begin{aligned}
  \|\nabla^2K_1*\widetilde\rho_t^\tau\|_{L^\infty}
  &\le C_d\|\widetilde\rho_t^\tau\|_{C^{0,1/2}},\\
  \|\nabla^2\log\rho_t^\tau
       -\nabla^2\log\rho_\infty^\tau\|_{L^\infty}
  &\le C_{\tau_{\mathrm L}}
       \|\widetilde\rho_t^\tau\|_{C^2}.
\end{aligned}
\]
Inserting both bounds into
\eqref{eq:equilibrium-centered-deformation} proves
\[
  \|\nabla u_t^\tau\|_{L^\infty}
  \le C_{\tau_{\mathrm L}}(1+\tau)
       \|\widetilde\rho_t^\tau\|_{C^{2,1/2}}.
\]
The first estimate of the proposition now gives
\eqref{eq:appD-velocity-gradient-accelerated-relaxation}.
\end{proof}

The preceding proposition covers $\tau\ge\tau_0$, while
Propositions~\ref{prop:low-temperature-relaxation}
and~\ref{prop:tau0-endpoint} cover $0\le\tau\le\tau_0$.  We
combine their bounds to finish the section.

\begin{proof}[Proof of Theorem~\ref{thm:uniform-envelope}]
Our goal is to use one constant in the velocity estimate for every
$\tau\ge0$.  The low-temperature estimates
\eqref{eq:low-temperature-velocity-gradient} and
\eqref{eq:tau-zero-velocity-gradient} give
\[
  \|\nabla u_t^\tau\|_{L^\infty}
  \le C e^{-(c_*/16)t},
  \qquad 0\le\tau\le\tau_0.
\]
Apply Proposition~\ref{prop:positive-temperature-halfline} with
$\tau_{\mathrm L}=\tau_0$.  Since $\tau_0$ is fixed by the
data, we may enlarge $C$ to cover both ranges.  If
$c$ is the rate from that proposition, choose
\[
  \gamma:=\min\left\{
    \frac{c_*}{16(1+\tau_0)},c
  \right\}>0.
\]
For $0\le\tau\le\tau_0$ we have
$\gamma(1+\tau)\le c_*/16$, and the positive-temperature
half-line already has the desired rate.  Thus
\[
  \|\nabla u_t^\tau\|_{L^\infty}
  \le C(1+\tau)e^{-\gamma(1+\tau)t},
  \qquad \tau\ge0,\quad t\ge0.
\]
Integrating this estimate gives the bound required in
\eqref{eq:uniform-velocity-gradient-decay}:
\[
  \Psi_\infty
  =\sup_{\tau\ge0}\int_0^\infty
       \|\nabla u_t^\tau\|_{L^\infty}\,dt
  \le\frac{C}{\gamma}<\infty.
\]
\end{proof}

\section{Proofs of the main theorems and consequences}
\label{sec:uniform-consequences}

Proposition~\ref{prop:assembly} reduces the particle estimates to
the time integral of $\|\nabla u_t^\tau\|_{L^\infty}$.
Theorem~\ref{thm:uniform-envelope} bounds that integral by
$\Psi_\infty$, uniformly in time and temperature.  In this section
we turn the resulting modulated energy estimates into the two main
theorems.  Lemma~\ref{lem:distance-coercivity} converts energy into
distances between measures, and the relative entropy part of the
modulated free energy controls particle marginals.  We then combine
the particle estimate with mean-field relaxation and the stationary
bias to quantify approximation of the prescribed target.

\subsection{Proofs of the main theorems}

Our first goal is to control the distance between a particle
configuration and a bounded, strictly positive density by its modulated energy.
Lemma~\ref{lem:distance-coercivity} supplies this comparison.
We then apply it to the energy bounds of
Proposition~\ref{prop:assembly} to prove
Theorem~\ref{thm:uniform-mean-field}.  The entropy term in the same
bound yields Theorem~\ref{thm:uniform-propagation-chaos}.

\begin{lemma}[Negative Sobolev and Wasserstein coercivity]
\label{lem:distance-coercivity}
Let $N\ge2$, let $X_N\in\mathcal X_N$, and let $\rho$ be a
probability density with $0<c_*\le\rho\le C_*$.  With
$F_N$ and $\err$ defined in Section~\ref{sec:modulated-identity},
\begin{equation}
  \mathcal W_2(\mu_{X_N},\rho)^2
  \le C\left(\frac12F_N(X_N,\rho)+\err\right).
  \label{eq:wasserstein-coercivity}
\end{equation}
For every $\zeta>d/2+1$, it follows that
\begin{equation}
  \|\mu_{X_N}-\rho\|_{\dot H^{-\zeta}}^2
  \le C_\zeta\left(\frac12F_N(X_N,\rho)+\err\right).
  \label{eq:negative-sobolev-coercivity}
\end{equation}
\end{lemma}

\begin{proof}
We first prove \eqref{eq:wasserstein-coercivity}.  The atomic
measure $\mu_{X_N}$ has infinite Coulomb self-interaction, so we
replace each point mass by a measure supported on a small sphere.
Let $\delta_{x_i}^{(\eta)}$ be the uniform probability measure on
the sphere of radius $\eta$ around $x_i$, and set
\[
  \eta:=\frac14\bigl(N\|\rho\|_{L^\infty}\bigr)^{-1/d},
  \qquad
  \mu_{X_N,\eta}:=\frac1N\sum_{i=1}^N\delta_{x_i}^{(\eta)}.
\]
Coupling each point $x_i$ to its sphere gives
$\mathcal W_2(\mu_{X_N},\mu_{X_N,\eta})^2\le\eta^2$.
Since $1\le\|\rho\|_{L^\infty}\le C_*$, the definition
\eqref{eq:finite-particle-error} also gives $\eta^2\le C\err$.

The smeared measure has finite Coulomb energy.  The smearing
estimate proved in Appendix~\ref{app:distance-estimates} compares
that energy with the modulated energy after the diagonal has been
removed.  In the normalization of
Lemma~\ref{prop:static-coulomb-bounds}, it gives
\begin{equation}
  \|\mu_{X_N,\eta}-\rho\|_{\dot H^{-1}}^2
  \le C\left(\frac12F_N(X_N,\rho)+\err\right).
  \label{eq:smeared-energy-coercivity}
\end{equation}
Since $\rho\ge c_*>0$, Peyre's comparison
\cite[Corollary~2.3, Eq.~(2.12)]{Peyre2018} gives
\[
  \mathcal W_2(\mu_{X_N,\eta},\rho)
  \le\frac2{\sqrt{c_*}}
       \|\mu_{X_N,\eta}-\rho\|_{\dot H^{-1}}.
\]
The triangle inequality, followed by
\eqref{eq:smeared-energy-coercivity}, now yields
\[
\begin{aligned}
  \mathcal W_2(\mu_{X_N},\rho)^2
  &\le 2\mathcal W_2(\mu_{X_N},\mu_{X_N,\eta})^2
       +2\mathcal W_2(\mu_{X_N,\eta},\rho)^2\\
  &\le 2\eta^2
       +\frac8{c_*}\|\mu_{X_N,\eta}-\rho\|_{\dot H^{-1}}^2
   \le C\left(\frac12F_N(X_N,\rho)+\err\right).
\end{aligned}
\]
Here $\eta^2\le C\err\le2C(\frac12F_N+\err)$ by
\eqref{eq:modulated-energy-lower-bound}.

For the second claim, Sobolev embedding makes every test function
of unit $\dot H^\zeta$ norm Lipschitz when
$\zeta>d/2+1$, after subtracting its spatial mean.  Applying
Kantorovich--Rubinstein duality to the mean-zero signed measure
$\mu_{X_N}-\rho$ gives
\begin{equation}
  \|\mu_{X_N}-\rho\|_{\dot H^{-\zeta}}
  \le C_\zeta\mathcal W_1(\mu_{X_N},\rho)
  \le C_\zeta\mathcal W_2(\mu_{X_N},\rho).
  \label{eq:negative-sobolev-wasserstein-comparison}
\end{equation}
Combining this with \eqref{eq:wasserstein-coercivity} proves
\eqref{eq:negative-sobolev-coercivity}.
\end{proof}

We now apply the lemma to the particle law.  The estimate of
Theorem~\ref{thm:uniform-envelope} makes the exponential factor
in Proposition~\ref{prop:assembly} bounded for every $t$ and
$\tau$.  The i.i.d.\ initialization then leaves an error of order
$\err$ in both temperature regimes.

\begin{proof}[Proof of Theorem~\ref{thm:uniform-mean-field}]
We first bound the expected shifted modulated energy at positive
temperature and at the deterministic endpoint.  Lemma~\ref{lem:distance-coercivity}
will then give the asserted measure distances.

For $\tau>0$, combine the particle estimate
\eqref{eq:noisy-particle-gronwall}, the initial bound
\eqref{eq:iid-initial-free-energy}, and
$\int_0^t\|\nabla u_s^\tau\|_{L^\infty}\,ds\le\Psi_\infty$.
This gives
\begin{equation}
  \sup_{\tau>0}\operatorname*{ess\,sup}_{t\ge0}
  \bigl(\mathcal E_N^\tau(t)+\err\bigr)
  \le e^{(C_d/2)\Psi_\infty}\err
  \le C\err.
  \label{eq:uniform-noisy-modulated-free-energy}
\end{equation}
The relative entropy in $\mathcal E_N^\tau$ is nonnegative.
By its definition, for almost every $t\ge0$,
\begin{equation}
  \mathbb E_{f_N^\tau(t)}
    \left[\frac12F_N(X_N,\rho_t^\tau)+\err\right]
  \le\mathcal E_N^\tau(t)+\err
  \le C\err.
  \label{eq:uniform-noisy-modulated-energy}
\end{equation}

At $\tau=0$, the pathwise estimate
\eqref{eq:deterministic-particle-gronwall} applies to almost every
i.i.d.\ initial configuration.  The same initial calculation and
velocity bound give, after taking expectation over those initial
positions,
\begin{equation}
  \mathbb E\left[
    \frac12F_N(X_N^0(t),\rho_t^0)+\err\right]
  \le e^{(C_d/2)\Psi_\infty}\err
  \le C\err,
  \qquad t\ge0.
  \label{eq:uniform-deterministic-modulated-energy}
\end{equation}
Apply \eqref{eq:wasserstein-coercivity} to the configurations in
\eqref{eq:uniform-noisy-modulated-energy} and
\eqref{eq:uniform-deterministic-modulated-energy}, then take
expectations. 
This proves the theorem.
Applying \eqref{eq:negative-sobolev-coercivity} in the same way
gives the negative Sobolev estimate stated immediately after
the theorem.
\end{proof}

For $\tau>0$, the bound
\eqref{eq:uniform-noisy-modulated-free-energy} also controls the
relative entropy per particle.  Entropy subadditivity transfers
that control to each fixed marginal, giving the stronger notion of
chaos stated in Theorem~\ref{thm:uniform-propagation-chaos}.

\begin{proof}[Proof of Theorem~\ref{thm:uniform-propagation-chaos}]
We bound the entropy of the $k$-particle marginal with respect
to the product density $(\rho_t^\tau)^{\otimes k}$, then apply
the Csisz\'ar--Kullback--Pinsker inequality.  The lower bound
\eqref{eq:modulated-energy-lower-bound} and the definition of
$\mathcal E_N^\tau$ first give
\begin{equation}
  \tau H_N\bigl(f_N^\tau(t)\mid(\rho_t^\tau)^{\otimes N}\bigr)
  \le\mathcal E_N^\tau(t)+\err
  \le C\err.
  \label{eq:chaos-entropy-control}
\end{equation}
Because $f_N^\tau(t)$ is symmetric, normalized entropy
subadditivity \cite[Eq.~(12)]{JabinWang2018} implies
\begin{equation}
  \int_{(\T^d)^k}f_{N;k}^\tau(t)
    \log\frac{f_{N;k}^\tau(t)}{(\rho_t^\tau)^{\otimes k}}
    \,dx_1\cdots dx_k
  \le kH_N\bigl(f_N^\tau(t)\mid(\rho_t^\tau)^{\otimes N}\bigr).
  \label{eq:chaos-marginal-entropy}
\end{equation}
Pinsker's inequality and \eqref{eq:chaos-entropy-control} therefore
yield, for $1\le k\le N$,
\[
  \|f_{N;k}^\tau(t)-(\rho_t^\tau)^{\otimes k}\|_{L^1}
  \le\left(2kH_N\bigl(f_N^\tau(t)
                  \mid(\rho_t^\tau)^{\otimes N}\bigr)\right)^{1/2}
  \le C\left(\frac{k\err}{\tau}\right)^{1/2}.
\]
This proves \eqref{eq:fixed-temperature-marginal-chaos} for almost
every $t$, with a constant independent of $N$, $k$, and
$\tau$.
\end{proof}

\subsection{Approximation of the prescribed target}
\label{subsec:target-approximation}

The preceding theorem controls the particles relative to the
mean-field density at the same time and temperature.  To estimate
their distance from the prescribed target $\pi$, we also need
the decay of the density toward $\rho_\infty^\tau$ and the
distance from that equilibrium to $\pi$.  These are the three
terms in the identity
\begin{equation}
  \mu_{X_N}^t-\pi
  =(\mu_{X_N}^t-\rho_t^\tau)
   +(\rho_t^\tau-\rho_\infty^\tau)
   +(\rho_\infty^\tau-\pi).
  \label{eq:target-error-decomposition}
\end{equation}
Corollary~\ref{cor:intro-target-approximation} bounds their sum.
Corollary~\ref{cor:thermal-bias-lower} shows that its stationary
term has exact order $\tau$ for a nonconstant target.

\begin{corollary}[Approximation of the prescribed target]
\label{cor:intro-target-approximation}
Under Assumption~\ref{assump:standing-data}, let $N\ge2$.
Let $\kappa_*=(2\pi)^2/(2C_*)$ be the constant from
Proposition~\ref{prop:hminus-exp}.
For $\tau>0$, take expectation under $f_N^\tau(t)$.
For $\tau=0$, take expectation over the i.i.d.\ initial
configuration.  Then, for $\tau\ge0$,
\begin{equation}
  \left(\mathbb E\,\mathcal W_2(\mu_{X_N}^t,\pi)^2\right)^{1/2}
  \le C\left(
       \err^{1/2}
       +e^{-\frac{c_*}{2}(\frac12+\kappa_*\tau)t}
       +\tau\right).
  \label{eq:wasserstein-target-approximation}
\end{equation}
For every $\zeta>d/2+1$, it follows that
\begin{equation}
  \left(\mathbb E
     \|\mu_{X_N}^t-\pi\|_{\dot H^{-\zeta}}^2\right)^{1/2}
  \le C_\zeta\left(
       \err^{1/2}
       +e^{-\frac{c_*}{2}(\frac12+\kappa_*\tau)t}
       +\tau\right).
  \label{eq:intro-target-approximation}
\end{equation}
Both bounds hold for almost every $t\ge0$ if $\tau>0$, and
for every $t\ge0$ if $\tau=0$.
\end{corollary}

\begin{proof}
We estimate the three terms in
\eqref{eq:target-error-decomposition} in order.  Theorem~\ref{thm:uniform-mean-field}
bounds the first one in the root mean square Wasserstein-2 distance
by $C\err^{1/2}$.  For the second, the lower bound
$\rho_\infty^\tau\ge c_\pi$ in
\eqref{eq:stationary-euler-lagrange} and Peyre's comparison give
\begin{equation}
  \mathcal W_2(\rho_t^\tau,\rho_\infty^\tau)
  \le\frac2{\sqrt{c_\pi}}
       \|\rho_t^\tau-\rho_\infty^\tau\|_{\dot H^{-1}}
  \le C e^{-\frac{c_*}{2}(\frac12+\kappa_*\tau)t}.
  \label{eq:target-relaxation-wasserstein}
\end{equation}
For $\tau>0$, the last inequality is
\eqref{eq:appD-Hminus-accelerated-relaxation}.  At $\tau=0$
it is \eqref{eq:hminus-relaxation-all-temperatures}, since the
displayed exponent then equals $c_*/4$.

The last term is the stationary bias.  For mean-zero $h$, the
Poincar\'e inequality gives
$\|h\|_{\dot H^{-1}}\le(2\pi)^{-1}\|h\|_{L^2}$ by duality.
Since $\rho_\infty^\tau-\pi$ has zero mean and $\pi\ge c_\pi>0$,
Peyre's comparison and Proposition~\ref{prop:stat-bias} therefore give
\begin{equation}
  \mathcal W_2(\rho_\infty^\tau,\pi)
  \le\frac2{\sqrt{c_\pi}}
       \|\rho_\infty^\tau-\pi\|_{\dot H^{-1}}
  \le C\|\rho_\infty^\tau-\pi\|_{L^2}
  \le C\tau.
  \label{eq:target-stationary-bias-wasserstein}
\end{equation}
The triangle inequality for $\mathcal W_2$ and Minkowski's
inequality in the probability space combine these three estimates
to give \eqref{eq:wasserstein-target-approximation}.
Finally, \eqref{eq:negative-sobolev-wasserstein-comparison} with
$\rho=\pi$ gives
\eqref{eq:intro-target-approximation}.
\end{proof}

At fixed positive temperature the density relaxes toward
$\rho_\infty^\tau$, which can differ from $\pi$.  The next
corollary determines the size of this difference in Wasserstein and
negative Sobolev distances when the target is nonconstant and
$\tau$ is small.

\begin{corollary}[Two-sided thermal bias bounds]
\label{cor:thermal-bias-lower}
Suppose that $\pi$ is nonconstant.  There are constants
$c,C,\tau_1>0$ such that, for $0<\tau\le\tau_1$,
\begin{equation}
  c\tau\le\mathcal W_1(\rho_\infty^\tau,\pi)
  \le\mathcal W_2(\rho_\infty^\tau,\pi)\le C\tau.
  \label{eq:nondegenerate-thermal-bias}
\end{equation}
Moreover, for every $\zeta>d/2+1$, there are constants
$c_\zeta,C_\zeta, \tau_\zeta >0$ such that, for
$0<\tau\le\tau_\zeta$,
\begin{equation}
  c_\zeta\tau
  \le\|\rho_\infty^\tau-\pi\|_{\dot H^{-\zeta}}
  \le C_\zeta\tau.
  \label{eq:nondegenerate-sobolev-thermal-bias}
\end{equation}
\end{corollary}

\begin{proof}
We prove the upper bound from the stationary bias estimate and the
lower bound by testing against a particular Lipschitz function.
Equation~\eqref{eq:target-stationary-bias-wasserstein} already gives
$\mathcal W_2(\rho_\infty^\tau,\pi)\le C\tau$.

For the lower bound, nonconstancy of $\pi$ implies
$L:=\|\nabla\log\pi\|_{L^\infty}>0$.  The function
$\phi:=-\log\pi/L$ is $1$-Lipschitz on $\T^d$.
Kantorovich--Rubinstein duality and the stationary expansion
\eqref{eq:stat-bias-expansion} give
\begin{equation}
\begin{aligned}
  \mathcal W_1(\rho_\infty^\tau,\pi)
  &\ge\int_{\T^d}\phi(\rho_\infty^\tau-\pi)\,dx\\
  &=-\frac{\tau}{L}
        \int_{\T^d}(\log\pi)\Delta\log\pi\,dx
        +O(\tau^2)\\
  &=\frac{\tau}{L}
        \|\nabla\log\pi\|_{L^2}^2+O(\tau^2).
\end{aligned}
  \label{eq:thermal-bias-dual-lower-bound}
\end{equation}
The leading coefficient is positive because $\pi$ is
nonconstant.  Decreasing $\tau_1$ makes the remainder at most
half of that term, which proves the lower bound in
\eqref{eq:nondegenerate-thermal-bias}.

The same expansion gives the negative Sobolev bound.  Write
$r_\tau:=\rho_\infty^\tau-\pi-\tau\Delta\log\pi$.
It has mean zero and $\|r_\tau\|_{L^\infty}\le C\tau^2$ by
\eqref{eq:stat-bias-expansion}.  Since $\dot H^{-\zeta}$ is
bounded by $L^2$ on mean-zero functions,
\begin{equation}
  \tau\|\Delta\log\pi\|_{\dot H^{-\zeta}}-C_\zeta\tau^2
  \le\|\rho_\infty^\tau-\pi\|_{\dot H^{-\zeta}}
  \le\tau\|\Delta\log\pi\|_{\dot H^{-\zeta}}+C_\zeta\tau^2.
  \label{eq:thermal-bias-sobolev-expansion}
\end{equation}
The function $\Delta\log\pi$ is nonzero, since the only harmonic
functions on the torus are constant and $\pi$ is nonconstant.
After choosing $\tau_\zeta>0$ sufficiently small,
\eqref{eq:thermal-bias-sobolev-expansion} proves
\eqref{eq:nondegenerate-sobolev-thermal-bias}.
\end{proof}

The triangle inequality also compares the root mean square target
error directly with the stationary bias.  With the same time
convention as in Corollary~\ref{cor:intro-target-approximation},
\[
  \left|
    \left(\mathbb E\,\mathcal W_2(\mu_{X_N}^t,\pi)^2\right)^{1/2}
    -\mathcal W_2(\rho_\infty^\tau,\pi)
  \right|
  \le C\left(
    \err^{1/2}
    +e^{-\frac{c_*}{2}(\frac12+\kappa_*\tau)t}
  \right).
\]
Thus, for fixed sufficiently small $\tau>0$, the error in
approximating $\pi$ approaches a quantity of exact order
$\tau$ as $N,t\to\infty$.

\begingroup
\sloppy
\bibliographystyle{alpha}
\bibliography{targeted_coulomb_references}
\endgroup

\appendix
\section{Deterministic particle flow and collision exclusion}
\label{app:deterministic-particles}

The main goal of this appendix is to showcase that the deterministic and stochastic particle systems exist globally and avoid collisions, making the subsequent energy calculations well defined. The pathwise particle calculation in
Proposition~\ref{prop:deterministic-modulated-identity} requires a
deterministic flow that remains away from collisions.  We obtain it from
the microscopic energy: the energy decreases along the deterministic
flow, while its value diverges if two particles meet.  The same energy
also gives a collision barrier for the stochastic particle system.
Lemma~\ref{lem:microscopic-energy-structure} identifies the particle
drift as an energy gradient.  Proposition~\ref{prop:deterministic-particle-flow}
and Proposition~\ref{prop:stochastic-particle-flow} then give the two
global flows, and Corollary~\ref{cor:random-initial-collision-free}
applies the deterministic result to the i.i.d.\ initial data of the
main theorems.

\subsection{The microscopic gradient-flow energy}
\label{subsec:microscopic-gradient-energy}

We first find an energy whose gradient is the particle drift and check
that the prescribed target does not disturb local well-posedness.
The periodic Coulomb kernel is even and smooth away from the origin.
It is bounded below and satisfies
$K_1(z)\to+\infty$ as $d_{\T^d}(z,0)\to0$
\cite[Section~3]{ChodronRosenzweigSerfaty2025}.  Set
\[
  \underline K
  :=\inf_{z\in\T^d\setminus\{0\}}K_1(z)>-\infty.
\]
For a collision-free configuration
$X_N=(x_1,\ldots,x_N)\in\mathcal X_N$, define the microscopic energy by
\begin{equation}
  \mathcal H_N^\pi(X_N)
  :=\frac1{2N}\sum_{i\ne j}K_1(x_i-x_j)
    -\sum_{i=1}^N(K_1*\pi)(x_i).
  \label{eq:microscopic-targeted-energy}
\end{equation}
The first sum counts each unordered pair once after the factor
$1/2$, and the second sum is the energy of the external potential
$-K_1*\pi$.  With this normalization, the Gibbs weight in
\eqref{eq:targeted-gibbs-weight} satisfies
$G_N=\exp(-\mathcal H_N^\pi/\tau)$ when $\tau>0$.

\begin{lemma}[Microscopic energy and target field]
\label{lem:microscopic-energy-structure}
Under Assumption~\ref{assump:standing-data}, the energy
$\mathcal H_N^\pi$ is twice continuously differentiable on
$\mathcal X_N$.  Its negative gradient is the drift in
\eqref{eq:intro-particle-system}.  Consequently, a trajectory
$X_N(t)=(x_1(t),\ldots,x_N(t))$ of the deterministic particle
system solves, for each $i=1,\ldots,N$,
\begin{equation}
  \dot x_i(t)
  =-\nabla_{x_i}\mathcal H_N^\pi(X_N(t))
  =-\frac1N\sum_{j\ne i}\nabla K_1(x_i(t)-x_j(t))
    +\nabla K_1*\pi(x_i(t)).
  \label{eq:microscopic-energy-gradient}
\end{equation}
Moreover, the target field is Lipschitz and obeys
\begin{equation}
  \|\nabla K_1*\pi\|_{W^{1,\infty}}
  \le C\|\nabla K_1\|_{L^1}\|\pi\|_{C^1}.
  \label{eq:target-field-regularity}
\end{equation}
Consequently, the deterministic drift is locally Lipschitz on
$\mathcal X_N$.
\end{lemma}

\begin{proof}
We differentiate \eqref{eq:microscopic-targeted-energy} with respect
to $x_i$.  Evenness of $K_1$ makes the two occurrences of each
pair contribute equally, which gives
\[
  \nabla_{x_i}\mathcal H_N^\pi(X_N)
  =\frac1N\sum_{j\ne i}\nabla K_1(x_i-x_j)
   -\nabla K_1*\pi(x_i).
\]
This is \eqref{eq:microscopic-energy-gradient}.  The convolution
$K_1*\pi$ is twice continuously differentiable because
$K_1\in L^1$ and $\pi\in C^4$.  To control its gradient, we use
$\nabla K_1\in L^1(\T^d)$ and move one derivative to $\pi$:
\[
  \nabla(\nabla K_1*\pi)=\nabla K_1*\nabla\pi.
\]
Young's convolution inequality proves
\eqref{eq:target-field-regularity}.  The pair field is smooth when its
arguments are distinct, so this estimate gives the asserted local
Lipschitz property.
\end{proof}

\subsection{Global existence and collision exclusion}
\label{subsec:deterministic-collision-exclusion}

The local Lipschitz property in
Lemma~\ref{lem:microscopic-energy-structure} gives a particle flow up
to a possible collision.  We first show that deterministic energy
dissipation keeps every pair apart.  We then apply the result to random
initial configurations and use a stopped energy estimate to treat
Brownian noise.

\begin{proposition}[Global deterministic particle flow]
\label{prop:deterministic-particle-flow}
\label{lem:no-collision}
Let $N\ge2$ and $X_N(0)\in\mathcal X_N$.  Equation
\eqref{eq:microscopic-energy-gradient} has a unique global solution
$X_N\in C^1([0,\infty);\mathcal X_N)$.  It satisfies
\begin{equation}
  \mathcal H_N^\pi(X_N(t))
  +\int_0^t\sum_{i=1}^N|\dot x_i(s)|^2\,ds
  =\mathcal H_N^\pi(X_N(0)),
  \qquad t\ge0.
  \label{eq:microscopic-energy-dissipation}
\end{equation}
There is a constant $\delta=\delta(N,\pi,X_N(0))>0$ such that
\begin{equation}
  \min_{i\ne j}d_{\T^d}(x_i(t),x_j(t))\ge\delta,
  \qquad t\ge0.
  \label{eq:uniform-pathwise-separation}
\end{equation}
\end{proposition}

\begin{proof}
We construct a local solution, prove the energy identity, and use the
energy to keep the solution in a compact subset of $\mathcal X_N$.
The last step allows the local solution to continue for all time.

\medskip
\noindent\emph{Local flow and dissipation.}
Lemma~\ref{lem:microscopic-energy-structure} and the local ODE theorem
give a unique maximal solution on $[0,T_{\max})$.  Along this
solution, the chain rule and
\eqref{eq:microscopic-energy-gradient} give
\[
  \frac{d}{dt}\mathcal H_N^\pi(X_N(t))
  =\sum_{i=1}^N
      \nabla_{x_i}\mathcal H_N^\pi(X_N(t))\cdot\dot x_i(t)
  =-\sum_{i=1}^N|\dot x_i(t)|^2.
\]
Integrating proves \eqref{eq:microscopic-energy-dissipation} for
$t<T_{\max}$.  In particular, throughout the maximal interval,
\[
  \mathcal H_N^\pi(X_N(t))
  \le\mathcal H_N^\pi(X_N(0)).
\]

\medskip
\noindent\emph{Pair separation and continuation.}
To see how the energy bound controls a single pair, fix
$i<j$ and any $X_N\in\mathcal X_N$.  The other
$\binom N2-1$ pair interactions are each at least
$\underline K$, while the target potential is bounded.  Thus
\begin{equation}
  \mathcal H_N^\pi(X_N)
  \ge\frac1N K_1(x_i-x_j)
    +\frac1N\left(\binom N2-1\right)\underline K
    -N\|K_1*\pi\|_{L^\infty}.
  \label{eq:single-pair-energy-lower-bound}
\end{equation}
Applying this estimate along the trajectory and using its decreasing
energy bounds every pair potential by the same finite constant:
\[
\begin{aligned}
  K_1(x_i(t)-x_j(t))
  &\le N\mathcal H_N^\pi(X_N(0))
      -\left(\binom N2-1\right)\underline K
      +N^2\|K_1*\pi\|_{L^\infty}\\
  & =:C_{N,X_N(0)}.
\end{aligned}
\]
Since $K_1(z)\to+\infty$ as $z\to0$, there is a
$\delta>0$ such that $K_1(z)>C_{N,X_N(0)}$ whenever
$0<d_{\T^d}(z,0)<\delta$.  The preceding bound therefore
proves \eqref{eq:uniform-pathwise-separation}.  In particular,
\[
  X_N(t)\in
  \mathcal K_\delta
  :=\left\{X_N\in(\T^d)^N:
       \min_{i\ne j}d_{\T^d}(x_i,x_j)\ge\delta\right\}
  \Subset\mathcal X_N,
  \qquad t<T_{\max}.
\]
The drift is bounded and locally Lipschitz on a neighborhood of this
compact set.  The continuation theorem for ODEs gives
$T_{\max}=\infty$, and the energy identity consequently holds for
every $t\ge0$.
\end{proof}

The separation constant in \eqref{eq:uniform-pathwise-separation}
can depend on the initial energy and on $N$.  For the i.i.d.\ data
used in the main results, the next corollary supplies an almost surely
defined collision-free flow and finite initial energies.

\begin{corollary}[I.i.d.\ initial configurations]
\label{cor:random-initial-collision-free}
Let $X_1(0),\ldots,X_N(0)$ be i.i.d.\ with density $\rho_0$.
Then Proposition~\ref{prop:deterministic-particle-flow} defines a
global flow almost surely, and
\begin{equation}
\begin{gathered}
  \mathbb P\bigl(X_N(t)\in\mathcal X_N
       \text{ for all }t\ge0\bigr)=1,\\
  \mathbb E\bigl[
    |\mathcal H_N^\pi(X_N(0))|
    +|F_N(X_N(0),\rho_0)|\bigr]<\infty.
\end{gathered}
  \label{eq:iid-microscopic-integrability}
\end{equation}
For every $t\ge0$, the law of $X_N(t)$ is invariant under
permutations of the particle labels.
\end{corollary}

\begin{proof}
The collision set is the finite union
\[
  (\T^d)^N\setminus\mathcal X_N
  =\bigcup_{i<j}\{X_N:x_i=x_j\}.
\]
Each set in this union has Lebesgue measure zero.  Since the joint
initial law has density $\rho_0^{\otimes N}$, the initial
configuration belongs to $\mathcal X_N$ almost surely, and
Proposition~\ref{prop:deterministic-particle-flow} applies.

For the stated integrability, $K_1\in L^1(\T^d)$ and
$\rho_0\in L^\infty$ give, for $i\ne j$,
\[
  \mathbb E|K_1(X_i(0)-X_j(0))|
  \le\|\rho_0\|_{L^\infty}\|K_1\|_{L^1}<\infty.
\]
The target potential and $K_1*\rho_0$ are bounded.  Their
contributions to $\mathcal H_N^\pi$ and to the expansion of
$F_N$ are therefore integrable, proving
\eqref{eq:iid-microscopic-integrability}.  Finally, the drift in
\eqref{eq:microscopic-energy-gradient} commutes with every
permutation of the labels.  Uniqueness of the flow and symmetry of
$\rho_0^{\otimes N}$ give the stated invariance of its law.
\end{proof}

Brownian noise removes the pathwise energy monotonicity used above.
The energy is still a collision barrier because the drift of its
It\^o decomposition has a uniform upper bound.  Stopping at a large
energy level turns this observation into a bound on the probability
of reaching that level.

\begin{proposition}[Global well-posedness of the particle SDE]
  \label{prop:stochastic-particle-flow}
Let $N\ge2$, $\tau>0$, and let the initial configuration
$X_N(0)\in\mathcal X_N$ be deterministic.  The SDE
\eqref{eq:intro-particle-system} has a unique global strong solution
for the given Brownian motions, and
\begin{equation}
  \mathbb P\bigl(X_N(t)\in\mathcal X_N
                  \text{ for every }t\ge0\bigr)=1.
  \label{eq:stochastic-no-collision}
\end{equation}
The same conclusion holds if $X_N(0)$ has law
$\rho_0^{\otimes N}$ and is independent of the Brownian motions.
\end{proposition}

\begin{proof}
We construct the solution up to its first collision and estimate the
chance that its energy reaches a large level.  The calculation adapts
the stopping-time argument of
\cite[Section~4]{RosenzweigSerfaty2023}; see
\cite[Section~1.1]{ChodronRosenzweigSerfaty2025}
for a discussion of the periodic Coulomb case.  Because every
bounded energy sublevel is separated from collisions, a bound that
vanishes as the level increases will rule out a finite lifetime.

\medskip
\noindent\emph{Local solution and energy equation.}
Lemma~\ref{lem:microscopic-energy-structure} gives a locally
Lipschitz drift on $\mathcal X_N$.  Standard local SDE theory
therefore supplies a unique strong solution up to its maximal
lifetime $T_{\max}$.  Away from collisions,
$\Delta K_1=1$ and $\Delta(K_1*\pi)=1-\pi$.  Hence
\[
  \sum_{i=1}^N\Delta_{x_i}\mathcal H_N^\pi(X_N)
  =(N-1)-\sum_{i=1}^N(1-\pi(x_i))
  =\sum_{i=1}^N\pi(x_i)-1.
\]
Since the particle drift is $-\nabla\mathcal H_N^\pi$,
It\^o's formula gives, for $t<T_{\max}$,
\begin{equation}
\begin{aligned}
  d\mathcal H_N^\pi(X_N(t))
  ={}&\left[-\sum_{i=1}^N
       |\nabla_{x_i}\mathcal H_N^\pi(X_N(t))|^2
       +\tau\left(\sum_{i=1}^N\pi(x_i(t))-1\right)\right]dt\\
  &+\sqrt{2\tau}\sum_{i=1}^N
       \nabla_{x_i}\mathcal H_N^\pi(X_N(t))\cdot dB_i(t).
\end{aligned}
  \label{eq:stochastic-microscopic-energy-ito}
\end{equation}
The first term in the drift is nonpositive.  Thus the drift is at
most $\tau(N\|\pi\|_{L^\infty}-1)$, which is nonnegative because
$\pi$ is a probability density on a torus of unit volume.

\medskip
\noindent\emph{Compact energy sublevels.}
The bound \eqref{eq:single-pair-energy-lower-bound}, applied to
every pair, has two consequences.  First, the energy has the global
lower bound
\[
  \mathcal H_N^\pi(X_N)\ge L_N,
  \qquad
  L_N:=\frac{N-1}{2}\underline K
       -N\|K_1*\pi\|_{L^\infty}.
\]
Second, for each finite $C$, the set
\begin{equation}
  \mathcal K_C
  :=\{X_N\in\mathcal X_N:\mathcal H_N^\pi(X_N)\le C\}
  \Subset\mathcal X_N
  \label{eq:microscopic-compact-sublevel}
\end{equation}
is compact.  Indeed, if a sequence in this set approached a collision,
the interaction of the colliding pair would diverge, contradicting
\eqref{eq:single-pair-energy-lower-bound}.

\medskip
\noindent\emph{Probability of reaching an energy level.}
Fix $C>\mathcal H_N^\pi(X_N(0))$ and set
\[
  T_C:=\inf\{t<T_{\max}:\mathcal H_N^\pi(X_N(t))\ge C\},
  \qquad \inf\varnothing:=\infty.
\]
If $T_{\max}<\infty$ but $T_C=\infty$, the solution stays in
$\mathcal K_C$ up to $T_{\max}$.  Local SDE continuation on the
compact set \eqref{eq:microscopic-compact-sublevel} contradicts
maximality.  Thus the process stopped at $T_C$ is defined for every
finite time, and $\nabla\mathcal H_N^\pi$ is bounded before stopping.

The stopped stochastic integral in
\eqref{eq:stochastic-microscopic-energy-ito} consequently has mean
zero.  Taking expectations and using the drift bound gives
\begin{equation}
  \mathbb E\bigl[\mathcal H_N^\pi(X_N(t\wedge T_C))\bigr]
  \le\mathcal H_N^\pi(X_N(0))
     +\tau(N\|\pi\|_{L^\infty}-1)t.
  \label{eq:stopped-energy-expectation}
\end{equation}
At $T_C$ the energy equals $C$ by continuity.  On the event
$\{T_C>t\}$, it is at least $L_N$.  Combining these two lower
bounds with \eqref{eq:stopped-energy-expectation} yields
\begin{equation}
  \mathbb P(T_C\le t)
  \le\frac{
    \mathcal H_N^\pi(X_N(0))-L_N
       +\tau(N\|\pi\|_{L^\infty}-1)t
  }{C-L_N}.
  \label{eq:microscopic-energy-hitting-probability}
\end{equation}

\medskip
\noindent\emph{Global continuation.}
On $\{T_{\max}\le t\}$, compact-sublevel continuation implies
$T_C<T_{\max}\le t$ for every $C>\mathcal H_N^\pi(X_N(0))$.
The bound \eqref{eq:microscopic-energy-hitting-probability} therefore
gives
\[
  0\le\mathbb P(T_{\max}\le t)
  \le\mathbb P(T_C\le t)
  \longrightarrow0 \qquad(C\to\infty).
\]
Taking the countable union
$\{T_{\max}<\infty\}
 =\bigcup_{m=1}^\infty\{T_{\max}\le m\}$
shows that $T_{\max}=\infty$ almost surely.  Local pathwise
uniqueness gives uniqueness of the global strong solution, and its
paths remain in $\mathcal X_N$.

Finally, an i.i.d.\ initial configuration lies in $\mathcal X_N$
almost surely by Corollary~\ref{cor:random-initial-collision-free}.
Conditioning on that configuration and using independence from the
Brownian motions applies the deterministic-initial-point argument
almost surely.
\end{proof}

\section{Entropy solutions and the singular Coulomb limit}
\label{app:targeted-entropy}

This appendix proves Proposition~\ref{prop:targeted-mfe}, which supplies
the positive-temperature particle inequality used in
Proposition~\ref{prop:assembly}.  We first identify its dissipation
and commutator for a smooth interaction.  We then construct an entropy
solution for the Coulomb interaction and show that the same comparison
holds for every entropy solution.  The construction adapts the
compactness argument in
\cite[Lemma~6.2 and Appendix~A]{ChodronRosenzweigSerfaty2025},
using heat regularization in place of their truncation.
The comparison follows the weighted method of
\cite[Propositions~2.2--2.3]{BreschJabinWang2023}.  We retain the
prescribed target field throughout the regularization.

In this appendix, constants may depend on $N$, $\tau$, $T$,
$d$, $\pi$, and the initial weighted entropy, but not on the
regularization parameter $\varepsilon$.  We write
$\rho_t:=\rho_t^\tau$ whenever the temperature is fixed.

\subsection{The smooth weighted calculation}
\label{subsec:smooth-weighted-calculation}

This subsection proves the identity \eqref{eq:smooth-targeted-mfe-identity}
when the interaction kernel and both densities are smooth.  It shows
how the particle entropy dissipation leaves precisely the Coulomb
commutator.  Subsection~\ref{subsec:singular-targeted-mfe-passage}
then passes from this calculation to the singular interaction.

\begin{lemma}[Smooth weighted identity]
\label{lem:appB-smooth-weighted-identity}
Fix $N\ge2$ and $\tau>0$.  For this calculation, replace $K_1$
in the particle and mean-field equations and their weights by a
smooth even kernel, denoted again by $K_1$.  Assume that the positive
particle density $f_N^\tau$ and the positive mean-field density
$\rho_t=\rho_t^\tau$ are smooth solutions.  With
$J_i=f_N^\tau\nabla_{x_i}\log(f_N^\tau/G_N)$, one has
\[
\begin{aligned}
  \frac{d}{dt}\mathcal E_N^\tau(t)
  ={}&-\frac1N\sum_{i=1}^N\int_{(\T^d)^N}
      \frac{|\tau J_i-f_N^\tau u_t^\tau(x_i)|^2}{f_N^\tau}\,dX_N\\
    &-\frac12\int_{(\T^d)^N}f_N^\tau
      \Comm_N[u_t^\tau](X_N,\rho_t)\,dX_N.
\end{aligned}
\]
The first integrand equals
$f_N^\tau|\tau\nabla_{x_i}\log(f_N^\tau/\rho_t^{\otimes N})
           +E_i(X_N,\rho_t)|^2$.  Thus this is
\eqref{eq:smooth-targeted-mfe-identity}.
\end{lemma}

\begin{proof}
We differentiate the two terms in the weighted representation
\eqref{eq:weighted-entropy-representation}.  The first gives a square
through the particle flux.  The mean-field equation then cancels the
time derivative of $\rho_t^{\otimes N}$, leaving the commutator.

\noindent\emph{The particle flux and the mean-field weight.}
Evenness of $K_1$ and the definition of $G_N$ in
\eqref{eq:targeted-gibbs-weight} give
\[
  -\tau\nabla_{x_i}\log G_N
  =\frac1N\sum_{j\ne i}\nabla K_1(x_i-x_j)
     -\nabla K_1*\pi(x_i).
\]
The Liouville equation \eqref{eq:intro-liouville-equation} therefore
takes the weighted form
\[
  \partial_t f_N^\tau
  =\tau\sum_{i=1}^N\nabla_{x_i}\cdot J_i,
  \qquad
  J_i:=\nabla_{x_i}f_N^\tau
       -f_N^\tau\nabla_{x_i}\log G_N
      =f_N^\tau\nabla_{x_i}\log\frac{f_N^\tau}{G_N}.
\]
For the tensorized mean-field weight $G_{\rho,N}$ of
\eqref{eq:mean-field-gibbs-weight}, the definition of $u_t^\tau$
in \eqref{eq:free-energy-velocity} gives
\[
  \tau\nabla_{x_i}\log
     \frac{\rho_t^{\otimes N}}{G_{\rho,N}(t)}
  =\tau\nabla\log\rho_t(x_i)
    +\nabla K_1*(\rho_t-\pi)(x_i)
  =u_t^\tau(x_i).
\]

\noindent\emph{Differentiate the weighted representation.}
By \eqref{eq:weighted-entropy-representation},
\[
  \mathcal E_N^\tau(t)
  =\frac\tau N\left[
     \int f_N^\tau\log\frac{f_N^\tau}{G_N}\,dX_N
     -\int f_N^\tau
        \log\frac{\rho_t^{\otimes N}}{G_{\rho,N}(t)}\,dX_N
  \right].
\]
The weight $G_N$ is independent of time.  Differentiating the
first integral, using conservation of mass, and integrating by parts
on the torus yield
\[
  \frac{d}{dt}\int f_N^\tau\log\frac{f_N^\tau}{G_N}\,dX_N
  =-\tau\sum_{i=1}^N\int
      J_i\cdot\nabla_{x_i}\log\frac{f_N^\tau}{G_N}\,dX_N
  =-\tau\sum_{i=1}^N\int\frac{|J_i|^2}{f_N^\tau}\,dX_N.
\]
For the second integral, we also differentiate its time-dependent
weight.  The flux equation and the preceding formula for its spatial
gradient give
\begin{equation}
\begin{aligned}
  \frac{d}{dt}\int f_N^\tau
     \log\frac{\rho_t^{\otimes N}}{G_{\rho,N}}\,dX_N
  ={}&\int f_N^\tau\left(
       \sum_{i=1}^N\partial_t\log\rho_t(x_i)
       -\partial_t\log G_{\rho,N}\right)dX_N\\
     &-\sum_{i=1}^N\int J_i\cdot u_t^\tau(x_i)\,dX_N.
\end{aligned}
\label{eq:appB-smooth-test-identity}
\end{equation}

\noindent\emph{Separate the dissipation.}
The difference between the particle flux and the mean-field flux is
identified by \eqref{eq:weighted-gradient-identities} and the field
error \eqref{eq:microscopic-field-error}:
\[
  \tau J_i-f_N^\tau u_t^\tau(x_i)
  =f_N^\tau\left[
       \tau\nabla_{x_i}\log\frac{f_N^\tau}{\rho_t^{\otimes N}}
       +E_i(X_N,\rho_t)\right].
\]
Completing the square in the two flux terms from the preceding
derivatives gives
\begin{equation}
\begin{aligned}
  -\tau^2\frac{|J_i|^2}{f_N^\tau}
    +\tau J_i\cdot u_t^\tau(x_i)
  ={}&-\frac{|\tau J_i-f_N^\tau u_t^\tau(x_i)|^2}{f_N^\tau}\\
     &-u_t^\tau(x_i)\cdot
        \bigl(\tau J_i-f_N^\tau u_t^\tau(x_i)\bigr).
\end{aligned}
\label{eq:appB-completed-square}
\end{equation}
It remains to combine the final term with the time derivative of
$\rho_t^{\otimes N}$.  Since
$\partial_t\rho_t=\nabla\cdot(\rho_tu_t^\tau)$,
\[
  \partial_t\log\rho_t
  =\nabla\cdot u_t^\tau
       +u_t^\tau\cdot\nabla\log\rho_t.
\]
Periodic integration by parts now shows that
\[
  \int u_t^\tau(x_i)\cdot
      f_N^\tau\tau\nabla_{x_i}
         \log\frac{f_N^\tau}{\rho_t^{\otimes N}}\,dX_N
  =-\tau\int f_N^\tau\partial_t\log\rho_t(x_i)\,dX_N.
\]
Together with the flux difference above, this proves
\begin{equation}
\begin{aligned}
  &\sum_{i=1}^N\int u_t^\tau(x_i)\cdot
    \bigl(\tau J_i-f_N^\tau u_t^\tau(x_i)\bigr)\,dX_N
      +\tau\int f_N^\tau
         \sum_{i=1}^N\partial_t\log\rho_t(x_i)\,dX_N\\
  &\qquad=\int f_N^\tau\sum_{i=1}^N
        u_t^\tau(x_i)\cdot E_i(X_N,\rho_t)\,dX_N.
\end{aligned}
\label{eq:appB-pointwise-remainder}
\end{equation}
Consequently, the desired identity reduces to an equality for each
fixed particle configuration:
\[
\begin{aligned}
  \frac{d}{dt}\mathcal E_N^\tau(t)
  ={}&-\frac1N\sum_{i=1}^N\int
       \frac{|\tau J_i-f_N^\tau u_t^\tau(x_i)|^2}{f_N^\tau}\,dX_N\\
     &+\int f_N^\tau\left[
       \frac\tau N\partial_t\log G_{\rho,N}
       -\frac1N\sum_{i=1}^N
          u_t^\tau(x_i)\cdot E_i(X_N,\rho_t)
      \right]dX_N.
\end{aligned}
\]

\noindent\emph{Identify the remaining term.}
The time derivative of $G_{\rho,N}$ is the response of the
mean-field weight to the changing density.  By evenness of $K_1$
and the mean-field equation,
\[
\begin{aligned}
  \frac\tau N\partial_t\log G_{\rho,N}
  &=-\frac1N\sum_{i=1}^N(K_1*\partial_t\rho_t)(x_i)
       +\int_{\T^d}(K_1*\rho_t)\partial_t\rho_t\,dx\\
  &=-\int_{\T^d}K_1*(\mu_{X_N}-\rho_t)
                      \partial_t\rho_t\,dx\\
  &=\int_{\T^d}\rho_tu_t^\tau\cdot
               \nabla K_1*(\mu_{X_N}-\rho_t)\,dx.
\end{aligned}
\]
The field error in \eqref{eq:microscopic-field-error} is the
evaluation of $\nabla K_1*(\mu_{X_N}-\rho_t)$ at $x_i$, with
the diagonal particle interaction omitted.  Subtracting its weighted
particle average from the last display and then exchanging the two
integration variables gives, by
\eqref{eq:field-error-pv-identity} and
\eqref{eq:commutator-symmetrization},
\begin{equation}
\begin{aligned}
  \frac\tau N\partial_t\log G_{\rho,N}
  &-\frac1N\sum_{i=1}^N E_i(X_N,\rho_t)\cdot u_t^\tau(x_i)\\
  &=-\int_{\T^d}u_t^\tau\cdot
       \nabla K_1*(\mu_{X_N}-\rho_t)
               \,d(\mu_{X_N}-\rho_t)\\
  &=-\frac12\Comm_N[u_t^\tau](X_N,\rho_t).
\end{aligned}
\label{eq:appB-pointwise-commutator}
\end{equation}
Here and in the first integral of the last display, all interactions
of a particle with itself are omitted.  Inserting this equality in
the previous expression for $d\mathcal E_N^\tau/dt$ proves the lemma.
\end{proof}
\subsection{Construction of targeted entropy solutions}
\label{subsec:targeted-entropy-compactness}

We construct the entropy solution in
Proposition~\ref{prop:targeted-mfe}\textup{(i)}.  Smoothing the pair interaction gives
a diffusion equation with a smooth drift.  The weighted entropy estimate
for that equation controls its density and flux uniformly as the
smoothing is removed.  We keep the prescribed target field
$\nabla K_1*\pi$ fixed throughout this approximation.

\begin{proposition}[Construction by regularization]
\label{prop:targeted-entropy-construction}
Let $N\ge2$, $\tau>0$, and let $f_{N,0}$ be a symmetric probability
density on $(\T^d)^N$ such that
\[
  \int_{(\T^d)^N} f_{N,0}
       \log\frac{f_{N,0}}{G_N}\,dX_N<\infty.
\]
Then there is a symmetric global entropy solution $f_N^\tau$ of
\eqref{eq:intro-liouville-equation} with initial density $f_{N,0}$.
In particular, its weighted fluxes $J_i$, defined by
\eqref{eq:singular-targeted-weighted-equation}, belong to
$L^1_{\mathrm{loc}}([0,\infty)\times(\T^d)^N)$, and for almost every
$t\ge0$,
\begin{equation}
\begin{aligned}
  &\int f_N^\tau(t)\log\frac{f_N^\tau(t)}{G_N}\,dX_N
  +\tau\sum_{i=1}^N\int_0^t\int
       \frac{|J_i(s)|^2}{f_N^\tau(s)}\,dX_N\,ds\\
  &\qquad\le
    \int f_{N,0}\log\frac{f_{N,0}}{G_N}\,dX_N.
\end{aligned}
\end{equation}
Under Assumption~\ref{assump:standing-data}, if
$f_{N,0}=\rho_0^{\otimes N}$, then $f_N^\tau(t)$ is the density
of the law of $X_N(t)$ for every $t\ge0$, where $X_N$ is the
global strong solution in
Proposition~\ref{prop:stochastic-particle-flow}, with initial
positions independent of the Brownian motions.
\end{proposition}

\begin{proof}
We first derive bounds for a smooth equation, then extract limits of its
density and weighted flux.  We identify their equation away from particle
collisions and pass to the limit in the entropy inequality.  The estimates
below follow the untargeted construction of
\cite[Appendix~A]{ChodronRosenzweigSerfaty2025}.  The bounded target
potential enters the Gibbs weight as a fixed factor.

\medskip
\noindent\emph{Step 1. Regularization and uniform estimates.}
For $\varepsilon>0$, set
\[
\begin{aligned}
  K_{1,\varepsilon}&:=e^{\varepsilon\Delta}K_1,\\
  G_{N,\varepsilon}(X_N)
  &:=\exp\left[-\frac1{2N\tau}\sum_{i\ne j}
                    K_{1,\varepsilon}(x_i-x_j)
               +\frac1\tau\sum_i(K_1*\pi)(x_i)\right].
\end{aligned}
\]
For fixed $\varepsilon$, the drift of the regularized equation
is continuously differentiable and globally Lipschitz.
The equation therefore has a unique global mild solution
$f_{N,\varepsilon}$ with initial density $f_{N,0}$.
By It\^o's formula and Fokker--Planck uniqueness,
$f_{N,\varepsilon}(t)$ is the density of the law of the
corresponding particle SDE.
This density is classical and positive for $t>0$.  Symmetry of the initial
density and coefficients makes this solution symmetric.  Its weighted
flux is
\[
  J_{i,\varepsilon}
   :=G_{N,\varepsilon}\nabla_{x_i}
                 \frac{f_{N,\varepsilon}}{G_{N,\varepsilon}}
    =\nabla_{x_i}f_{N,\varepsilon}
       -f_{N,\varepsilon}\nabla_{x_i}\log G_{N,\varepsilon}.
\]
Here the only change to the Gibbs weight $G_N$ in
\eqref{eq:targeted-gibbs-weight} is the replacement of the pair
kernel.  If $G_{N,\varepsilon}^0$ denotes the same weight without the
target potential, then
\[
\begin{aligned}
  \log G_{N,\varepsilon}
  &=\log G_{N,\varepsilon}^0
     +\frac1\tau\sum_i(K_1*\pi)(x_i),\\
  \left|\int f\log\frac{f}{G_{N,\varepsilon}}
       -\int f\log\frac{f}{G_{N,\varepsilon}^0}\right|
  &\le\frac N\tau\|K_1*\pi\|_{L^\infty}
\end{aligned}
\]
for any probability density $f$.  Thus the target factor does not
affect bounds that are uniform in $\varepsilon$.

The heat kernel has unit mass and is nonnegative, so
$K_{1,\varepsilon}\ge\underline K$, where
$\underline K:=\inf_{x\ne0}K_1(x)$.  The identity
$\Delta K_1=1-\delta_0$ gives the more precise monotonicity
\[
  \partial_\varepsilon(K_{1,\varepsilon}-\varepsilon)
  =-e^{\varepsilon\Delta}\delta_0\le0,
  \qquad
  K_{1,\varepsilon}-\varepsilon\uparrow K_1
       \quad\text{as }\varepsilon\downarrow0.
\]
We use this relation to compare the initial entropies.  For a
probability density $f$, define
\[
  \mathcal H_\varepsilon(f)
  :=\int f\log\frac{f}{G_{N,\varepsilon}}\,dX_N,
  \qquad
  \mathcal H(f)
  :=\int f\log\frac{f}{G_N}\,dX_N.
\]
Since $f_{N,0}$ is fixed, monotone convergence after subtracting
the lower bound $\underline K-1$ gives
\begin{equation}
  \mathcal H_\varepsilon(f_{N,0})
      -\frac{N-1}{2\tau}\varepsilon
  \uparrow \mathcal H(f_{N,0})<\infty.
  \label{eq:appB-initial-entropy-convergence}
\end{equation}
In particular, the initial entropy is bounded uniformly for
$0<\varepsilon<1$.

The weighted equation is
$\partial_t f_{N,\varepsilon}
 =\tau\sum_i\nabla_{x_i}\cdot J_{i,\varepsilon}$.
Integrating the weighted entropy identity of
Subsection~\ref{subsec:smooth-weighted-calculation} from $s>0$
and then letting $s\downarrow0$, as justified below, gives
\begin{equation}
\begin{aligned}
  \mathcal H_\varepsilon(f_{N,\varepsilon}(t))
  +\tau\sum_{i=1}^N\int_0^t\int
       \frac{|J_{i,\varepsilon}|^2}{f_{N,\varepsilon}}
                 \,dX_N\,ds
  \le \mathcal H_\varepsilon(f_{N,0}).
\end{aligned}
  \label{eq:appB-regularized-entropy-inequality}
\end{equation}
For fixed $\varepsilon>0$, Jensen's inequality for the
regularized Markov semigroup gives
$\mathcal H_\varepsilon(f_{N,\varepsilon}(s))
\le \mathcal H_\varepsilon(f_{N,0})$.
Since $f_{N,\varepsilon}(s)\to f_{N,0}$ in $L^1$ as
$s\downarrow0$, this bound and lower semicontinuity of
$\mathcal H_\varepsilon$ imply convergence of the entropy
to its initial value.
The nonnegative dissipation integral converges by
monotone convergence.

The bound on $K_{1,\varepsilon}$ and the boundedness of
$K_1*\pi$ imply $\log G_{N,\varepsilon}\le C$, with
$C$ independent of $\varepsilon$.  Since the configuration torus
has volume one, every probability density $f$ satisfies
$\int f\log f\ge0$.  Consequently,
\[
  \mathcal H_\varepsilon(f)
     \ge\int f\log f\,dX_N-C\ge-C.
\]
Applying this lower bound to
\eqref{eq:appB-regularized-entropy-inequality} separates the two
uniform estimates:
\[
\begin{aligned}
  \sup_{0<\varepsilon<1}\operatorname*{ess\,sup}_{0<t<T}
       \int f_{N,\varepsilon}(t)\log f_{N,\varepsilon}(t)\,dX_N
  &\le C_T,\\
  \sup_{0<\varepsilon<1}\sum_i
       \int_0^T\int\frac{|J_{i,\varepsilon}|^2}{f_{N,\varepsilon}}
                   \,dX_N\,dt
  &\le C_T.
\end{aligned}
\]
The first bound controls concentration of the densities.  The second
will control the fluxes on sets where the densities have small mass.

\medskip
\noindent\emph{Step 2. Compactness of densities and fluxes.}
Put $Q_T:=(0,T)\times(\T^d)^N$.  The entropy bound gives, for every
measurable $A\subset Q_T$ and every $R>e$,
\[
  \int_A f_{N,\varepsilon}\,dX_N\,dt
  \le R|A|+\frac{C_T}{\log R}.
\]
Indeed, on $\{f_{N,\varepsilon}>R\}$, one has
$f_{N,\varepsilon}\le
 (f_{N,\varepsilon}\log f_{N,\varepsilon})/\log R$.
On the complement, the density is at most $R$.  The displayed bound
proves uniform integrability of $f_{N,\varepsilon}$.

For each flux, Cauchy--Schwarz and Step~1 now give
\begin{equation}
  \int_A|J_{i,\varepsilon}|\,dX_N\,dt
   \le\left(\int_A\frac{|J_{i,\varepsilon}|^2}
                         {f_{N,\varepsilon}}\right)^{1/2}
       \left(\int_A f_{N,\varepsilon}\right)^{1/2}
   \le C_T\left(\int_A f_{N,\varepsilon}\right)^{1/2}.
\end{equation}
Thus the fluxes are uniformly integrable as well.  The
Dunford--Pettis theorem supplies a subsequence and weak $L^1(Q_T)$
limits
\[
  f_{N,\varepsilon}\rightharpoonup f_N^\tau,
  \qquad J_{i,\varepsilon}\rightharpoonup J_i.
\]

The spacetime limits alone do not identify the initial law or the
entropy at a fixed time.  To obtain time traces, test the regularized
equation against $\phi\in C^1((\T^d)^N)$.  Since each density has
unit mass, Step~1 yields
\[
\begin{aligned}
  \left|\int\phi\bigl(f_{N,\varepsilon}(t)
                   -f_{N,\varepsilon}(s)\bigr)\,dX_N\right|
  &\le\tau\sum_i\int_s^t\int
          |J_{i,\varepsilon}|\,|\nabla_{x_i}\phi|\,dX_N\,dr\\
  &\le C_T|t-s|^{1/2}
          \max_i\|\nabla_{x_i}\phi\|_{L^\infty}.
\end{aligned}
\]
Probability measures on the compact configuration torus are narrowly
compact.  Applying Arzel\`a--Ascoli to a countable dense set of test
functions and taking a further subsequence gives
\begin{equation}
  f_{N,\varepsilon}(t)\,dX_N
     \rightharpoonup f_N^\tau(t)\,dX_N
  \quad\text{narrowly, uniformly for }t\in[0,T].
  \label{eq:appB-uniform-narrow-convergence}
\end{equation}
The entropy bound in Step~1 and lower semicontinuity of ordinary
entropy ensure that every limiting time marginal has a density.
The spacetime density is the same $f_N^\tau$ as in the weak
$L^1$ convergence.

\medskip
\noindent\emph{Step 3. The equation and its weighted flux.}
We first pass to the regularized weak equation.  For
$\phi\in C^1([0,T]\times(\T^d)^N)$, Steps~1 and~2 give
\[
\begin{aligned}
  \int\phi(T)f_N^\tau(T)\,dX_N
    -\int\phi(0)f_{N,0}\,dX_N
  ={}&\int_0^T\int f_N^\tau\,\partial_t\phi\,dX_N\,dt\\
     &-\tau\sum_i\int_0^T\int
                   J_i\cdot\nabla_{x_i}\phi\,dX_N\,dt.
\end{aligned}
\]
The endpoint terms follow from
\eqref{eq:appB-uniform-narrow-convergence}; the spacetime terms follow
from weak $L^1$ convergence.  In particular,
$\partial_t f_N^\tau=\tau\sum_i\nabla_{x_i}\cdot J_i$ holds on the
whole configuration space, and the initial law is $f_{N,0}$.

It remains to identify $J_i$ with the weighted flux in
\eqref{eq:singular-targeted-weighted-equation}.  Let
$\Phi\in C_c^1((0,T)\times\mathcal X_N;\mathbb R^d)$.
Integration by parts in its $x_i$ coordinate gives
\[
  \int_{Q_T}J_{i,\varepsilon}\cdot\Phi
   =-\int_{Q_T}f_{N,\varepsilon}\nabla_{x_i}\cdot\Phi
     -\int_{Q_T}f_{N,\varepsilon}
            \nabla_{x_i}\log G_{N,\varepsilon}\cdot\Phi.
\]
The gradients of the regularized kernels converge uniformly on the
support of $\Phi$, which stays away from collisions.  Passing to the
limit therefore gives
\[
  J_i=\nabla_{x_i}f_N^\tau
          -f_N^\tau\nabla_{x_i}\log G_N
     =G_N\nabla_{x_i}\frac{f_N^\tau}{G_N}
    \quad\text{in distributions on }(0,T)\times\mathcal X_N.
\]
The collision set has zero Lebesgue measure.  Hence the $L^1$ flux
$J_i$ is an extension of this distributional expression across that
set.  Symmetry of the limit follows from symmetry of the approximants.

\medskip
\noindent\emph{Step 4. Passage to the entropy inequality.}
The initial entropy converges by
\eqref{eq:appB-initial-entropy-convergence}.  For the final entropy,
the ordinary part $\int f_{N,\varepsilon}(t)
\log f_{N,\varepsilon}(t)$ is lower semicontinuous under the narrow
convergence \eqref{eq:appB-uniform-narrow-convergence}.
The target contribution converges because $K_1*\pi$ is bounded and
continuous.

For the pair contribution, define the shifted pair energy
\[
  P_\eta(X_N)
  :=\sum_{i\ne j}\bigl(K_{1,\eta}(x_i-x_j)-\eta\bigr),
  \qquad \eta>0.
\]
Fix $\delta>0$.  The monotonicity from Step~1 gives
$P_\varepsilon\ge P_\delta$ whenever
$0<\varepsilon<\delta$.  Since $P_\delta$ is bounded and
continuous, narrow convergence yields
\[
  \liminf_{\varepsilon\downarrow0}
     \int P_\varepsilon f_{N,\varepsilon}(t)\,dX_N
  \ge \int P_\delta f_N^\tau(t)\,dX_N.
\]
Letting $\delta\downarrow0$
and using monotone convergence after subtracting a common lower
bound gives
\[
  \int\sum_{i\ne j}K_1(x_i-x_j)f_N^\tau(t)\,dX_N
   \le\liminf_{\varepsilon\downarrow0}
       \int\sum_{i\ne j}K_{1,\varepsilon}(x_i-x_j)
                      f_{N,\varepsilon}(t)\,dX_N.
\]
The shifts vanish because each $f_{N,\varepsilon}(t)$ has unit
mass.  Combining the three parts of the entropy yields
\[
  \mathcal H(f_N^\tau(t))
      \le\liminf_{\varepsilon\downarrow0}
                     \mathcal H_\varepsilon(f_{N,\varepsilon}(t)).
\]

We also need the dissipation to pass to the limit.  With its standard
extended values at $r=0$, the map
\begin{equation}
  (r,z)\longmapsto
  \begin{cases}
    |z|^2/r,&r>0,\\
    0,&r=0,\ z=0,\\
    +\infty,&r=0,\ z\ne0
  \end{cases}
  \label{eq:quadratic-perspective}
\end{equation}
is convex and lower semicontinuous on
$[0,\infty)\times\mathbb R^d$.  Weak $L^1$ convergence of the pair
$(f_{N,\varepsilon},J_{i,\varepsilon})$ therefore implies, for
each $t\le T$,
\[
  \sum_i\int_0^t\int\frac{|J_i|^2}{f_N^\tau}
           \,dX_N\,ds
   \le\liminf_{\varepsilon\downarrow0}
       \sum_i\int_0^t\int
           \frac{|J_{i,\varepsilon}|^2}{f_{N,\varepsilon}}
                  \,dX_N\,ds.
\]
The entropy and dissipation bounds apply along the same subsequence.
Thus \eqref{eq:appB-regularized-entropy-inequality} and
\eqref{eq:appB-initial-entropy-convergence} give
\eqref{eq:targeted-entropy-inequality} for almost every $t\in[0,T]$.
The narrow convergence at $t=0$, the flux identity, and the
distributional equation verify the remaining conditions in
Definition~\ref{def:targeted-entropy-solution}.

All estimates hold on arbitrary finite intervals.  Choosing nested
subsequences for $T=1,2,\ldots$ and then taking a diagonal
subsequence produces a symmetric global entropy solution.

\medskip
\emph{Identification with the joint law of the particle system.}
Here and below, $\varepsilon$ runs through the diagonal
subsequence chosen above.
For $f_{N,0}=\rho_0^{\otimes N}$, let $X_N$ be the solution in
Proposition~\ref{prop:stochastic-particle-flow}, with initial
positions independent of the Brownian motions.
Let $X_{N,\varepsilon}$ solve the particle SDE corresponding to
$K_{1,\varepsilon}$, with the same initial positions and Brownian
motions.

We regard $X_N,X_{N,\varepsilon}$ as lifts to $\mathbb R^{dN}$
satisfying $X_{N,\varepsilon}(0)=X_N(0)$, and denote their
periodic drifts by $b,b_\varepsilon$.
Since the Brownian motions are the same, their
difference satisfies
\[
  X_{N,\varepsilon}(t)-X_N(t)
  =\int_0^t
    \bigl[b_\varepsilon(X_{N,\varepsilon}(s))
          -b(X_N(s))\bigr]\,ds.
\]

Fix $T<\infty$. Proposition~\ref{prop:stochastic-particle-flow}
and continuity imply
\[
  \min_{0\le t\le T}\min_{i\ne j}
  d_{\T^d}(X_t^i,X_t^j)>0
  \qquad\text{almost surely}.
\]
On each such path, choose $r>0$ so that the closed neighborhood
\[
  U_r:=\{x\in\mathbb R^{dN}:
          \operatorname{dist}(x,X_N([0,T]))\le r\}
\]
has its torus image contained in $\mathcal X_N$.
Then $b$ has bounded first derivatives on the compact set $U_r$.
The locally uniform convergence
$\nabla K_{1,\varepsilon}\to\nabla K_1$ on
$\T^d\setminus\{0\}$ gives $b_\varepsilon\to b$ uniformly
on $U_r$. Until
$|X_{N,\varepsilon}(t)-X_N(t)|$ first reaches $r$,
both $X_N(t)$ and $X_{N,\varepsilon}(t)$, together with the
segment joining them, lie in $U_r$.
The mean-value theorem and Gr\"onwall's inequality give
\[
\begin{aligned}
  |X_{N,\varepsilon}(t)-X_N(t)|
  &\le\int_0^t
       \bigl(\|b_\varepsilon-b\|_{L^\infty(U_r)}
             +C(T,r,X_N)|X_{N,\varepsilon}(s)-X_N(s)|\bigr)\,ds\\
  &\le t e^{C(T,r,X_N)t}\|b_\varepsilon-b\|_{L^\infty(U_r)}.
\end{aligned}
\]
For sufficiently small $\varepsilon$, the last bound is
uniformly less than $r$ on $[0,T]$. By continuity, the
difference cannot reach $r$, so the estimate holds throughout
$[0,T]$. Letting $\varepsilon\to0$ gives
\[
  X_{N,\varepsilon}\longrightarrow X_N
  \quad\text{in }C([0,T];(\T^d)^N),
  \qquad\text{almost surely}.
\]
The time marginals therefore converge narrowly to those of
$X_N$. Comparing this convergence with
\eqref{eq:appB-uniform-narrow-convergence} gives
\[
  f_N^\tau(t)\,dX_N=\operatorname{Law}(X_N(t))
  \qquad\text{for every }t\ge0,
\]
since $T$ was arbitrary.
\end{proof}
\subsection{The singular modulated free-energy inequality}
\label{subsec:singular-targeted-mfe-passage}

The weighted entropy estimate from
Subsection~\ref{subsec:targeted-entropy-compactness} gives a particle law
with an integrable flux, but it does not allow us to differentiate the
modulated free energy.  We prove here that the calculation of
Subsection~\ref{subsec:smooth-weighted-calculation} remains valid as an
inequality for every entropy solution.  The two points to check are that
the time-dependent mean-field weight is an admissible test function
and that the terms containing the Coulomb force can be combined before
the collision cutoff is removed.

\begin{lemma}[The singular weighted comparison]
\label{lem:singular-targeted-weighted-comparison}
Let $N\ge2$ and $\tau>0$.  Let $f_N^\tau$ be any entropy solution of
Definition~\ref{def:targeted-entropy-solution} on $[0,T]$, with weighted
fluxes $J_i$.  Let $\rho=\rho^\tau$ be a positive solution of
\eqref{eq:intro-mean-field-pde} satisfying
\[
  \rho\in C([0,T];W^{2,\infty}(\T^d)),
  \qquad \inf_{[0,T]\times\T^d}\rho>0,
\]
and write $u_t=\tau\nabla\log\rho_t+\nabla K_1*(\rho_t-\pi)$.
Then, for almost every $t\in[0,T]$,
\begin{equation}
\begin{aligned}
  \mathcal E_N^\tau(t)
  &+\frac1N\sum_{i=1}^N\int_0^t\!\int
    \frac{|\tau J_i(s)-f_N^\tau(s)u_s(x_i)|^2}
         {f_N^\tau(s)}\,dX_N\,ds\\
  &\le \mathcal E_N^\tau(0)
    -\frac12\int_0^t
       \mathbb E_{f_N^\tau(s)}
       [\Comm_N[u_s](\,\cdot\,,\rho_s)]\,ds.
\end{aligned}
\end{equation}
The quotient in this display has the convention of
\eqref{eq:quadratic-perspective}.  By
\eqref{eq:weighted-gradient-identities}, this is the flux form of
\eqref{eq:targeted-mfe-inequality}.
\end{lemma}

\begin{proof}
We first combine the entropy inequality with the weak particle
equation.  We then remove a cutoff from the identity that moves the
derivative of $f_N^\tau$ onto $u_t$.  Finally, the same
symmetrization as in the smooth calculation identifies the remaining
term with the commutator.

For every $C^1$ function $\varphi$ of the particle positions, the
distributional equation makes
$t\mapsto\int\varphi f_N^\tau(t)\,dX_N$ absolutely continuous, with
derivative $-\tau\sum_i\int J_i\cdot\nabla_{x_i}\varphi\,dX_N$.
Taking a countable dense family of such functions supplies a narrowly
continuous representative of the particle law.  The initial trace in
Definition~\ref{def:targeted-entropy-solution} identifies its value at
$t=0$.

\medskip
\noindent\emph{Step 1: Test with the mean-field weight.}
Set $f=f_N^\tau$ and define the function used in
\eqref{eq:weighted-entropy-representation} by
\[
  \Phi_t(X_N)
  :=\log\frac{\rho_t^{\otimes N}(X_N)}{G_{\rho,N}(t,X_N)}.
\]
The lower bound on $\rho$ and $K_1\in L^1(\T^d)$ make $\Phi$
bounded.  The equation for $\rho$, together with
$\nabla K_1\in L^1(\T^d)$, gives bounded weak derivatives in time
and in every particle variable.  More precisely,
\[
  \Phi\in W^{1,\infty}((0,T)\times(\T^d)^N),
  \qquad
  \tau\nabla_{x_i}\Phi_t(X_N)=u_t(x_i).
\]
To use $\Phi$ in the weak
equation, we mollify it in space and time, apply
\eqref{eq:singular-targeted-weighted-equation}, and pass to the
limit by bounded convergence against $f$ and $J_i$.  In this way,
\begin{equation}
\begin{aligned}
  \int f(t)\Phi_t\,dX_N-\int f(0)\Phi_0\,dX_N
  ={}&\int_0^t\!\int f(s)\,\partial_s\Phi_s\,dX_N\,ds\\
     &-\sum_i\int_0^t\!\int J_i(s)\cdot u_s(x_i)\,dX_N\,ds.
\end{aligned}
\label{eq:appB-singular-test-identity}
\end{equation}
Here
$\partial_s\Phi_s
 =\sum_i\partial_s\log\rho_s(x_i)
  -\partial_s\log G_{\rho,N}(s)$.
The identity holds for almost every $t$.  Combining it with the entropy
inequality \eqref{eq:targeted-entropy-inequality} and the
representation \eqref{eq:weighted-entropy-representation} gives
\begin{equation}
\begin{aligned}
  \mathcal E_N^\tau(t)
  \le{}&\mathcal E_N^\tau(0)
    -\frac{\tau^2}{N}\sum_i\int_0^t\!\int
      \frac{|J_i(s)|^2}{f(s)}\,dX_N\,ds\\
  &-\frac\tau N\int_0^t\!\int
      f(s)\left(\sum_i\partial_s\log\rho_s(x_i)
                 -\partial_s\log G_{\rho,N}(s)\right)dX_N\,ds\\
  &+\frac\tau N\sum_i\int_0^t\!\int
       J_i(s)\cdot u_s(x_i)\,dX_N\,ds.
\end{aligned}
\label{eq:appB-weighted-comparison}
\end{equation}

\medskip
\noindent\emph{Step 2: Remove the collision cutoff.}
The flux terms in \eqref{eq:appB-weighted-comparison} are
integrable by the entropy dissipation and the boundedness of $u$.
The separate particle forces in
$E_i=N^{-1}\sum_{j\ne i}\nabla K_1(x_i-x_j)
       -\nabla K_1*\rho(x_i)$ need not be integrable against $f$.
Their pairing with $u$, however, has the symmetric form
\begin{equation}
\begin{aligned}
  \sum_i u_s(x_i)\cdot E_i(X_N,\rho_s)
  ={}&\frac1{2N}\sum_{i\ne j}
      (u_s(x_i)-u_s(x_j))\cdot\nabla K_1(x_i-x_j)\\
     &-\sum_i u_s(x_i)\cdot(\nabla K_1*\rho_s)(x_i).
\end{aligned}
\label{eq:appB-symmetrized-particle-field}
\end{equation}
Near the origin the periodic Coulomb kernel satisfies
$|z|\,|\nabla K_1(z)|\le C(1+(K_1(z))_+)$.
Thus the first line of \eqref{eq:appB-symmetrized-particle-field}
is bounded by the Lipschitz seminorm of $u_s$ times a sum of
positive Coulomb energies.  The weighted entropy controls this sum
because $\int f\log f\ge0$, $K_1$ is bounded below, and the target
potential is bounded.  In particular,
\[
  \operatorname*{ess\,sup}_{0<s<T}
     \int f(s)\sum_{i\ne j}(K_1(x_i-x_j))_+\,dX_N<\infty.
\]
The second line of \eqref{eq:appB-symmetrized-particle-field}
is bounded since $\nabla K_1*\rho_s\in L^\infty$.
These observations make the symmetric pairing integrable in
spacetime.  They also make $\Comm_N[u_s](X_N,\rho_s)$ integrable.
For the mixed terms in the commutator, use
$\nabla K_1\in L^1$, $\rho_s\in L^\infty$, and the boundedness of $u_s$.

We now justify the integration by parts used in
\eqref{eq:appB-pointwise-remainder}.  For small $\delta>0$, let
$\chi_\delta$ be the product of smooth cutoffs for the particle
pairs.  Each factor vanishes if its pair distance is below $\delta$
and equals one if that distance exceeds $2\delta$.  Define
\[
  \mathcal C_{2\delta}
  :=\{X_N:\min_{i<j}d_{\T^d}(x_i,x_j)<2\delta\}.
\]
Figure~\ref{fig:singular-pair-cutoff} shows where one pair factor
changes and where the two particle derivatives must be combined.

\begin{figure}[htbp]
\centering
\begin{tikzpicture}[x=1.15cm,y=1cm]
  \fill[gray!22] (2,0.08) rectangle (4,0.45);
  \draw[->] (0,0) -- (6.3,0) node[right] {$r$};
  \draw (2,-0.12) -- (2,0.5);
  \draw (4,-0.12) -- (4,0.5);
  \node[below] at (0,0) {$0$};
  \node[below] at (2,0) {$\delta$};
  \node[below] at (4,0) {$2\delta$};
  \node[above] at (1,0.5) {$\chi_\delta^{(i,j)}=0$};
  \node[above] at (3,0.5) {$0<\chi_\delta^{(i,j)}<1$};
  \node[above] at (5,0.5) {$\chi_\delta^{(i,j)}=1$};
\end{tikzpicture}
\caption{One pair factor of the collision cutoff, shown against the
distance $r=d_{\T^d}(x_i,x_j)$ near a collision.  Its derivative is
supported in the shaded interval $\delta<r<2\delta$ and is of order
$\delta^{-1}$.  The derivatives in $x_i$ and $x_j$ have opposite
signs, so summing them produces $u(x_i)-u(x_j)$.  The Lipschitz bound
on this difference is of order $\delta$, leaving a bounded cutoff
error.}
\label{fig:singular-pair-cutoff}
\end{figure}
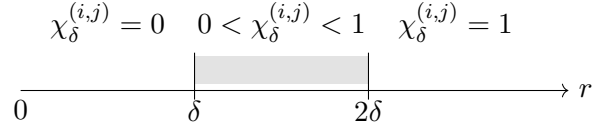

Off the collision set the definition of the weighted flux gives,
in distributions,
\[
  \tau J_i-fu_s(x_i)
  =\tau\bigl(\nabla_{x_i}f
       -f\nabla\log\rho_s(x_i)\bigr)
       +fE_i(X_N,\rho_s).
\]
Multiply this identity by $\chi_\delta u_s(x_i)$, sum over $i$,
and integrate.  Since
$\partial_s\log\rho_s=\nabla\cdot u_s
                 +u_s\cdot\nabla\log\rho_s$,
the derivative of $f$ cancels the time derivative of $\log\rho_s$.
The only additional term is
$-\tau\int f\sum_i u_s(x_i)\cdot\nabla_{x_i}\chi_\delta$.
For each differentiated pair factor the two particle derivatives
combine as $u_s(x_i)-u_s(x_j)$.  The derivative is supported where
the pair distance is of order $\delta$, so
\[
  \left|\sum_i u_s(x_i)\cdot\nabla_{x_i}\chi_\delta\right|
  \le C_N\|\nabla u_s\|_{L^\infty}
       \mathbf1_{\mathcal C_{2\delta}}.
\]
The collision set has Lebesgue measure zero.  Since $f$ is an
$L^1$ density and $\|\nabla u_s\|_{L^\infty}$ is bounded on
$[0,T]$, dominated convergence gives
\[
\begin{aligned}
  &\tau\left|\int_0^t\!\int f(s)
     \sum_i u_s(x_i)\cdot\nabla_{x_i}\chi_\delta\,dX_N\,ds\right|\\
  &\qquad\le C_N\tau\int_0^t\!\|\nabla u_s\|_{L^\infty}
     \left(\int_{\mathcal C_{2\delta}}f(s)\,dX_N\right)ds
  \longrightarrow0.
\end{aligned}
\]
The integrability established above also lets us remove
$\chi_\delta$ from the other terms.  We obtain
\begin{equation}
\begin{aligned}
  &\sum_i\int_0^t\!\int u_s(x_i)\cdot
    \bigl(\tau J_i(s)-f(s)u_s(x_i)\bigr)\,dX_N\,ds\\
  &\quad+\tau\int_0^t\!\int f(s)
        \sum_i\partial_s\log\rho_s(x_i)\,dX_N\,ds\\
  &=\int_0^t\!\int f(s)\sum_i
       u_s(x_i)\cdot E_i(X_N,\rho_s)\,dX_N\,ds.
\end{aligned}
\label{eq:appB-flux-integration-by-parts}
\end{equation}

\medskip
\noindent\emph{Step 3: Complete the square.}
The quadratic perspective \eqref{eq:quadratic-perspective} and
boundedness of $u$ justify the identity
\[
  -\tau^2\frac{|J_i|^2}{f}+\tau J_i\cdot u_s(x_i)
  =-\frac{|\tau J_i-fu_s(x_i)|^2}{f}
    -u_s(x_i)\cdot\bigl(\tau J_i-fu_s(x_i)\bigr)
\]
almost everywhere, with both sides integrated against $dX_N\,ds$.
Substituting it into \eqref{eq:appB-weighted-comparison} and then
using \eqref{eq:appB-flux-integration-by-parts} leaves
\[
  \int_0^t\!\int f(s)\left[
    \frac\tau N\partial_s\log G_{\rho,N}(s)
    -\frac1N\sum_i u_s(x_i)\cdot E_i(X_N,\rho_s)
  \right]dX_N\,ds.
\]
For a collision-free configuration, the calculation in
\eqref{eq:appB-pointwise-commutator} identifies the expression
in brackets with $-\tfrac12\Comm_N[u_s](X_N,\rho_s)$.
The previous integrability bounds permit this identity to be
integrated against $f$.  This proves the displayed inequality in
the lemma.

Finally, the identity
$\tau J_i-fu_s(x_i)
 =f\bigl[\tau\nabla_{x_i}\log(f/\rho_s^{\otimes N})
                  +E_i(X_N,\rho_s)\bigr]$
holds on $\{f>0\}\cap\mathcal X_N$ by
\eqref{eq:weighted-gradient-identities}.  The flux quotient in
the lemma therefore agrees with the dissipation in
\eqref{eq:targeted-mfe-inequality} under the convention
\eqref{eq:quadratic-perspective}.  The collision set has Lebesgue
measure zero, while finite entropy dissipation implies $J_i=0$
almost everywhere on $\{f=0\}$.
\end{proof}

Lemma~\ref{lem:singular-targeted-weighted-comparison} proves
Proposition~\ref{prop:targeted-mfe}\textup{(ii)} for every entropy
solution, including solutions other than the one constructed in
Subsection~\ref{subsec:targeted-entropy-compactness}.

\section{Stationary estimates and low-temperature regularity}
\label{app:maximum-principle}
\label{app:low-temperature-regularity}

This appendix proves the stationary and low-temperature estimates used
in Section~\ref{sec:mean-field-analysis}.  We first construct the thermal
equilibrium and bound its derivatives in
Propositions~\ref{prop:thermal-equilibrium-family}
and~\ref{prop:stationary-all-orders}.  We then prove the interpolation,
Coulomb, and maximum-principle estimates behind
Propositions~\ref{prop:uniform-lipschitz}
and~\ref{prop:low-temperature-relaxation}.  Finally, a compactness
argument transfers the bounds to the deterministic solution in
Proposition~\ref{prop:tau0-endpoint}.

\subsection{Stationary equilibria: construction and regularity}
\label{subsec:stationary-estimates}

The proofs of Propositions~\ref{prop:thermal-equilibrium-family}
and~\ref{prop:stationary-all-orders} are given here.  The first
constructs the density about which the mean-field equation is recentered.
The second bounds its derivatives without a constant that diverges as
$\tau\downarrow0$.  These bounds are used in the low-temperature
estimates of Section~\ref{sec:mean-field-analysis}.

\begin{proof}[Proof of Proposition~\ref{prop:thermal-equilibrium-family}]
Fix $\tau>0$.  We seek a density of the form $\rho_\infty^\tau=e^{q_\tau}$.
The Euler--Lagrange identity~\eqref{eq:stationary-euler-lagrange}
reduces, after applying $-\Delta$, to a local equation for $q_\tau$.
We first solve this equation by minimizing
\[
  J_\tau(q)
  :=\frac\tau2\int_{\T^d}|\nabla q|^2\,dx
    +\int_{\T^d}e^q\,dx
    -\int_{\T^d}\pi q\,dx,
  \qquad q\in H^1(\T^d),
\]
with $J_\tau(q)=+\infty$ if $e^q\notin L^1(\T^d)$.
The exponential makes this functional strictly convex.  We establish
existence of a minimizer, bound it between $\log c_\pi$ and
$\log C_\pi$, and then recover~\eqref{eq:stationary-euler-lagrange}.

\medskip
\noindent\emph{Step 1: existence of a minimizer.}
Write $q=m+h$, where $m:=\int_{\T^d}q\,dx$ and $\int h\,dx=0$.
Since $\int\pi=1$, Jensen's inequality and the Poincar\'e inequality
give
\begin{equation}
  J_\tau(q)
  \ge \frac\tau2\|\nabla q\|_{L^2}^2
       -C\|\nabla q\|_{L^2}+e^m-m.
  \label{eq:stationary-J-coercivity}
\end{equation}
For fixed $\tau>0$, this lower bound tends to infinity whenever
$\|q\|_{H^1}$ does.  A minimizing sequence therefore has a subsequence
that converges weakly in $H^1$, strongly in $L^2$, and almost
everywhere to some $q_\tau\in H^1$.  Fatou's lemma applies to the
exponential, while weak lower semicontinuity applies to the gradient.
The linear term converges by strong $L^2$ convergence, so
\[
  J_\tau(q_\tau)\le\liminf_{n\to\infty}J_\tau(q_n).
\]
Thus $q_\tau$ minimizes $J_\tau$.  For
$\varphi\in H^1(\T^d)\cap L^\infty(\T^d)$, the perturbations
$q_\tau+s\varphi$ have integrable exponentials for $s$ near zero.
Their first variation gives
\begin{equation}
  \tau\int_{\T^d}\nabla q_\tau\cdot\nabla\varphi\,dx
   +\int_{\T^d}(e^{q_\tau}-\pi)\varphi\,dx=0,
  \qquad
  -\tau\Delta q_\tau+e^{q_\tau}=\pi
       \quad\text{in }\mathcal D'(\T^d).
  \label{eq:stationary-log-equation}
\end{equation}

\medskip
\noindent\emph{Step 2: pointwise bounds and mass.}
The weak equation permits bounded $H^1$ test functions, even though
$q_\tau$ is not yet known to be bounded.  Put $a=\log C_\pi$ and
test with $\varphi=\min\{(q_\tau-a)_+,M\}$, where $M>0$.
Both terms in
\begin{equation}
  0=\tau\int_{\{a<q_\tau<a+M\}}|\nabla q_\tau|^2\,dx
     +\int_{\T^d}(e^{q_\tau}-\pi)
               \min\{(q_\tau-a)_+,M\}\,dx
  \label{eq:stationary-upper-truncation}
\end{equation}
are nonnegative.  Since $e^{q_\tau}>C_\pi\ge\pi$ wherever
$q_\tau>a$, the second term can vanish only if
$(q_\tau-a)_+=0$ almost everywhere.  For the lower bound, take
$b=\log c_\pi$ and test with
$\varphi=\min\{(b-q_\tau)_+,M\}$.  On $\{q_\tau<b\}$, both
the gradient term and $(e^{q_\tau}-\pi)\varphi$ are nonpositive.
The same argument yields
\begin{equation}
  \log c_\pi\le q_\tau\le\log C_\pi,
  \qquad c_\pi\le e^{q_\tau}\le C_\pi
  \quad\text{almost everywhere}.
  \label{eq:stationary-q-bounds}
\end{equation}
Taking $\varphi\equiv1$ in~\eqref{eq:stationary-log-equation}
shows that $\rho_\infty^\tau:=e^{q_\tau}$ has mass one.

\medskip
\noindent\emph{Step 3: stationary identity and uniqueness.}
Because $\rho_\infty^\tau-\pi$ has zero mean and
$-\Delta K_1=\delta_0-1$, equation~\eqref{eq:stationary-log-equation}
implies
\[
  -\Delta\bigl[\tau q_\tau
       +K_1*(\rho_\infty^\tau-\pi)\bigr]
  =-\tau\Delta q_\tau+\rho_\infty^\tau-\pi=0.
\]
The expression in brackets is constant on $\T^d$.  This is
\eqref{eq:stationary-euler-lagrange} for a constant $\lambda_\tau$.
If $q_1,q_2$ are two solutions of~\eqref{eq:stationary-log-equation},
their bounds~\eqref{eq:stationary-q-bounds} allow us to test the
difference equation with $q_1-q_2$.  We obtain
\[
  0=\tau\|\nabla(q_1-q_2)\|_{L^2}^2
    +\int_{\T^d}(e^{q_1}-e^{q_2})(q_1-q_2)\,dx.
\]
Both terms are nonnegative and the second is zero only when
$q_1=q_2$ almost everywhere.  Hence $\rho_\infty^\tau$ is
unique, and its identity~\eqref{eq:stationary-euler-lagrange}
determines $\lambda_\tau$ uniquely.
\end{proof}

\begin{proof}[Proof of Proposition~\ref{prop:stationary-all-orders}]
Set $q_\tau=\log\rho_\infty^\tau$ for $\tau>0$.  We want to
bound $q_\tau$ and $e^{q_\tau}$ through four spatial derivatives
with constants independent of $\tau$.  We first establish enough
regularity to differentiate~\eqref{eq:stationary-log-equation} for
each fixed temperature.  We then use its positive zeroth-order
coefficient, bounded below by~\eqref{eq:stationary-q-bounds}, to
obtain uniform estimates.

\medskip
\noindent\emph{Step 1: differentiability at fixed temperature.}
Choose $p>\max\{d,2\}$ and write the logarithmic equation as
\[
  (1-\Delta)q_\tau
  =q_\tau+\frac{\pi-e^{q_\tau}}\tau.
\]
The right-hand side is bounded by~\eqref{eq:stationary-q-bounds}.
Periodic elliptic $L^p$ regularity therefore gives
$q_\tau\in W^{2,p}(\T^d)$
\cite[Section~3.5.4, Eqs.~(14) and~(24)]{SchmeisserTriebel1987}.
The Sobolev chain rule, the algebra property of $W^{2,p}$, and
$\pi\in C^4$ permit two further applications of the same estimate:
\begin{equation}
  q_\tau\in W^{2,p}(\T^d)
  \longrightarrow W^{4,p}(\T^d)
  \longrightarrow W^{6,p}(\T^d)
  \hookrightarrow C^{5,\beta}(\T^d),
  \qquad \beta:=1-\frac dp>0.
  \label{eq:stationary-fixed-tau-bootstrap}
\end{equation}
The constants in this step may depend on $\tau$.  Its purpose is
to justify the differentiated weak equations in the next step.

\medskip
\noindent\emph{Step 2: a temperature-independent resolvent bound.}
Write $L_\tau w:=\rho_\infty^\tau w-\tau\Delta w$.  If
$w\in H^1(\T^d)$ solves $L_\tau w=f$ weakly with $f\in L^\infty$,
the bound $\rho_\infty^\tau\ge c_\pi$ gives
\begin{equation}
  \|w\|_{L^\infty}
  \le c_\pi^{-1}\|f\|_{L^\infty}.
  \label{eq:app-stationary-resolvent}
\end{equation}
Indeed, put $A=c_\pi^{-1}\|f\|_{L^\infty}$ and use
$(w-A)_+$ as a test function.  On its support
$\rho_\infty^\tau w-f\ge c_\pi(w-A)$, so
\[
  0
   =\tau\int_{\T^d}|\nabla(w-A)_+|^2\,dx
      +\int_{\T^d}(\rho_\infty^\tau w-f)(w-A)_+\,dx
   \ge c_\pi\int_{\T^d}(w-A)_+^2\,dx.
\]
Thus $w\le A$.  Applying the same test to $-w$ gives
$w\ge-A$ and proves~\eqref{eq:app-stationary-resolvent}.

\medskip
\noindent\emph{Step 3: induction over derivative order.}
For a multi-index $\alpha$ with $|\alpha|=k\ge1$, the chain
rule separates the only term containing $D^kq_\tau$:
\[
  \partial^\alpha e^{q_\tau}
   =\rho_\infty^\tau\bigl(\partial^\alpha q_\tau
      +\mathcal P_\alpha(Dq_\tau,\ldots,D^{k-1}q_\tau)\bigr).
\]
Here $\mathcal P_\alpha$ is a finite sum of products of derivatives
of order at most $k-1$.  It is zero for $k=1$.  Differentiating
\eqref{eq:stationary-log-equation} now gives the weak equation
\begin{equation}
  L_\tau\partial^\alpha q_\tau
  =\partial^\alpha\pi
    -\rho_\infty^\tau
      \mathcal P_\alpha(Dq_\tau,\ldots,D^{k-1}q_\tau),
  \qquad 1\le k\le4.
  \label{eq:app-stationary-derivative-equation}
\end{equation}
The regularity in~\eqref{eq:stationary-fixed-tau-bootstrap} ensures
that $\partial^\alpha q_\tau\in H^1$, so the weak resolvent
bound~\eqref{eq:app-stationary-resolvent} applies.  When $k=1$,
the polynomial vanishes, and we obtain
\[
  \|Dq_\tau\|_{L^\infty}
  \le Cc_\pi^{-1}\|D\pi\|_{L^\infty}.
\]
If derivatives through order $k-1$ have been bounded uniformly,
then the right-hand side of~\eqref{eq:app-stationary-derivative-equation}
is bounded independently of $\tau$.  Using
$\rho_\infty^\tau\le C_\pi$ and
\eqref{eq:app-stationary-resolvent} closes the induction through
$k=4$:
\[
  \sup_{\tau>0}\|q_\tau\|_{C^4}<\infty,
  \qquad
  \sup_{\tau>0}\|e^{q_\tau}\|_{C^4}<\infty.
\]
The second estimate follows from the first by the chain rule and
\eqref{eq:stationary-q-bounds}.  At $\tau=0$, the convention
$\rho_\infty^0=\pi$ and Assumption~\ref{assump:standing-data}
give the same estimates.  This proves
\eqref{eq:uniform-stationary-regularity}.
\end{proof}
\subsection{Interpolation and periodic elliptic estimates}
\label{subsec:appC-interpolation}

The low-temperature proof starts with decay in weak norms and then
needs estimates for derivatives of the density error.  The first
three lemmas below pass between bounded derivatives and H\"older
norms.  The last lemma converts H\"older control of a mean-zero
density into control of its Coulomb potential.  They are used in the
maximum-principle arguments below and in
Appendix~\ref{app:positive-temperature-wellposedness}.

\begin{lemma}[Landau--Kolmogorov inequality]
\label{lem:appC-landau-kolmogorov}
There is a constant $C_d>0$ such that every
$g\in C^2(\T^d)$ satisfies
\begin{equation}
  \|\nabla g\|_{L^\infty}^2
  \le C_d\|g\|_{L^\infty}
          \|\nabla^2g\|_{L^\infty}.
  \label{eq:appC-landau-kolmogorov}
\end{equation}
\end{lemma}

\begin{proof}
We apply the one-dimensional Landau--Kolmogorov inequality
\cite[Satz~1 and p.~44]{Landau1913} to the periodic extension of
$g$ along each coordinate line.  It gives
\[
  \|\partial_jg\|_{L^\infty}^2
  \le C\|g\|_{L^\infty}
         \|\partial_{jj}g\|_{L^\infty},
  \qquad 1\le j\le d.
\]
Taking the maximum over $j$ and comparing the component norms
with the full gradient and Hessian norms proves
\eqref{eq:appC-landau-kolmogorov}.
\end{proof}

\begin{lemma}[Periodic interpolation]
\label{lem:appC-ehrling}
Let $\alpha\in(0,1)$.  For every $\varepsilon>0$, there is a
constant $C_{d,\alpha,\varepsilon}>0$ such that
\begin{equation}
  \|\partial_jg\|_{C^{0,\alpha}}
  \le\varepsilon\|\nabla^2g\|_{L^\infty}
     +C_{d,\alpha,\varepsilon}\|g\|_{L^\infty},
  \qquad g\in C^2(\T^d),\quad 1\le j\le d.
  \label{eq:appC-periodic-interpolation}
\end{equation}
\end{lemma}

\begin{proof}
We split the H\"older quotient of $\partial_jg$ into pairs of
points closer than $\delta$ and pairs farther apart, where
$0<\delta<1/2$ will be chosen.  The mean value theorem controls
the close pairs by the Hessian, while the far pairs are controlled
by the supremum norm of $\partial_jg$.  Consequently,
\[
  \|\partial_jg\|_{C^{0,\alpha}}
  \le C\delta^{1-\alpha}\|\nabla^2g\|_{L^\infty}
     +(1+2\delta^{-\alpha})
       \|\partial_jg\|_{L^\infty}.
\]
Lemma~\ref{lem:appC-landau-kolmogorov} and Young's inequality give,
for any $\eta>0$,
\[
  \|\partial_jg\|_{L^\infty}
  \le\eta\|\nabla^2g\|_{L^\infty}
     +C_d\eta^{-1}\|g\|_{L^\infty}.
\]
Choose $\delta$ so that $C\delta^{1-\alpha}\le\varepsilon/2$.
Then choose $\eta$ so that
$(1+2\delta^{-\alpha})\eta\le\varepsilon/2$.
Substitution gives \eqref{eq:appC-periodic-interpolation}.
\end{proof}

\begin{lemma}[H\"older interpolation]
\label{lem:appC-holder-interpolation}
Let $\alpha\in(0,1)$.  Every Lipschitz function
$g:\T^d\to\mathbb R$ satisfies
\begin{equation}
  [g]_{C^{0,\alpha}}
  \le 2^{1-\alpha}
     \|g\|_{L^\infty}^{1-\alpha}
     \|\nabla g\|_{L^\infty}^{\alpha}.
  \label{eq:appC-holder-interpolation}
\end{equation}
\end{lemma}

\begin{proof}
Put $A:=\|g\|_{L^\infty}$ and
$L:=\|\nabla g\|_{L^\infty}$.  If $A=0$ or $L=0$, then
$g$ is constant and the seminorm vanishes.  Otherwise, for
$r=d_{\T^d}(x,y)>0$, the boundedness and the Lipschitz property
give
\[
  \frac{|g(x)-g(y)|}{r^\alpha}
  \le\min\{2Ar^{-\alpha},Lr^{1-\alpha}\}
  \le(2A)^{1-\alpha}L^\alpha.
\]
Taking the supremum over $x\ne y$ proves
\eqref{eq:appC-holder-interpolation}.
\end{proof}

\begin{lemma}[Periodic Schauder estimate]
\label{lem:periodic-schauder}
For $\alpha\in(0,1)$, let
$h\in C^{0,\alpha}(\T^d)$ have zero mean.  Then
\begin{equation}
  \|K_1*h\|_{C^{2,\alpha}}
  \le C_{d,\alpha}\|h\|_{C^{0,\alpha}}.
  \label{eq:appC-schauder}
\end{equation}
In particular,
$\|\nabla^2K_1*h\|_{L^\infty}
 \le C_{d,\alpha}\|h\|_{C^{0,\alpha}}$.
\end{lemma}

\begin{proof}
Set $\varphi:=K_1*h$.  Since $h$ has zero mean and
$(-\Delta)K_1=\delta_0-1$, the function $\varphi$ is the
mean-zero solution of
\[
  -\Delta\varphi=h
  \quad\text{on }\T^d.
\]
The periodic Schauder estimate, for example
\cite[Lemma~2.21]{BoufadeneVialard2025}, bounds its full H\"older
norm by
\[
  \|\varphi\|_{C^{2,\alpha}}
  \le C_{d,\alpha}
      \bigl(\|h\|_{C^{0,\alpha}}
            +\|\varphi\|_{L^\infty}\bigr).
\]
Finally, $K_1\in L^1(\T^d)$ and Young's inequality give
$\|\varphi\|_{L^\infty}
 \le\|K_1\|_{L^1}\|h\|_{L^\infty}$.
Substituting this bound proves \eqref{eq:appC-schauder}.
\end{proof}
\subsection{Coulomb-field estimates}
\label{subsec:appC-coulomb-field-estimates}

The low-temperature argument needs two estimates for the Coulomb
potential generated by a mean-zero density error $f$.  The first uses
its $\dot H^{-1}$ norm to control the Coulomb field $\nabla K_1*f$.
The second controls the Hessian $\nabla^2K_1*f$ when one derivative of
$f$ is available.  We prove both estimates by splitting the kernel at
a radius $r$, then choose $r$ in the proofs of
Lemmas~\ref{lem:potential-estimate} and~\ref{lem:yudovich}.

\begin{proposition}[Coulomb estimates at a fixed scale]
\label{prop:appC-coulomb-scale-estimates}
Let $0<r\le 1/8$, and let $f$ have zero mean on $\T^d$.
If $f\in L^\infty(\T^d)\cap\dot H^{-1}(\T^d)$, then
\begin{equation}
  \|\nabla K_1*f\|_{L^\infty}
  \le C\bigl(r\|f\|_{L^\infty}
              +r^{-d/2}\|f\|_{\dot H^{-1}}\bigr).
  \label{eq:appC-field-fixed-scale}
\end{equation}
If, in addition, $f\in W^{1,\infty}(\T^d)$, then
\begin{equation}
  \|\nabla^2K_1*f\|_{L^\infty}
  \le C\left[
     \bigl(1+\log(1/r)\bigr)\|f\|_{L^\infty}
     +r\|\nabla f\|_{L^\infty}\right].
  \label{eq:appC-hessian-fixed-scale}
\end{equation}
Here $C$ depends only on $d$.
\end{proposition}

\begin{figure}[t]
\centering
\begin{tikzpicture}[scale=0.95]
  \fill[gray!8] (-3.7,-2) rectangle (3.7,2);
  \fill[gray!18] (0,0) circle (1.5);
  \fill[gray!38] (0,0) circle (0.78);
  \draw (-3.7,-2) rectangle (3.7,2);
  \draw (0,0) circle (0.78);
  \draw (0,0) circle (1.5);
  \fill (0,0) circle (1.5pt);
  \node[above left] at (0,0) {$0$};
  \node at (0,-0.43) {$\chi_r=1$};
  \node at (2.45,0) {$1-\chi_r$};
  \node[above right] at (0.54,0.54) {$r$};
  \node[above right] at (1.06,1.06) {$2r$};
\end{tikzpicture}
\caption{The near and far splitting around the Coulomb singularity at
$0$, shown in a local chart of $\T^d$.  The cutoff $\chi_r$ equals one
inside radius $r$, varies between radii $r$ and $2r$, and vanishes
outside radius $2r$.  The near part is controlled by the local kernel
integral, while the far part is controlled by $\dot H^{-1}$ duality for
the field and by a logarithmic integral for the Hessian.}
\label{fig:appC-near-far-split}
\end{figure}
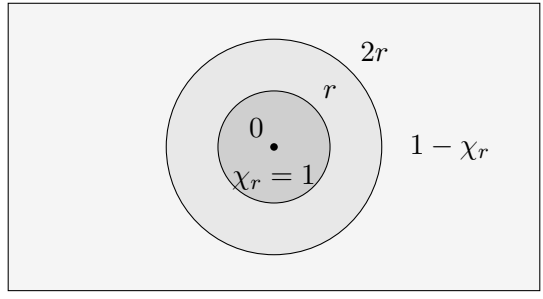

\begin{proof}
We first prove \eqref{eq:appC-field-fixed-scale}.  Choose a smooth
cutoff $\chi_r$ that equals one on the ball of radius $r$, vanishes
outside the ball of radius $2r$, and satisfies
$0\le\chi_r\le1$ and $\|\nabla\chi_r\|_{L^\infty}\le C/r$.
Figure~\ref{fig:appC-near-far-split} shows the regions separated by
this cutoff.
The Coulomb kernel satisfies
\[
  |\nabla K_1(z)|\le C|z|^{1-d},\qquad
  |\nabla^2K_1(z)|\le C|z|^{-d},
  \qquad 0<|z|\le\tfrac14.
\]
Thus the part of $\nabla K_1*f$ supported within distance $2r$
of the origin obeys
\[
  \|(\chi_r\nabla K_1)*f\|_{L^\infty}
  \le C\|f\|_{L^\infty}
        \int_0^{2r}s^{1-d}s^{d-1}\,ds
  \le Cr\|f\|_{L^\infty}.
\]

For the remaining part, set $G_r=(1-\chi_r)\nabla K_1$.
Since $f$ has zero mean, subtracting the spatial mean of each
component of $G_r$ does not change $G_r*f$.  Duality between
$\dot H^{-1}$ and $\dot H^1$, applied to every translate of $G_r$,
therefore gives
\[
  \|G_r*f\|_{L^\infty}
  \le \|f\|_{\dot H^{-1}}\|\nabla G_r\|_{L^2}.
\]
The derivative of $G_r$ contains both the Hessian of $K_1$ outside
the cutoff and $\nabla\chi_r\otimes\nabla K_1$ in the annulus
$r\le |z|\le2r$.  The preceding kernel bounds give
\[
\begin{aligned}
  \|\nabla G_r\|_{L^2}^2
  &\le C\left(1+\int_r^{1/4}s^{-2d}s^{d-1}\,ds
       +r^{-2}\int_r^{2r}s^{2(1-d)}s^{d-1}\,ds\right)\\
  &\le Cr^{-d}.
\end{aligned}
\]
The term $1$ accounts for the smooth part of the kernel away from
the origin.  Combining the near and far bounds proves
\eqref{eq:appC-field-fixed-scale}.

We next prove \eqref{eq:appC-hessian-fixed-scale}.  For a smooth $f$,
the distributional Hessian of $K_1$ can be decomposed as
\[
\begin{aligned}
  \nabla^2K_1*f(x)
  ={}&\int_{\T^d}\chi_r(z)\nabla^2K_1(z)
                      [f(x-z)-f(x)]\,dz\\
     &+f(x)\langle\nabla^2K_1,\chi_r\rangle
       +\int_{\T^d}(1-\chi_r(z))\nabla^2K_1(z)f(x-z)\,dz.
\end{aligned}
\]
The first integral is absolutely convergent because the difference
of the values of $f$ cancels one power of the singularity.  Its
absolute value is at most
\[
  C\|\nabla f\|_{L^\infty}
    \int_0^{2r}s^{-d}s\,s^{d-1}\,ds
  \le Cr\|\nabla f\|_{L^\infty}.
\]
The distributional term is uniformly bounded.  Indeed, integration
by parts and the support of $\nabla\chi_r$ give
\[
  |\langle\nabla^2K_1,\chi_r\rangle|
  =\left|\int_{\T^d}\nabla K_1\otimes\nabla\chi_r\,dz\right|
  \le Cr^{-1}\int_r^{2r}s^{1-d}s^{d-1}\,ds
  \le C.
\]
Finally, integration of the Hessian outside the cutoff gives
\[
  \int_{\T^d}(1-\chi_r)|\nabla^2K_1|\,dz
  \le C\left(1+\int_r^{1/4}s^{-d}s^{d-1}\,ds\right)
  \le C\bigl(1+\log(1/r)\bigr).
\]
These three bounds prove \eqref{eq:appC-hessian-fixed-scale} for
smooth $f$.  Periodic mollification gives the same bound for
$f\in W^{1,\infty}(\T^d)$.
\end{proof}

\begin{proof}[Proof of Lemma~\ref{lem:potential-estimate}]
The fixed-scale estimate \eqref{eq:appC-field-fixed-scale} uses the
near and far splitting of
\cite[Lemma~A.4(i)]{ChizatColomboColomboFernandezReal2026}.
We choose its radius so that the near term, which increases with $r$,
and the far term, which decreases with $r$, are comparable.  The
claim is immediate for $f=0$.  Otherwise, set
\[
  A:=\|f\|_{L^\infty},\qquad
  B:=\|f\|_{\dot H^{-1}},\qquad
  r:=\min\left\{\frac18,\left(\frac BA\right)^{2/(d+2)}\right\}.
\]
When $r=(B/A)^{2/(d+2)}$, both terms in
\eqref{eq:appC-field-fixed-scale} equal
$A^{d/(d+2)}B^{2/(d+2)}$.  Otherwise $B\ge 8^{-(d+2)/2}A$, so
\eqref{eq:appC-field-fixed-scale} is bounded by $CB$.
Together these cases give \eqref{eq:potential-estimate}.
\end{proof}

\begin{proof}[Proof of Lemma~\ref{lem:yudovich}]
The fixed-scale Hessian estimate \eqref{eq:appC-hessian-fixed-scale}
is a direct version of the logarithmic Newtonian estimate in
\cite[Lemma~A.5(i)]{ChizatColomboColomboFernandezReal2026}.
If $f=0$, both sides of \eqref{eq:yudovich-newtonian-estimate}
vanish.  For a nonzero mean-zero function $f$, let
\[
  A:=\|f\|_{L^\infty}>0,\qquad
  L:=\|\nabla f\|_{L^\infty}>0,\qquad
  r:=\min\left\{\frac18,\frac AL\right\}.
\]
This choice ensures $rL\le A$.  If $r=A/L$, then
$\log(1/r)=\log(L/A)$.  If $r=1/8$, then $L/A\le8$ and the
logarithm is bounded by $\log8$.  Substitution into
\eqref{eq:appC-hessian-fixed-scale} gives
\[
  \|\nabla^2K_1*f\|_{L^\infty}
  \le CA\left[1+\log_+\!\left(\frac LA\right)\right],
\]
as claimed in \eqref{eq:yudovich-newtonian-estimate}.
\end{proof}
\subsection{The error equation and Hamilton's maximum lemma}
\label{subsec:appC-centered-zero-order}

To estimate the relaxation toward equilibrium, we first write an equation
for the density difference.  We then give a maximum lemma that turns this
equation and its spatial derivatives into differential inequalities for
their supremum norms.  These two ingredients are used in the estimates of
Subsections~\ref{subsec:appC-first-order} and
\ref{subsec:appC-second-order}.

Fix $\tau>0$, and let $\rho_\infty^\tau$ be the stationary density
from Proposition~\ref{prop:thermal-equilibrium-family}.  On each finite
time interval $[0,T]$, the global regularity proved in
Subsection~\ref{subsec:appD-global-regularity} gives
\[
  \rho^\tau\in C([0,T];C^4(\T^d))
  \cap C^1([0,T];C^2(\T^d)).
\]
We write $\widetilde\rho_t^\tau:=\rho_t^\tau-\rho_\infty^\tau$
and, for $r=0,1,2$, set
\[
  M_r^\tau(t)
  :=\max_{|\beta|=r}
       \|\partial^\beta\widetilde\rho_t^\tau\|_{L^\infty}.
\]
Thus $M_r^\tau$ is the largest supremum norm among the derivatives
of order $r$, and
$\|\nabla^r\widetilde\rho_t^\tau\|_{L^\infty}
\le C_{d,r}M_r^\tau(t)$.

\begin{lemma}[Equation for the density difference]
\label{lem:appC-centered-error-equation}
For every $\tau>0$, the function
$\widetilde\rho^\tau=\rho^\tau-\rho_\infty^\tau$ satisfies
\begin{equation}
\begin{aligned}
  \partial_t\widetilde\rho^\tau
  &={}
  \tau\Delta\widetilde\rho^\tau
  +\nabla\cdot\left(
     \rho^\tau\nabla K_1*\widetilde\rho^\tau
     -\tau\widetilde\rho^\tau\nabla\log\rho_\infty^\tau
  \right)\\
  &={}
  \tau\Delta\widetilde\rho^\tau
  +\nabla\rho^\tau\cdot(\nabla K_1*\widetilde\rho^\tau)
  -\tau\nabla\log\rho_\infty^\tau\cdot
        \nabla\widetilde\rho^\tau
  -\rho^\tau\widetilde\rho^\tau
  -(\rho_\infty^\tau-\pi)\widetilde\rho^\tau.
\end{aligned}
\label{eq:appC-error-equation}
\end{equation}
In particular, the term $-\rho^\tau\widetilde\rho^\tau$ has a
coefficient bounded below by $c_*$, as given by
Proposition~\ref{prop:maximum-principle}.
\end{lemma}

\begin{proof}
We subtract the stationary form of the mean-field equation
\eqref{eq:intro-mean-field-pde} from the equation for $\rho^\tau$.
The stationary Euler--Lagrange equation
\eqref{eq:stationary-euler-lagrange} gives
\[
  \nabla K_1*(\rho_\infty^\tau-\pi)
  =-\tau\nabla\log\rho_\infty^\tau.
\]
Since $\rho^\tau=\rho_\infty^\tau+\widetilde\rho^\tau$, the
difference of the two fluxes is therefore
\[
\begin{aligned}
  \rho^\tau\nabla K_1*(\rho^\tau-\pi)
   -\rho_\infty^\tau\nabla K_1*(\rho_\infty^\tau-\pi)
  &=\rho^\tau\nabla K_1*\widetilde\rho^\tau
    -\tau\widetilde\rho^\tau\nabla\log\rho_\infty^\tau.
\end{aligned}
\]
Taking its divergence yields the first line of
\eqref{eq:appC-error-equation}.  To expand it, we use the fact that
$\int_{\T^d}\widetilde\rho^\tau=0$ and
$-\Delta K_1=\delta_0-1$.  The resulting Coulomb contribution is
\[
  \nabla\cdot(\rho^\tau\nabla K_1*\widetilde\rho^\tau)
  =\nabla\rho^\tau\cdot(\nabla K_1*\widetilde\rho^\tau)
   -\rho^\tau\widetilde\rho^\tau.
\]
Finally, $\rho_\infty^\tau-\pi
=\tau\Delta\log\rho_\infty^\tau$ by
\eqref{eq:stationary-log-equation}.  Expanding the remaining divergence
gives the second line of \eqref{eq:appC-error-equation}.
The form $-(\rho_\infty^\tau-\pi)\widetilde\rho^\tau$ is useful
at order zero, while
$-\tau(\Delta\log\rho_\infty^\tau)\widetilde\rho^\tau$
shows its explicit factor of $\tau$ after differentiation.
\end{proof}

To apply the spatial maximum principle to
\eqref{eq:appC-error-equation}, we also need to differentiate
$M_r^\tau$ in time.  The following form of Hamilton's maximum
lemma provides this step.  It also applies when the maximization
runs over a finite family of spatial derivatives
\cite[Lemma~3.5]{Hamilton1986} and
\cite[Lemma~9]{ChowLu2004}.

\begin{lemma}[Hamilton's maximum lemma]
\label{lem:hamilton}
Let $S$ be a compact metric space, and let $u$ and
$\partial_tu$ be continuous on $[0,T]\times S$.  Define
\[
  M(t):=\max_{x\in S}u(t,x),
  \qquad
  A(t):=\{x\in S:u(t,x)=M(t)\}.
\]
Then $M$ is Lipschitz and, for every $t<T$, its upper right
derivative satisfies
\begin{equation}
  D^+M(t)
  :=\limsup_{h\downarrow0}\frac{M(t+h)-M(t)}{h}
  =\max_{x\in A(t)}\partial_tu(t,x).
  \label{eq:appC-hamilton-dini}
\end{equation}
At every differentiability time $t\in(0,T)$ of $M$,
\begin{equation}
  M'(t)=\partial_tu(t,x)
  \qquad\text{for every }x\in A(t).
  \label{eq:appC-hamilton-ae}
\end{equation}
The assertions also hold for the maximum of finitely many such
functions by adjoining their index set to $S$.
\end{lemma}

\begin{proof}
We first control the time variation of the maximum.  For
$0\le s<t\le T$, comparison with the values at the same point gives
\[
  M(t)-M(s)
  \le\max_{x\in S}\bigl(u(t,x)-u(s,x)\bigr)
  \le\|\partial_tu\|_{L^\infty([0,T]\times S)}(t-s).
\]
Interchanging $s$ and $t$ proves that $M$ is Lipschitz.

We next identify its upper right derivative.  If $x\in A(t)$,
then keeping this point fixed at time $t+h$ gives
\[
  \frac{M(t+h)-M(t)}h
  \ge\frac{u(t+h,x)-u(t,x)}h,
  \qquad h>0.
\]
Passing to the upper limit and maximizing over $A(t)$ gives
the lower bound in \eqref{eq:appC-hamilton-dini}.

For the reverse bound, choose $h_n\downarrow0$ along which the
quotient on the left converges to $D^+M(t)$, and choose
$x_n\in A(t+h_n)$.  Compactness gives a subsequence with
$x_n\to x_*$.  Continuity of $u$ and $M$ implies
$x_*\in A(t)$.  Moreover,
\[
\begin{aligned}
  \frac{M(t+h_n)-M(t)}{h_n}
  &\le\frac{u(t+h_n,x_n)-u(t,x_n)}{h_n}\\
  &=\int_0^1\partial_tu(t+\theta h_n,x_n)\,d\theta
   \longrightarrow\partial_tu(t,x_*).
\end{aligned}
\]
This proves \eqref{eq:appC-hamilton-dini}.  If $M$ is
differentiable at $t\in(0,T)$, the fixed-point comparison above
for positive and negative increments shows that
$M'(t)=\partial_tu(t,x)$ for every $x\in A(t)$, proving
\eqref{eq:appC-hamilton-ae}.
\end{proof}

For later use, apply Lemma~\ref{lem:hamilton} to the functions
$\sigma\partial^\beta\widetilde\rho_t^\tau(x)$, where
$|\beta|=r$, $\sigma\in\{-1,1\}$, and $r=0,1,2$.
At every $t<T$, there are $\beta_t$, $\sigma_t$, and
$x_t\in\T^d$ for which
\[
\begin{gathered}
  M_r^\tau(t)
   =\sigma_t\partial^{\beta_t}\widetilde\rho_t^\tau(x_t),
  \qquad
  D^+M_r^\tau(t)
   =\sigma_t\partial_t\partial^{\beta_t}
      \widetilde\rho_t^\tau(x_t),\\
  \nabla\bigl(\sigma_t\partial^{\beta_t}
      \widetilde\rho_t^\tau\bigr)(x_t)=0,
  \qquad
  \Delta\bigl(\sigma_t\partial^{\beta_t}
      \widetilde\rho_t^\tau\bigr)(x_t)\le0.
\end{gathered}
\]
The last line is the spatial maximum principle on $\T^d$.
Together, these identities allow us to evaluate a differentiated
version of \eqref{eq:appC-error-equation} at a point where its
order-$r$ maximum norm is attained.
\subsection{First-order maximum-principle estimate and Osgood argument}
\label{subsec:appC-first-order}

The zeroth-order estimate \eqref{eq:linfty-relaxation} makes
$\widetilde\rho_t^\tau=\rho_t^\tau-\rho_\infty^\tau$ small in
$L^\infty$.  To obtain the H\"older decay needed for the Coulomb
Hessian, we first need a bound on its gradient that is uniform as
$\tau\downarrow0$.  The following differential inequality is the
main calculation behind Proposition~\ref{prop:uniform-lipschitz}.

\begin{lemma}[First derivative inequality]
\label{lem:appC-first-order-inequality}
There is $C>0$, independent of $0<\tau\le\tau_0$ and $t\ge0$,
such that, with $A(t):=M_0^\tau(t)$ and $B(t):=M_1^\tau(t)$,
\begin{equation}
\begin{aligned}
  D^+B(t)
  &\le -(c_*-C\tau)B(t)+CA(t)B(t)\\
  &\quad+C(1+B(t))A(t)
    \left[1+\log_+\left(\frac{CB(t)}{A(t)}\right)\right]
    +CA(t).
\end{aligned}
\label{eq:appC-first-order-inequality}
\end{equation}
The entire term containing the quotient is set to zero when $A(t)=0$.
When $A(t)>0$, the mean-zero function
$\widetilde\rho_t^\tau$ is nonconstant, so
$B(t)>0$ and the quotient is defined.
\end{lemma}

\begin{proof}
We derive the equation for a first derivative, then evaluate it at a
maximum of one of the signed derivatives.  The error equation
\eqref{eq:appC-error-equation} has the additional identity
$\rho_\infty^\tau-\pi=\tau\Delta\log\rho_\infty^\tau$, from
\eqref{eq:stationary-poisson-boltzmann}.  For brevity, write
$\rho=\rho_t^\tau$, $g=\widetilde\rho_t^\tau$,
$R=\rho_\infty^\tau$, $q=\log R$, and $b=\nabla K_1*g$.
Then \eqref{eq:appC-error-equation} reads
\[
  \partial_tg=\tau\Delta g+(b-\tau\nabla q)\cdot\nabla g
    +b\cdot\nabla R-\rho g-\tau(\Delta q)g.
\]

\noindent\emph{Step 1. Differentiate the error equation.}
Differentiating the displayed equation with respect to $x_k$ and
using $\nabla\rho=\nabla g+\nabla R$ gives
\begin{equation}
\begin{aligned}
  \partial_t\partial_kg
  &={}
  \tau\Delta\partial_kg
  +(b-\tau\nabla q)\cdot\nabla\partial_kg
  +(\nabla\partial_kR)\cdot b
  +\nabla\rho\cdot\partial_kb\\
  &\quad-\tau(\partial_k\nabla q)\cdot\nabla g
  -\rho\,\partial_kg
  -(\partial_kg+\partial_kR)g
  -\tau(\Delta q)\partial_kg
  -\tau(\partial_k\Delta q)g.
\end{aligned}
\label{eq:appC-first-differentiated-equation}
\end{equation}
In particular, the two products containing $\partial_kb$ have
combined into $\nabla\rho\cdot\partial_kb$.  This term contains
the Coulomb Hessian $\nabla^2K_1*g$, which cannot be bounded from
$\|g\|_{L^\infty}$ alone.

\noindent\emph{Step 2. Evaluate at a spatial maximum.}
By definition, $B(t)$ is the largest value of the signed first
derivatives.  Lemma~\ref{lem:hamilton} supplies a maximizing triple
$(k,\sigma,x_t)$ for which
\begin{equation}
\begin{gathered}
  \sigma\partial_kg_t(x_t)=B(t),\qquad
  \nabla(\sigma\partial_kg_t)(x_t)=0,\qquad
  \Delta(\sigma\partial_kg_t)(x_t)\le0,\\
  D^+B(t)=\sigma\partial_t\partial_kg_t(x_t).
\end{gathered}
\label{eq:appC-first-maximum}
\end{equation}
Thus the diffusion term is nonpositive and the transport term in
\eqref{eq:appC-first-differentiated-equation} vanishes at $x_t$.
The lower density bound in Proposition~\ref{prop:maximum-principle}
gives the main damping term
\begin{equation}
  -\sigma\rho_t(x_t)\partial_kg_t(x_t)
  =-\rho_t(x_t)B(t)\le-c_*B(t).
  \label{eq:appC-M1-damping}
\end{equation}

To control the Coulomb Hessian in the term identified in Step~1,
apply Lemma~\ref{lem:yudovich} to the mean-zero function $g$.
Since $\|\nabla g\|_{L^\infty}\le CB$ and the stationary
gradient is uniformly bounded by
Proposition~\ref{prop:stationary-all-orders}, we obtain
\[
\begin{aligned}
  \|\nabla\rho\|_{L^\infty}&\le C(1+B),\\
  |\nabla\rho\cdot\partial_kb|
  &\le C(1+B)A
      \left[1+\log_+\left(\frac{CB}{A}\right)\right].
\end{aligned}
\]
If $A=0$, then $g\equiv0$, so $B=0$ and the second
left-hand side also vanishes.

It remains to bound the terms without a Coulomb Hessian.
Proposition~\ref{prop:stationary-all-orders} bounds $R$ in
$C^2$ and $q$ in $C^3$, uniformly in temperature.
Since $\nabla K_1\in L^1(\T^d)$, we also have
$\|b\|_{L^\infty}\le C A$, and hence
\begin{equation}
\begin{aligned}
  |\nabla\partial_kR\cdot b|
     &\le CA,\\
  |(\partial_kg+\partial_kR)g|
     &\le C(1+B)A,\\
  \tau|\partial_k\nabla q\cdot\nabla g|
   +\tau|\Delta q|\,|\partial_kg|
     &\le C\tau B,\\
  \tau|\partial_k\Delta q|\,|g|
     &\le C\tau A.
\end{aligned}
\label{eq:appC-M1-tau-low}
\end{equation}
Combining these estimates with \eqref{eq:appC-M1-damping} proves
\eqref{eq:appC-first-order-inequality}, after taking $\tau_0\le1$
and absorbing $\tau A$ into $A$.
\end{proof}

\begin{proof}[Proof of Proposition~\ref{prop:uniform-lipschitz}]
We now integrate \eqref{eq:appC-first-order-inequality} to bound
$B=M_1^\tau$.  Reduce $\tau_0$, if necessary, so that
$C\tau_0\le c_*/2$.  The damping then remains nonpositive and may
be discarded.  The remaining obstacle is the logarithm of $B/A$
when the decaying quantity $A$ is small.

Proposition~\ref{prop:linfty-exp} provides $A(t)\le
C_0e^{-\gamma_1t}$, where $\gamma_1=c_*/8$ and $C_0$
does not depend on $\tau$.  For each $a>0$, the function
$x\mapsto x[1+\log_+(a/x)]$ is nondecreasing on $(0,\infty)$.
Use this with $a=\max\{1,CB(t)\}$ and
$x=A(t)$ to get
\[
\begin{aligned}
  A(t)\left[1+\log_+\left(\frac{CB(t)}{A(t)}\right)\right]
  &\le C_0e^{-\gamma_1t}
     \left[1+\log_+\left(
       \frac{\max\{1,CB(t)\}}{C_0e^{-\gamma_1t}}
     \right)\right]\\
  &\le C(1+t)e^{-\gamma_1t}\log(e+B(t)).
\end{aligned}
\]
The same exponential factor bounds $A(t)$, and its product with
$1+B(t)$ bounds $A(t)B(t)$.  Consequently,
\[
  D^+M_1^\tau(t)
  \le C(1+t)e^{-\gamma_1t}
    (1+M_1^\tau(t))\log(e+M_1^\tau(t)).
\]
This is the Osgood inequality \eqref{eq:first-derivative-osgood}
stated in Proposition~\ref{prop:uniform-lipschitz}.

By Lemma~\ref{lem:hamilton}, $M_1^\tau$ is locally Lipschitz,
so we may divide \eqref{eq:first-derivative-osgood} by
$(e+M_1^\tau)\log(e+M_1^\tau)$ at almost every time.
The resulting left-hand side equals
$\frac{d}{dt}\log\log(e+M_1^\tau(t))$.  Thus
\[
  \log\log(e+M_1^\tau(t))
  \le \log\log(e+M_1^\tau(0))
       +C\int_0^t(1+s)e^{-\gamma_1s}\,ds.
\]
The stationary $C^1$ bound in
Proposition~\ref{prop:stationary-all-orders} and the fixed initial
density $\rho_0$ bound $M_1^\tau(0)$ independently of $\tau$.
The integral is finite uniformly in $t$.  Hence $B(t)$ is
uniformly bounded, and
$\rho_t^\tau=g_t+\rho_\infty^\tau$, together with the
stationary bounds, proves \eqref{eq:uniform-lipschitz-bound}.
\end{proof}
\subsection{Second-order maximum-principle estimate and decay}
\label{subsec:appC-second-order}

The decay of the Coulomb field in
\eqref{eq:coulomb-field-relaxation} controls the first term of the
recentered velocity gradient
\eqref{eq:equilibrium-centered-deformation}.  To control its thermal
term, we need two derivatives of the density error
$\widetilde\rho_t^\tau=\rho_t^\tau-\rho_\infty^\tau$.  We therefore
prove a differential inequality for $M_2^\tau$, the largest
$L^\infty$ norm of its second partial derivatives, as defined in
Subsection~\ref{subsec:appC-centered-zero-order}.  This inequality
supplies the remaining estimate in
Proposition~\ref{prop:low-temperature-relaxation}.

\begin{lemma}[Second-order maximum estimate]
\label{lem:appC-second-order-maximum-estimate}
There are $\tau_0>0$ and $C<\infty$, depending only on the standing
data, such that, for $0<\tau\le\tau_0$ and $t\ge0$,
\[
  D^+M_2^\tau(t)
  \le
  \left(-\frac{c_*}{2}+Ce^{-\gamma_2t}\right)M_2^\tau(t)
  +Ce^{-\gamma_2t},
  \qquad \gamma_2=\frac{c_*}{16}.
\]
Here $D^+$ is the upper right derivative defined in
Lemma~\ref{lem:hamilton}.  The value of $\tau_0$ may be decreased
from its choice in Proposition~\ref{prop:uniform-lipschitz}.
\end{lemma}

\begin{proof}
Fix $0<\tau\le\tau_0$ and suppress $\tau$ and $t$ on the densities
when no confusion is possible.  We differentiate the error equation
\eqref{eq:appC-error-equation}, identify the term that damps the second
derivatives, and estimate the remaining terms by earlier decay bounds.
The regularity established in Subsection~\ref{subsec:appC-centered-zero-order}
justifies these differentiations.

\medskip
\noindent\emph{Step 1: the differentiated equation.}
Write $V=\nabla K_1*\widetilde\rho$ and
$q_\tau=\log\rho_\infty^\tau$.  By the stationary equation
\eqref{eq:stationary-log-equation},
\[
  \rho_\infty^\tau-\pi=\tau\Delta q_\tau,
  \qquad
  \partial_t\widetilde\rho
  =\tau\Delta\widetilde\rho
   +(V-\tau\nabla q_\tau)\cdot\nabla\widetilde\rho
   +\nabla\rho_\infty^\tau\cdot V
   -\rho\widetilde\rho
   -\tau(\Delta q_\tau)\widetilde\rho.
\]
The first-order term in this form will be a transport term at a
spatial extremum of $\partial_{ij}\widetilde\rho$.  For the
zeroth-order product, the Leibniz rule gives
\[
\begin{aligned}
  -\partial_{ij}(\rho\widetilde\rho)
  ={}&-(\rho+\widetilde\rho)\partial_{ij}\widetilde\rho
       -(\partial_i\rho)(\partial_j\widetilde\rho)
       -(\partial_j\rho)(\partial_i\widetilde\rho)\\
     &-(\partial_{ij}\rho_\infty^\tau)\widetilde\rho.
\end{aligned}
\]
Thus $-\rho\,\partial_{ij}\widetilde\rho$ retains the same
positive coefficient as the damping term at zeroth order.  Applying
$\partial_{ij}$ to all the terms gives the full equation
\begin{equation}
\begin{aligned}
  \partial_t\partial_{ij}\widetilde\rho
  &={}
  \tau\Delta\partial_{ij}\widetilde\rho
  +\left(V-\tau\nabla q_\tau\right)
      \cdot\nabla\partial_{ij}\widetilde\rho
  +\nabla\partial_{ij}\rho_\infty^\tau\cdot V\\
  &\quad
  +(\partial_i\nabla\rho)\cdot\partial_jV
  +(\partial_j\nabla\rho)\cdot\partial_iV
  +\nabla\rho\cdot\partial_{ij}V\\
  &\quad
  -\tau(\partial_i\nabla q_\tau)\cdot
       \nabla\partial_j\widetilde\rho
  -\tau(\partial_j\nabla q_\tau)\cdot
       \nabla\partial_i\widetilde\rho
  -\tau(\partial_{ij}\nabla q_\tau)\cdot
       \nabla\widetilde\rho\\
  &\quad
  -(\rho+\widetilde\rho)\partial_{ij}\widetilde\rho
  -(\partial_i\rho)(\partial_j\widetilde\rho)
  -(\partial_j\rho)(\partial_i\widetilde\rho)
  -(\partial_{ij}\rho_\infty^\tau)\widetilde\rho\\
  &\quad
  -\tau(\Delta q_\tau)\partial_{ij}\widetilde\rho
  -\tau(\partial_i\Delta q_\tau)\partial_j\widetilde\rho
  -\tau(\partial_j\Delta q_\tau)\partial_i\widetilde\rho
  -\tau(\partial_{ij}\Delta q_\tau)\widetilde\rho.
\end{aligned}
\label{eq:appC-second-differentiated-equation}
\end{equation}
The derivatives of $q_\tau$ that appear here are uniformly bounded
through order four by Proposition~\ref{prop:stationary-all-orders}.

To turn \eqref{eq:appC-second-differentiated-equation} into an
inequality for $M_2$, apply Lemma~\ref{lem:hamilton} to the finite
family $\{\pm\partial_{ij}\widetilde\rho\}_{i,j}$.  At each time
$t$, choose a maximizing sign $\sigma$, indices $i,j$, and a point
$x_t$ that realize $D^+M_2(t)$.  At that point,
\[
\begin{gathered}
  \sigma\partial_{ij}\widetilde\rho_t(x_t)=M_2(t),
  \qquad
  D^+M_2(t)=
    \sigma\partial_t\partial_{ij}\widetilde\rho_t(x_t),\\
  \nabla\partial_{ij}\widetilde\rho_t(x_t)=0,
  \qquad
  \sigma\Delta\partial_{ij}\widetilde\rho_t(x_t)\le0.
\end{gathered}
\]
Hence diffusion is nonpositive and the first-order transport term
vanishes after multiplication by $\sigma$.  It remains to estimate
the other terms at $x_t$.

\medskip
\noindent\emph{Step 2: Coulomb terms.}
The field $V$ and its first derivative decay by
\eqref{eq:coulomb-field-relaxation}.  Since
$\|\rho_\infty^\tau\|_{C^3}\le C$ and
$\|\nabla^2\rho_t\|_{L^\infty}\le C(1+M_2(t))$, this gives
\begin{equation}
\begin{aligned}
  &|\nabla\partial_{ij}\rho_\infty^\tau\cdot V|
   +|\partial_i\nabla\rho\cdot\partial_jV|
   +|\partial_j\nabla\rho\cdot\partial_iV|\\
  &\hspace{30mm}\le
     Ce^{-\gamma_2t}\bigl(1+M_2(t)\bigr).
\end{aligned}
\label{eq:appC-M2-decaying-field}
\end{equation}
The remaining Coulomb term contains two derivatives of $V$.
Moving one of these derivatives onto $\widetilde\rho$ makes
the periodic Schauder estimate in Lemma~\ref{lem:periodic-schauder}
applicable.  Lemma~\ref{lem:appC-ehrling} then provides, for every
$\varepsilon>0$,
\begin{equation}
\begin{aligned}
  \partial_{ij}V
  &=\partial_i\nabla K_1*\partial_j\widetilde\rho,\\
  |\nabla\rho\cdot\partial_{ij}V|
  &\le C\|\partial_j\widetilde\rho\|_{C^{0,1/2}}
   \le\varepsilon M_2(t)
       +C_\varepsilon\|\widetilde\rho_t\|_{L^\infty}.
\end{aligned}
\label{eq:appC-M2-singular-term}
\end{equation}
The first identity holds in distributions and
$\partial_j\widetilde\rho$ has mean zero.  The constant multiplying
$M_2$ in the second line is included in the choice of $\varepsilon$.

\medskip
\noindent\emph{Step 3: damping and local products.}
The density lower bound \eqref{eq:maximum-principle} and the
maximizing identity at $x_t$ give
\begin{equation}
  -\sigma(\rho+\widetilde\rho)
      \partial_{ij}\widetilde\rho
  \le -c_*M_2(t)
       +\|\widetilde\rho_t\|_{L^\infty}M_2(t).
  \label{eq:appC-M2-local-damping}
\end{equation}
The products of first derivatives involve
$\|\nabla\rho_t\|_{L^\infty}$, which is bounded by
Proposition~\ref{prop:uniform-lipschitz}.  The
Landau--Kolmogorov inequality of
Lemma~\ref{lem:appC-landau-kolmogorov}, followed by Young's
inequality, gives
\[
\begin{aligned}
  &|\partial_i\rho\,\partial_j\widetilde\rho|
   +|\partial_j\rho\,\partial_i\widetilde\rho|\\
  &\qquad\le C\|\nabla\widetilde\rho_t\|_{L^\infty}
   \le C\bigl(\|\widetilde\rho_t\|_{L^\infty}M_2(t)\bigr)^{1/2}
   \le\varepsilon M_2(t)
      +C_\varepsilon\|\widetilde\rho_t\|_{L^\infty}.
\end{aligned}
\]
The other product from this line of the equation is bounded by
$|\partial_{ij}\rho_\infty^\tau\,\widetilde\rho|
\le C\|\widetilde\rho_t\|_{L^\infty}$.

The remaining terms have an explicit factor $\tau$.  The bounds on
$q_\tau$ through order four control their coefficients.  Terms with
two derivatives of $\widetilde\rho$ cost $C\tau M_2(t)$.
For the terms with one derivative, the same
Landau--Kolmogorov and Young estimates give
\begin{equation}
\begin{aligned}
  &\tau|\partial_{ij}\nabla q_\tau
                   \cdot\nabla\widetilde\rho|
    +\tau|\partial_i\Delta q_\tau\,
                   \partial_j\widetilde\rho|
    +\tau|\partial_j\Delta q_\tau\,
                   \partial_i\widetilde\rho|
    +\tau|\partial_{ij}\Delta q_\tau\,
                   \widetilde\rho|\\
  &\hspace{25mm}\le
    \varepsilon\tau M_2(t)
    +C_\varepsilon\tau\|\widetilde\rho_t\|_{L^\infty}.
\end{aligned}
\label{eq:appC-M2-tau-low}
\end{equation}
This accounts for every term in
\eqref{eq:appC-second-differentiated-equation} after diffusion and
transport have been removed.

\medskip
\noindent\emph{Step 4: absorption.}
The earlier density estimate \eqref{eq:linfty-relaxation} and
the Coulomb estimate \eqref{eq:coulomb-field-relaxation} give
\[
  \|\widetilde\rho_t\|_{L^\infty}
  \le Ce^{-\gamma_1t}
  \le Ce^{-\gamma_2t},
  \qquad
  \gamma_1=\frac{c_*}{8}>\gamma_2.
\]
Combining this decay with
\eqref{eq:appC-M2-decaying-field}--\eqref{eq:appC-M2-tau-low}
and the $C\tau M_2$ terms gives
\[
  D^+M_2(t)
  \le
  -\bigl(c_*-C\varepsilon-C\tau_0\bigr)M_2(t)
  +Ce^{-\gamma_2t}M_2(t)
  +C_\varepsilon e^{-\gamma_2t}.
\]
Choose $\varepsilon$ so that $C\varepsilon\le c_*/4$ and then
decrease $\tau_0$ so that $C\tau_0\le c_*/4$.  This proves the
claimed inequality, which is
\eqref{eq:C2-differential-inequality} in the main text.
\end{proof}

\begin{proof}[Proof of Proposition~\ref{prop:low-temperature-relaxation}]
We integrate the maximum estimate and then use it in the recentered
velocity identity.  Since $M_2$ is locally Lipschitz by
Lemma~\ref{lem:hamilton}, its differential inequality holds almost
everywhere.  For $0\le s\le t$ the integrating factor is bounded by
\[
  \exp\!\left(\int_s^t
       \left[-\frac{c_*}{2}+Ce^{-\gamma_2r}\right]dr\right)
  \le e^{C/\gamma_2}e^{-c_*(t-s)/2}.
\]
The initial value $M_2^\tau(0)$ is bounded uniformly in $\tau$ by
Assumption~\ref{assump:standing-data} and the stationary estimate
\eqref{eq:uniform-stationary-regularity}.  Variation of constants,
with $\gamma_2=c_*/16<c_*/2$, gives
\[
\begin{aligned}
  M_2^\tau(t)
  &\le Ce^{-c_*t/2}M_2^\tau(0)
       +C\int_0^t e^{-c_*(t-s)/2}e^{-\gamma_2s}\,ds\\
  &\le Ce^{-\gamma_2t}.
\end{aligned}
\]
To include lower derivatives, combine this bound with
$\|\widetilde\rho_t^\tau\|_{L^\infty}
\le Ce^{-\gamma_1t}$ from \eqref{eq:linfty-relaxation}.
Lemma~\ref{lem:appC-landau-kolmogorov} yields
\[
\begin{aligned}
  \|\widetilde\rho_t^\tau\|_{C^2}
  &\le C\left(
    \|\widetilde\rho_t^\tau\|_{L^\infty}
    +M_2^\tau(t)
    +\bigl(\|\widetilde\rho_t^\tau\|_{L^\infty}
             M_2^\tau(t)\bigr)^{1/2}\right)\\
  &\le Ce^{-\gamma_2t}.
\end{aligned}
\]
By \eqref{eq:maximum-principle} and the stationary bounds
\eqref{eq:uniform-stationary-regularity}, the two densities are
bounded below by $c_*$ and uniformly bounded in $C^2$.  Applying
the identity
\[
  \nabla^2\log r
  =r^{-1}\nabla^2r-r^{-2}\nabla r\otimes\nabla r
\]
to $r=\rho_t^\tau$ and $r=\rho_\infty^\tau$ therefore gives
\[
  \|\nabla^2\log\rho_t^\tau
       -\nabla^2\log\rho_\infty^\tau\|_{L^\infty}
  \le C\|\widetilde\rho_t^\tau\|_{C^2}
  \le Ce^{-\gamma_2t}.
\]
These are the two estimates in \eqref{eq:C2-relaxation}.

Finally, the recentered identity
\eqref{eq:equilibrium-centered-deformation} separates the velocity
gradient into Coulomb and thermal terms.  The first decays by
\eqref{eq:coulomb-field-relaxation}, while the estimate just proved
controls the second.  Since $\tau\le\tau_0$,
\[
\begin{aligned}
  \|\nabla u_t^\tau\|_{L^\infty}
  &\le \|\nabla^2K_1*\widetilde\rho_t^\tau\|_{L^\infty}
     +\tau\|\nabla^2\log\rho_t^\tau
                 -\nabla^2\log\rho_\infty^\tau\|_{L^\infty}\\
  &\le Ce^{-\gamma_2t}.
\end{aligned}
\]
This proves \eqref{eq:low-temperature-velocity-gradient}.
\end{proof}
\subsection{The zero-temperature limit}
\label{subsec:zero-temperature-limit}

The estimates for positive temperatures in
Subsection~\ref{subsec:low-temperature} have constants independent of
$\tau$ as $\tau\downarrow0$.  We first identify their limit with
the bounded solution of the deterministic mean-field equation.  We
then transfer the density and velocity estimates to that solution,
which proves Proposition~\ref{prop:tau0-endpoint} and supplies its
input to Theorem~\ref{thm:uniform-envelope}.

\begin{lemma}[Compactness at zero temperature]
\label{lem:appC-zero-temperature-compactness}
Let $\rho^\tau$, $0<\tau\le\tau_0$, be the solutions of
\eqref{eq:intro-mean-field-pde} under
Assumption~\ref{assump:standing-data}, and let $\rho^0$ be its
unique bounded weak solution at $\tau=0$ with initial density
$\rho_0$.  Then
\[
  \rho^\tau\longrightarrow\rho^0
  \quad\text{in }C_{\mathrm{loc}}
    \bigl([0,\infty);C^1(\T^d)\bigr)
  \quad\text{as }\tau\downarrow0.
\]
\end{lemma}

\begin{proof}
We first obtain compactness from the uniform spatial estimates and
the equation.  We then identify every subsequential limit using the
uniqueness of bounded weak solutions at $\tau=0$.

\medskip
\noindent\emph{Step 1: compactness.}
The $C^2$-decay in \eqref{eq:C2-relaxation} and the stationary
regularity in \eqref{eq:uniform-stationary-regularity} give
\begin{equation}
  \sup_{0<\tau\le\tau_0}\sup_{t\ge0}
  \|\rho_t^\tau\|_{W^{2,\infty}}\le C.
  \label{eq:appC-endpoint-uniform-spatial-bound}
\end{equation}
To control the time increments, write the mean-field equation using
$-\Delta K_1*(\rho_t^\tau-\pi)=\rho_t^\tau-\pi$:
\[
  \partial_t\rho_t^\tau
  =\tau\Delta\rho_t^\tau
   +\nabla\rho_t^\tau\cdot
     \bigl(\nabla K_1*(\rho_t^\tau-\pi)\bigr)
   -\rho_t^\tau(\rho_t^\tau-\pi).
\]
The kernel satisfies $\nabla K_1\in L^1(\T^d)$.  Thus
\eqref{eq:appC-endpoint-uniform-spatial-bound}, the fixed bound on
$\pi$, and $\tau\le\tau_0$ imply
\[
  \sup_{0<\tau\le\tau_0}\sup_{t\ge0}
  \|\partial_t\rho_t^\tau\|_{L^\infty}\le C,
  \qquad
  \|\rho_t^\tau-\rho_s^\tau\|_{L^\infty}
     \le C|t-s|.
\]
The two uniform bounds above give
$\|g\|_{L^\infty}\le C\min\{|t-s|,1\}$.  The
Landau--Kolmogorov inequality in
Lemma~\ref{lem:appC-landau-kolmogorov}, applied to
$g=\rho_t^\tau-\rho_s^\tau$, upgrades this to
\[
  \|\rho_t^\tau-\rho_s^\tau\|_{C^1}
  \le \|g\|_{L^\infty}
       +C\bigl(\|g\|_{L^\infty}
                   \|\nabla^2g\|_{L^\infty}\bigr)^{1/2}
  \le C|t-s|^{1/2}.
\]
The bounded sets of $W^{2,\infty}(\T^d)$ are relatively compact
in $C^1(\T^d)$.  Arzel\`a--Ascoli on each bounded time interval
and a diagonal argument therefore show that every sequence
$\tau_n\downarrow0$ has a subsequence such that
\begin{equation}
  \rho^{\tau_n}\longrightarrow\bar\rho
  \quad\text{in }C_{\mathrm{loc}}
       \bigl([0,\infty);C^1(\T^d)\bigr).
  \label{eq:appC-endpoint-subsequential-limit}
\end{equation}

\medskip
\noindent\emph{Step 2: identification.}
The spatial bound in
\eqref{eq:appC-endpoint-uniform-spatial-bound} makes the diffusion
term vanish on every bounded time interval:
\[
  \|\tau_n\Delta\rho^{\tau_n}\|_{L^\infty([0,T]\times\T^d)}
  \le C\tau_n\longrightarrow0.
\]
The convergence in \eqref{eq:appC-endpoint-subsequential-limit} and
$\nabla K_1\in L^1(\T^d)$ also give
\[
  \rho^{\tau_n}\bigl(\nabla K_1*(\rho^{\tau_n}-\pi)\bigr)
  \longrightarrow
  \bar\rho\bigl(\nabla K_1*(\bar\rho-\pi)\bigr)
  \quad\text{uniformly on }[0,T]\times\T^d.
\]
Passing to the limit in the weak formulation of
\eqref{eq:intro-mean-field-pde}, we find that
\[
  \partial_t\bar\rho
   =\nabla\cdot\bigl(\bar\rho\,
       \nabla K_1*(\bar\rho-\pi)\bigr),
  \qquad \bar\rho\big|_{t=0}=\rho_0.
\]
The initial value follows from
\eqref{eq:appC-endpoint-subsequential-limit}.  Uniqueness in
\cite[Proposition~2.4]{ChodronRosenzweig2026} thus identifies
$\bar\rho=\rho^0$.  Since this applies to every convergent
subsequence, the whole family has the asserted limit.
\end{proof}

\begin{proof}[Proof of Proposition~\ref{prop:tau0-endpoint}]
Lemma~\ref{lem:appC-zero-temperature-compactness} proves
\eqref{eq:tau-zero-strong-convergence}.  We now pass the
positive-temperature relaxation estimates to this limit.  The
distinction is that the first derivatives converge strongly, while
the Hessians are only bounded in $L^\infty$.  Their limiting bound
will therefore be stated in $W^{2,\infty}$.

\medskip
\noindent\emph{Step 1: the density error.}
The stationary bias estimate \eqref{eq:stat-bias-linear} gives
$\rho_\infty^\tau\to\pi$ in $C^2$.  In combination with the
lemma, this yields, locally uniformly in time,
\[
  g_t^\tau:=\rho_t^\tau-\rho_\infty^\tau
  \longrightarrow g_t^0:=\rho_t^0-\pi
  \quad\text{in }C^1(\T^d).
\]
In particular, the bounds $c_*\le\rho_t^\tau\le C_*$ from
\eqref{eq:maximum-principle} pass to the limit.  Taking the limit in
the zeroth- and first-order parts of \eqref{eq:C2-relaxation} gives
\[
  c_*\le\rho_t^0\le C_*,
  \qquad
  \|g_t^0\|_{C^1}\le Ce^{-\gamma_2t},
  \qquad t\ge0.
\]
For the Hessian, fix any $t\ge0$.  The remaining part of
\eqref{eq:C2-relaxation} says that
\[
  \sup_{0<\tau\le\tau_0}
  \|\nabla^2g_t^\tau\|_{L^\infty}
  \le Ce^{-\gamma_2t}.
\]
Along any sequence $\tau_n\downarrow0$, a further subsequence
converges weak-* in $L^\infty(\T^d)$.  Strong $C^1$ convergence
identifies its limit in distributions as the weak Hessian of
$g_t^0$.  Weak-* lower semicontinuity therefore gives
\[
  \|\nabla^2g_t^0\|_{L^\infty}
  \le\liminf_{n\to\infty}
       \|\nabla^2g_t^{\tau_n}\|_{L^\infty}
  \le Ce^{-\gamma_2t}.
\]
Because $t$ was arbitrary, this bound and the $C^1$ bound
prove \eqref{eq:tau-zero-W2inf-relaxation} for every $t\ge0$.
The Hessian here is a weak derivative.  The argument does not
assert convergence in $C^2$.

\medskip
\noindent\emph{Step 2: the velocity gradient.}
At zero temperature the velocity in
\eqref{eq:free-energy-velocity} becomes
$u_t^0=\nabla K_1*g_t^0$.  To bound its gradient, use the
$C^1$ convergence of $g_t^\tau$ from Step~1 and the periodic
Schauder estimate of Lemma~\ref{lem:periodic-schauder}:
\[
  \nabla^2K_1*g_t^\tau
  \longrightarrow\nabla^2K_1*g_t^0
  \quad\text{in }L^\infty(\T^d)
  \qquad\text{for each }t\ge0.
\]
Indeed, $C^1$ convergence implies $C^{0,1/2}$ convergence,
which is the norm used by the Schauder estimate.  Passing to the
limit in \eqref{eq:coulomb-field-relaxation} now yields
\[
  \|\nabla u_t^0\|_{L^\infty}
  =\|\nabla^2K_1*g_t^0\|_{L^\infty}
  \le Ce^{-\gamma_2t},
\]
as claimed in \eqref{eq:tau-zero-velocity-gradient}.

The same convergence also transfers the weak relaxation estimate,
which is used later with the whole temperature range.  Since
$g_t^\tau\to g_t^0$ in $L^2$ and both have zero mean,
$\|h\|_{\dot H^{-1}}\le(2\pi)^{-1}\|h\|_{L^2}$ shows that
$g_t^\tau\to g_t^0$ in $\dot H^{-1}$.  Thus
\eqref{eq:hminus-relaxation} implies
\[
  \|\rho_t^0-\pi\|_{\dot H^{-1}}
  \le Ce^{-c_*t/4},
  \qquad t\ge0.
\]
\end{proof}

\section{Positive-temperature regularity}
\label{app:positive-temperature-wellposedness}

This appendix supplies the two positive-temperature inputs used in
Subsection~\ref{subsec:high-temperature-estimates}.
Lemma~\ref{lem:positive-temperature-wp} gives a global bound on four
derivatives of the mean-field solution when $\tau$ is bounded away
from zero.  Lemma~\ref{lem:appD-weak-strong-interpolation} combines
that bound with the $\dot H^{-1}$-decay of
Proposition~\ref{prop:hminus-exp} to control two derivatives in a
H\"older norm.  Together, they lead to the velocity estimate in
Proposition~\ref{prop:positive-temperature-halfline}.

\subsection{Global regularity on positive-temperature half-lines}
\label{subsec:appD-global-regularity}

This subsection proves Lemma~\ref{lem:positive-temperature-wp}.
The main point is the bound~\eqref{eq:positive-temperature-W4} for
all $\tau\ge\tau_{\mathrm L}>0$.  It supplies the strong norm used in
Subsection~\ref{subsec:high-temperature-estimates} to turn
$\dot H^{-1}$-decay into decay of the velocity gradient.
The proof first constructs a solution for a short time, then bounds its
first four derivatives on every time interval, and finally extends it
globally.  We use the contraction and continuation argument of
\cite[proof of Proposition~2.1]{ChodronRosenzweigSerfaty2025} and the
derivative estimates of
\cite[proof of Lemma~2.7]{ChodronRosenzweigSerfaty2025}.

The local construction uses the periodic heat semigroup
$e^{a t\Delta}$, where $a>0$ is its diffusion coefficient.
We record the bounds needed both for that construction and for the
later derivative estimates.

\begin{lemma}[Periodic heat-semigroup bounds]
\label{lem:appD-heat-smoothing}
For $a,t>0$, integers $k,m\ge0$, a bounded scalar function
$g$, and a vector field $G\in C^m(\T^d;\mathbb R^d)$,
\begin{equation}
  \|\nabla^k e^{at\Delta}g\|_{L^\infty}
  \le C_k(at)^{-k/2}\|g\|_{L^\infty},
  \qquad
  \|e^{at\Delta}\nabla\!\cdot G\|_{C^m}
  \le C_m(at)^{-1/2}\|G\|_{C^m}.
  \label{eq:appD-heat-smoothing}
\end{equation}
The constants depend only on $d,k,m$.
\end{lemma}

\begin{proof}
The periodic heat kernel is the periodization of the Euclidean
Gaussian.  Its $k$-th derivative has $L^1$-norm at most
$C_k(at)^{-k/2}$.  Young's inequality gives the first bound.
For the second, commute the $m$ spatial derivatives with the heat
semigroup and put the single divergence derivative on its kernel.
The first-derivative $L^1$-bound gives the stated factor.
See also \cite[Section~3, Eq.~(3.3)]{ChodronRosenzweigSerfaty2025}.
\end{proof}

\begin{proof}[\normalfont\bfseries Proof of Lemma~\ref{lem:positive-temperature-wp}]
Fix $\tau_{\mathrm L}>0$.  Until the final uniqueness argument,
every constant denoted by $C_{\tau_{\mathrm L}}$ may depend on
$\tau_{\mathrm L}^{-1}$ and the data in
Assumption~\ref{assump:standing-data}, but is independent of
$t\ge0$ and $\tau\ge\tau_{\mathrm L}$.
We prove local existence, establish the uniform derivative bounds,
and use them to continue the solution for all time.

\medskip
\noindent\emph{Step 1: local existence and a density bound.}
Write the mean-field equation~\eqref{eq:intro-mean-field-pde} in its
mild form.  For a candidate density $v$, set
\[
  F(v):=v\,\nabla K_1*(v-\pi),
  \qquad
  \Phi_\tau(v)(t)
  :=e^{\tau t\Delta}\rho_0
    +\int_0^t e^{\tau(t-s)\Delta}\nabla\!\cdot F(v_s)\,ds.
\]
Here $F(v)$ is the flux in the equation and a fixed point of
$\Phi_\tau$ is a mild solution.  The Coulomb field
$\nabla K_1$ belongs to $L^1(\T^d)$, so differentiation can be
passed from the convolution kernel to $v-\pi$.  The product rule
therefore gives, for $m=0,4$,
\[
\begin{aligned}
  \|F(v)\|_{C^m}
  &\le C_m\|v\|_{C^m}
       \bigl(\|v\|_{C^m}+\|\pi\|_{C^m}\bigr),\\
  \|F(v)-F(w)\|_{C^m}
  &\le C_m\bigl(\|v\|_{C^m}+\|w\|_{C^m}
                   +\|\pi\|_{C^m}\bigr)\|v-w\|_{C^m}.
\end{aligned}
\]
For $m=0$, the same inequalities hold with the $L^\infty$-norm.
Since
$\int_0^T[\tau(T-s)]^{-1/2}ds=2\sqrt{T/\tau}$,
Lemma~\ref{lem:appD-heat-smoothing} makes $\Phi_\tau$ a contraction
on a suitable ball of
$C([0,T];C^4(\T^d))$ when $T$ is sufficiently small in terms
of $\tau_{\mathrm L}$ and $\|\rho_0\|_{C^4}+\|\pi\|_{C^4}$.
For example, take the ball of radius $R=2\|\rho_0\|_{C^4}$ and
choose $T$ so that
\[
  C\sqrt{T/\tau_{\mathrm L}}
      \bigl(2R+\|\pi\|_{C^4}\bigr)\le\tfrac12.
\]
The Banach fixed-point theorem gives a mild solution with four
continuous spatial derivatives on a
common initial interval $[0,h_0]$, for some $h_0>0$ independent
of $\tau\ge\tau_{\mathrm L}$, and
\begin{equation}
  \sup_{\tau\ge\tau_{\mathrm L}}\sup_{0\le t\le h_0}
  \|\rho_t^\tau\|_{C^4}
  \le C_{\tau_{\mathrm L}}.
  \label{eq:appD-local-C4}
\end{equation}
The equation and heat regularity make this solution classical.
The $L^\infty$ version of the difference estimate will give
uniqueness among bounded mild solutions in Step 3.

Integrating the mild equation over the torus shows that
$\int_{\T^d}\rho_t^\tau=1$ on the maximal classical interval
$[0,T_*^\tau)$.  Thus $\rho_t^\tau-\pi$ has mean zero, and
$-\Delta K_1*(\rho_t^\tau-\pi)=\rho_t^\tau-\pi$.
The local maximum-principle calculation used in
Proposition~\ref{prop:density-bounds} now applies on this interval:
\[
  \partial_t\rho_t^\tau
  =\tau\Delta\rho_t^\tau
   +\nabla K_1*(\rho_t^\tau-\pi)\cdot\nabla\rho_t^\tau
   -\rho_t^\tau(\rho_t^\tau-\pi).
\]
At a spatial extremum the transport term vanishes.
Lemma~\ref{lem:hamilton} and scalar comparison, starting from
$c_0\le\rho_0\le C_0$, give the bounds needed in Step 2:
\begin{equation}
  c_*\le\rho_t^\tau(x)\le C_*,
  \qquad 0\le t<T_*^\tau,\quad x\in\T^d.
  \label{eq:appD-density-on-maximal-interval}
\end{equation}

\medskip
\noindent\emph{Step 2: derivative bounds on the maximal interval.}
We show that, for $n=1,\ldots,4$,
\[
  \sup_{\tau\ge\tau_{\mathrm L}}
  \sup_{0\le t<T_*^\tau}
  \|\nabla^n\rho_t^\tau\|_{L^\infty}
  \le C_{n,\tau_{\mathrm L}}.
\]
The case $n=0$ is~\eqref{eq:appD-density-on-maximal-interval}.
Assume the claim through order $n-1$.  In $\nabla^nF(\rho)$,
the only factors containing an $n$-th derivative of $\rho$
are $\nabla^n\rho\,\nabla K_1*(\rho-\pi)$ and
$\rho\,\nabla K_1*\nabla^n\rho$.
Since $\nabla K_1\in L^1$, the induction hypothesis yields
\begin{equation}
  \|\nabla^nF(\rho_t^\tau)\|_{L^\infty}
  \le R_{n,\tau_{\mathrm L}}
       +C\|\nabla^n\rho_t^\tau\|_{L^\infty},
  \qquad 0\le t<T_*^\tau.
  \label{eq:appD-highest-derivative-linear}
\end{equation}
Here $R_{n,\tau_{\mathrm L}}$ depends on the already bounded
derivatives and on $\|\pi\|_{C^n}$.  The constant $C$ depends
only on the data and $n\le4$.

To use~\eqref{eq:appD-highest-derivative-linear} without assuming
an earlier bound for $\nabla^n\rho$, restart the mild formula at
$a=t-h$, where $t\ge h$ and $h\le h_0/2$ is fixed below.
For $0<u\le h$, define the weighted derivative maximum
\[
  Z_n(u):=\sup_{0<r\le u}
    (\tau r)^{n/2}
    \|\nabla^n\rho_{a+r}^\tau\|_{L^\infty}.
\]
We split the Duhamel integral at $u/2$, as shown in
Figure~\ref{fig:appD-duhamel-split}.  When $r\le u/2$, the heat
time $u-r$ is at least $u/2$, so we put all $n+1$
derivatives on the heat kernel.  When $r\ge u/2$, we put the
$n$ derivatives on the flux and leave only one heat-kernel
derivative, whose $(u-r)^{-1/2}$ singularity is integrable.

\begin{figure}[ht]
\centering
\begin{tikzpicture}[x=10cm,y=1cm]
  \draw[->] (0,0)--(1.08,0) node[right] {$r$};
  \draw (0,-0.08)--(0,0.08) node[above=2pt] {$0$};
  \draw (.5,-0.08)--(.5,0.08) node[above=2pt] {$u/2$};
  \draw (1,-0.08)--(1,0.08) node[above=2pt] {$u$};
  \draw[very thick] (0,-.12)--(.5,-.12);
  \draw[very thick] (.5,-.12)--(1,-.12);
  \node[align=center,anchor=north] at (.25,-.27)
    {$(n+1)$ derivatives on heat kernel\\$u-r\ge u/2$};
  \node[align=center,anchor=north] at (.75,-.27)
    {$n$ derivatives on flux\\$1$ derivative on heat kernel};
\end{tikzpicture}
\caption{The two parts of the Duhamel integral after restarting at
  time $a=t-h$.  The variable $r$ is the elapsed time after
  $a$, and $u\le h$ is the observation time.  The derivative
  placement shown on each interval keeps the heat-kernel singularity
  integrable in the estimate for $\nabla^n\rho_{a+u}^\tau$.}
\label{fig:appD-duhamel-split}
\end{figure}

Writing $F_{a+r}:=F(\rho_{a+r}^\tau)$, the split gives
\[
\begin{aligned}
  \nabla^n\rho_{a+u}^\tau
  ={}&\nabla^ne^{\tau u\Delta}\rho_a^\tau
    +\int_0^{u/2}\nabla^n\nabla\!\cdot
       e^{\tau(u-r)\Delta}F_{a+r}\,dr\\
  &+\int_{u/2}^{u}e^{\tau(u-r)\Delta}
          \nabla\!\cdot\nabla^nF_{a+r}\,dr.
\end{aligned}
\]
By Lemma~\ref{lem:appD-heat-smoothing} and
\eqref{eq:appD-density-on-maximal-interval}, the initial term,
after multiplication by $(\tau u)^{n/2}$, is at most $CC_*$.
The flux itself is bounded by $C$, so the early integral obeys
\[
  (\tau u)^{n/2}
  \int_0^{u/2}[\tau(u-r)]^{-(n+1)/2}
       \|F_{a+r}\|_{L^\infty}\,dr
  \le C\sqrt{u/\tau}.
\]
On the late interval, use~\eqref{eq:appD-highest-derivative-linear}
and $(\tau r)^{-n/2}\le2^{n/2}(\tau u)^{-n/2}$.  Its contribution
is at most
\[
  C R_{n,\tau_{\mathrm L}}(\tau u)^{n/2}\sqrt{u/\tau}
  +C2^{n/2}\sqrt{u/\tau}\,Z_n(u).
\]
The three estimates, applied at each time up to $u$, imply
\[
  Z_n(u)
  \le C\bigl(1+\sqrt{u/\tau_{\mathrm L}}\bigr)
     +C R_{n,\tau_{\mathrm L}}(\tau u)^{n/2}\sqrt{u/\tau}
     +C2^{n/2}\sqrt{u/\tau_{\mathrm L}}\,Z_n(u).
\]
Choose $h\in(0,h_0/2]$, independently of $t$ and
$\tau\ge\tau_{\mathrm L}$, so that
\begin{equation}
  C2^{n/2}\sqrt{h/\tau_{\mathrm L}}\le\tfrac12,
  \qquad n=1,\ldots,4.
  \label{eq:appD-window-choice}
\end{equation}
The last term can then be absorbed.  At $u=h$, divide by
$(\tau h)^{n/2}$ to obtain
\[
  \|\nabla^n\rho_t^\tau\|_{L^\infty}
  \le C(\tau_{\mathrm L}h)^{-n/2}
       \bigl(1+\sqrt{h/\tau_{\mathrm L}}\bigr)
      +C R_{n,\tau_{\mathrm L}}\sqrt{h/\tau_{\mathrm L}}.
\]
This controls $h\le t<T_*^\tau$.  The local estimate
\eqref{eq:appD-local-C4} controls $0\le t\le h$, completing
the induction through $n=4$.

\medskip
\noindent\emph{Step 3: continuation and uniqueness.}
Step 2 gives $\|\rho_{t_0}^\tau\|_{C^4}\le
C_{\tau_{\mathrm L}}$ for every $t_0<T_*^\tau$.
The contraction in Step 1 can therefore be restarted at any such
$t_0$ on an interval of length $T>0$ that depends on
$\tau_{\mathrm L}$ and this bound but not on $t_0$ or $\tau$:
\[
  C\sqrt{T/\tau_{\mathrm L}}
       \bigl(4C_{\tau_{\mathrm L}}+\|\pi\|_{C^4}\bigr)
  \le\tfrac12.
\]
If $T_*^\tau$ were finite, restarting sufficiently close to
$T_*^\tau$ would extend the original solution past that time,
contrary to maximality.  Thus $T_*^\tau=\infty$, and the bounds
from Steps 1 and 2 give~\eqref{eq:positive-temperature-W4}.

Finally, let $\rho_1$ and $\rho_2$ be bounded mild solutions
with the same initial density.  The mild formula and the integrable
heat-kernel bound~\eqref{eq:appD-heat-smoothing} give both solutions
the same $L^\infty$ trace at time zero and continuity in
$L^\infty$ on finite intervals.  Fix $T_0>0$ and write
$B:=\max_{j=1,2}\sup_{0\le t\le T_0}
\|\rho_{j,t}\|_{L^\infty}<\infty$.  The $L^\infty$ flux
difference estimate from Step 1 gives, on any interval
$[s,s+\delta]\subset[0,T_0]$ where they agree at $s$,
\[
  \|\rho_1-\rho_2\|_{C([s,s+\delta];L^\infty)}
  \le C\sqrt{\delta/\tau}
  \bigl(2B+\|\pi\|_{L^\infty}\bigr)
  \|\rho_1-\rho_2\|_{C([s,s+\delta];L^\infty)}.
\]
Choose $\delta$ to make the coefficient smaller than one and
reuse it at successive times up to $T_0$.  Since $T_0$ was
arbitrary, this proves uniqueness
among bounded mild solutions and completes the lemma.
\end{proof}

\subsection{A weak--strong interpolation estimate}
\label{subsec:appD-interpolation}

The global estimate \eqref{eq:positive-temperature-W4} bounds four
derivatives of the mean-field density, while
Proposition~\ref{prop:hminus-exp} gives decay in $\dot H^{-1}$.
To use these estimates in the velocity gradient, we need decay of two
derivatives in a H\"older norm.  The following lemma converts the weak
decay and the strong bound into that intermediate estimate.

\begin{lemma}[Periodic weak--strong interpolation]
\label{lem:appD-weak-strong-interpolation}
Let $0<\beta<1$ and
$\vartheta:=(2-\beta)/(5+d/2)$.  Every mean-zero function
$f\in\dot H^{-1}(\T^d)\cap W^{4,\infty}(\T^d)$ satisfies
\begin{equation}
  \|f\|_{C^{2,\beta}}
  \le C_{d,\beta}
     \|f\|_{\dot H^{-1}}^{\vartheta}
     \|f\|_{W^{4,\infty}}^{1-\vartheta}.
  \label{eq:appD-weak-strong-interpolation}
\end{equation}
\end{lemma}

\begin{proof}
We estimate the low and high Fourier frequencies separately.  Their
dyadic bounds have exponents $-1-d/2$ and $4$, respectively.
Interpolating those bounds produces $2+\beta$, the exponent
corresponding to $C^{2,\beta}$.

Let $(\Delta_j)_{j\ge-1}$ be a smooth periodic
Littlewood--Paley decomposition as in
\cite[Section~2, Eq.~(4)]{CuvillierFanelliSalguero2023}.
For $j\ge0$, the block $\Delta_j$ is supported at frequencies
$|k|\asymp2^j$.  The block $\Delta_{-1}$ contains finitely
many lattice frequencies.  Since $f$ has zero mean, the latter block
contains no zero mode, and hence
\[
  \|\Delta_{-1}f\|_{L^\infty}
  \le C\|f\|_{\dot H^{-1}}.
\]
On every block $j\ge0$, the periodic Bernstein inequality
\cite[Lemma~2.1]{CuvillierFanelliSalguero2023} and frequency
localization give
\begin{equation}
\begin{aligned}
  \|\Delta_jf\|_{L^\infty}
  &\le C2^{jd/2}\|\Delta_jf\|_{L^2}\\
  &\le C2^{j(1+d/2)}\|f\|_{\dot H^{-1}}.
\end{aligned}
\label{eq:appD-weak-frequency-bound}
\end{equation}
This bounds the low-regularity dyadic norm by
$\|f\|_{\dot H^{-1}}$.

For the high-regularity norm, we move four derivatives onto $f$.
On the frequency support of $\Delta_j$, the multiplier
$|k|^{-4}$ has size $O(2^{-4j})$.  Its smooth localized
kernel has a uniformly bounded $L^1$ norm after this factor
is removed.  Therefore
\begin{equation}
  \|\Delta_j f\|_{L^\infty}
  \le C2^{-4j}\|(-\Delta)^2f\|_{L^\infty}
  \le C2^{-4j}\|f\|_{W^{4,\infty}},
  \qquad j\ge0.
  \label{eq:appD-strong-frequency-bound}
\end{equation}
The finite low block satisfies the same estimate after enlarging
$C$.  The periodization or rescaling to the unit torus changes
only constants.

Write
$\|f\|_{B^s_{\infty,\infty}}
 :=\sup_{j\ge-1}2^{js}\|\Delta_jf\|_{L^\infty}$
for the periodic dyadic norm.  The preceding two estimates show
\[
  \|f\|_{B^{-1-d/2}_{\infty,\infty}}
    \le C\|f\|_{\dot H^{-1}},
  \qquad
  \|f\|_{B^4_{\infty,\infty}}
    \le C\|f\|_{W^{4,\infty}}.
\]
The choice of $\vartheta$ in the statement satisfies
\[
  2+\beta
  =\vartheta(-1-d/2)+(1-\vartheta)4.
\]
Raising the two dyadic bounds to these respective powers and
then taking the supremum gives
\begin{equation}
  \|f\|_{B^{2+\beta}_{\infty,\infty}}
  \le C\|f\|_{\dot H^{-1}}^{\vartheta}
         \|f\|_{W^{4,\infty}}^{1-\vartheta}.
  \label{eq:appD-interpolated-dyadic-bound}
\end{equation}
For the noninteger exponent $2+\beta$, the periodic
identification
$B^{2+\beta}_{\infty,\infty}(\T^d)=C^{2,\beta}(\T^d)$
\cite[Section~3.5.4, Eq.~(15)]{SchmeisserTriebel1987}
turns this into \eqref{eq:appD-weak-strong-interpolation}.
\end{proof}

\section{Distance estimates from the modulated energy}
\label{app:distance-estimates}

This section proves the spherical-smearing estimate used in
Lemma~\ref{lem:distance-coercivity}.  Smearing gives the empirical
measure finite Coulomb energy.  We compare that energy with the
diagonal-removed energy \(F_N\), and then use the comparison to bound
the Wasserstein and negative Sobolev distances.  The result is stated
separately so that the energy calculation can be used without repeating
the distance argument.

\begin{lemma}[Energy after spherical smearing]
\label{lem:app-distance-smearing}
Let \(N\ge2\), \(X_N=(x_1,\ldots,x_N)\in\mathcal X_N\), and let
\(\rho\) be a probability density with
\(\|\rho\|_{L^\infty}\le C_*\).  Write \(\delta_x^{(r)}\) for the
uniform probability measure on the sphere of radius \(r\) centered at
\(x\) in \(\T^d\).  Set
\[
  \eta:=\frac14\bigl(N\|\rho\|_{L^\infty}\bigr)^{-1/d},
  \qquad
  \mu_{X_N,\eta}:=\frac1N\sum_{i=1}^N\delta_{x_i}^{(\eta)}.
\]
With \(F_N\) and \(\err\) defined in
Section~\ref{sec:modulated-identity}, one has
\(\mathcal W_2(\mu_{X_N},\mu_{X_N,\eta})^2\le\eta^2\).
There is also a constant depending only on \(d\) and \(C_*\) such
that
\begin{equation}
  \begin{aligned}
    \|\mu_{X_N,\eta}-\rho\|_{\dot H^{-1}}^2
    &\le F_N(X_N,\rho)+C\err\\
    &\le C\left(\frac12F_N(X_N,\rho)+\err\right).
  \end{aligned}
  \label{eq:app-distance-smearing}
\end{equation}
\end{lemma}

\begin{proof}
We first couple each particle to its sphere.  We then compare the
interactions between distinct spheres with those between the original
particles.  The remaining self-interaction and density corrections
are estimated at the chosen radius \(\eta\).

\medskip
\noindent\emph{Step 1: the transport cost of smearing.}
Let \(d_{\T^d}\) be the geodesic distance on the unit torus.  The
probability measure
\[
  \gamma:=\frac1N\sum_{i=1}^N
              \delta_{x_i}\otimes\delta_{x_i}^{(\eta)}
\]
has marginals \(\mu_{X_N}\) and \(\mu_{X_N,\eta}\).  Since
\(\eta<1/4\), each transported pair is at torus distance \(\eta\).
The definition of \(\mathcal W_2\) therefore gives
\begin{equation}
  \mathcal W_2(\mu_{X_N},\mu_{X_N,\eta})^2
  \le\int_{\T^d\times\T^d}d_{\T^d}(x,y)^2\,d\gamma(x,y)
  =\eta^2.
  \label{eq:app-distance-transport}
\end{equation}
Figure~\ref{fig:app-distance-smearing} shows the transport in this
coupling and the two interactions in the energy calculation.

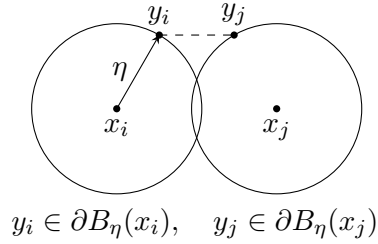
\begin{figure}[ht]
\centering
\begin{tikzpicture}[scale=0.9,>=stealth]
  \coordinate (a) at (0,0);
  \coordinate (b) at (2.35,0);
  \draw (a) circle (1.25);
  \draw (b) circle (1.25);
  \fill (a) circle (1.5pt);
  \fill (b) circle (1.5pt);
  \coordinate (p) at (0.625,1.083);
  \coordinate (q) at (1.725,1.083);
  \fill (p) circle (1.5pt);
  \fill (q) circle (1.5pt);
  \draw[->] (a) -- (p) node[midway,left] {$\eta$};
  \draw[dashed] (p) -- (q);
  \node[below] at (a) {$x_i$};
  \node[below] at (b) {$x_j$};
  \node[above] at (p) {$y_i$};
  \node[above] at (q) {$y_j$};
  \node at (1.15,-1.68) {$y_i\in\partial B_\eta(x_i),\quad
                       y_j\in\partial B_\eta(x_j)$};
\end{tikzpicture}
\caption{A planar cross-section of the spherical smearing.  The
point \(x_i\) is coupled to a point \(y_i\) on its radius-\(\eta\)
sphere, while the interaction of distinct particles is averaged over
\(y_i\) and \(y_j\).  The spheres may overlap.  The proof uses
spherical averages to compare their interaction and treats the
interaction of a sphere with itself separately.}
\label{fig:app-distance-smearing}
\end{figure}

\medskip
\noindent\emph{Step 2: interactions between distinct spheres.}
For a function \(h\), denote its spherical average about \(z\) by
\[
  A_rh(z):=\int_{\T^d}h(z-y)\,d\delta_0^{(r)}(y).
\]
The relation \(\Delta K_1=1-\delta_0\) and the divergence theorem
give, for \(z\ne0\) and almost every \(0<r<\eta\),
\[
  \frac{d}{dr}A_rK_1(z)
  =\frac{|B(0,r)|-\mathbf1_{\{z\in B(0,r)\}}}
          {|\partial B(0,r)|}
  \le\frac rd.
\]
Here the balls are read in the local Euclidean chart, which is
unambiguous because \(r<1/4\).  Since \(A_rK_1(z)\to K_1(z)\) as
\(r\downarrow0\), integration yields
\begin{equation}
  A_\eta K_1(z)\le K_1(z)+\frac{\eta^2}{2d},
  \qquad z\ne0.
  \label{eq:app-distance-spherical-average}
\end{equation}
Apply this inequality in each variable of the interaction between
the \(i\)-th and \(j\)-th spheres.  When one sphere passes through
the other center, the exceptional coincidence point still has zero
spherical measure.  Thus, for \(i\ne j\),
\begin{equation}
  \iint K_1(x-y)\,
    d\delta_{x_i}^{(\eta)}(x)d\delta_{x_j}^{(\eta)}(y)
  \le K_1(x_i-x_j)+\frac{\eta^2}{d}.
  \label{eq:app-distance-two-spheres}
\end{equation}
Thus the off-diagonal interaction increases by at most
\((N-1)\eta^2/(dN)\) after division by \(N^2\).

\medskip
\noindent\emph{Step 3: interaction with the density.}
Put \(V:=K_1*\rho\), so that \(\Delta V=1-\rho\).  Applying the
same spherical-average identity to \(V\) gives
\[
  \left|A_\eta V(x_i)-V(x_i)\right|
  \le\int_0^\eta\frac{|B(0,r)|}{|\partial B(0,r)|}
                   \|1-\rho\|_{L^\infty}\,dr
  =\frac{\eta^2}{2d}\|1-\rho\|_{L^\infty}.
\]
One may first derive this identity for smooth \(\rho\) and then
approximate a bounded density.  Indeed, \(K_1\in L^1(\T^d)\)
implies that \(V\) is continuous, and its mollifications converge
uniformly.  The resulting bound controls the change in the
particle-density interaction by
\begin{equation}
  \left|\frac2N\sum_{i=1}^N
       \bigl(A_\eta V(x_i)-V(x_i)\bigr)\right|
  \le\frac{\eta^2}{d}\|1-\rho\|_{L^\infty}.
  \label{eq:app-distance-density-correction}
\end{equation}

The interaction of a sphere with itself has not been included in
\(F_N\).  By translation invariance, all these diagonal terms have
the same value.  Denote their combined contribution by
\[
  D_\eta:=\frac1N\iint K_1(x-y)\,
                   d\delta_0^{(\eta)}(x)d\delta_0^{(\eta)}(y).
\]
The Coulomb singularity is integrable against the product of two
spherical surface measures.  In local coordinates on a sphere, the
corresponding radial integral near the diagonal is bounded by
\[
  \int_0^\eta |\log r|\,dr<\infty\quad(d=2),
  \qquad
  \int_0^\eta r^{2-d}r^{d-2}\,dr<\infty\quad(d\ge3).
\]
We may therefore expand the finite \(\dot H^{-1}\) energy of
\(\mu_{X_N,\eta}-\rho\).  Applying
\eqref{eq:app-distance-two-spheres} and
\eqref{eq:app-distance-density-correction} now gives
\begin{equation}
  \|\mu_{X_N,\eta}-\rho\|_{\dot H^{-1}}^2
  \le F_N(X_N,\rho)+D_\eta
       +\frac{N-1}{dN}\eta^2
       +\frac{\eta^2}{d}\|1-\rho\|_{L^\infty}.
  \label{eq:app-distance-energy-expansion}
\end{equation}
The next step bounds the one term in this expansion that comes from
particle self-interaction.

\medskip
\noindent\emph{Step 4: self-interaction and the choice of radius.}
The Euclidean Coulomb kernel and its normalizing constant are
\[
  K_E(r):=
  \begin{cases}
    \log(1/r),&d=2,\\
    r^{2-d},&d\ge3,
  \end{cases}
  \qquad
  c_d^{\mathrm R}:=
  \begin{cases}
    2\pi,&d=2,\\
    \displaystyle\frac{4\pi^{d/2}}{\Gamma((d-2)/2)},&d\ge3.
  \end{cases}
\]
Both \(K_1(z)\) and \(K_E(|z|)/c_d^{\mathrm R}\) have the same
point-mass singularity.  Their difference is smooth near the origin
\cite[Eq.~(3.1)]{ChodronRosenzweigSerfaty2025}.  That result covers
\(|z|<1/4\).  On the compact annulus
\(1/4\le|z|\le1/2\), both kernels are smooth and their difference
is bounded.  Since \(2\eta<1/2\), this comparison applies
to every pair of points on one radius-\(\eta\) sphere.

The Euclidean potential of a uniformly charged sphere is constant
inside the sphere and equals \(K_E(\eta)\) on its boundary.  Hence
\begin{equation}
  D_\eta
  \le\frac{K_E(\eta)}{c_d^{\mathrm R}N}+\frac CN
  \le\frac CN
  \begin{cases}
    1+\log(1/\eta),&d=2,\\
    1+\eta^{2-d},&d\ge3.
  \end{cases}
  \label{eq:app-distance-self-energy}
\end{equation}
In the second inequality, the case distinction means the
corresponding expression in each dimension.

Finally, put \(M:=\|\rho\|_{L^\infty}\).  As \(\rho\) is a
probability density, \(1\le M\le C_*\) and
\(\|1-\rho\|_{L^\infty}\le1+M\le2M\).  The definition of
\(\eta\) then gives
\[
  \eta^2\bigl(1+\|1-\rho\|_{L^\infty}\bigr)
  \le C\|\rho\|_{L^\infty}^{(d-2)/d}N^{-2/d}.
\]
The definition \eqref{eq:finite-particle-error} bounds
\(N^{-2/d}\) by \(C\err\), and also bounds
\((\log N)/N\) by \(C\err\) when \(d=2\).  Hence the
right-hand side and \(D_\eta\) are at most \(C\err\).
Substitution in
\eqref{eq:app-distance-energy-expansion} proves the first energy
bound in \eqref{eq:app-distance-smearing}.  For its shifted form, set
\(S_N:=\tfrac12F_N(X_N,\rho)+\err\).  The lower bound
\eqref{eq:modulated-energy-lower-bound} gives
\(S_N\ge\err/2\).  Therefore
\[
  F_N(X_N,\rho)+C\err
  =2S_N+(C-2)\err
  \le C' S_N,
\]
after increasing the constant if needed.  This completes the proof.
\end{proof}

\begin{proof}[Proof of Lemma~\ref{lem:distance-coercivity}]
The estimate just proved supplies the energy comparison
\eqref{eq:smeared-energy-coercivity} used in the main text.  For
completeness, we show how it gives the two distances.  Since
\(\rho\ge c_*>0\), \cite[Corollary~2.3, Eq.~(2.12)]{Peyre2018} gives
\[
  \mathcal W_2(\mu_{X_N,\eta},\rho)
  \le\frac{2}{\sqrt{c_*}}
       \|\mu_{X_N,\eta}-\rho\|_{\dot H^{-1}}.
\]
The triangle inequality and \eqref{eq:app-distance-smearing} imply
\[
  \mathcal W_2(\mu_{X_N},\rho)^2
  \le 2\eta^2+\frac{8}{c_*}
       \|\mu_{X_N,\eta}-\rho\|_{\dot H^{-1}}^2
  \le C\left(\frac12F_N(X_N,\rho)+\err\right),
\]
where \(\eta^2\le C\err\le2C S_N\) follows from
\eqref{eq:finite-particle-error} and
\eqref{eq:modulated-energy-lower-bound}.  This is
\eqref{eq:wasserstein-coercivity}.  For \(\zeta>d/2+1\), the
Sobolev-to-Wasserstein comparison
\eqref{eq:negative-sobolev-wasserstein-comparison} gives
\[
  \|\mu_{X_N}-\rho\|_{\dot H^{-\zeta}}
  \le C_\zeta\mathcal W_1(\mu_{X_N},\rho)
  \le C_\zeta\mathcal W_2(\mu_{X_N},\rho).
\]
Squaring this inequality proves
\eqref{eq:negative-sobolev-coercivity}.
\end{proof}

\end{document}